\documentclass[a4paper]{amsproc}
\usepackage{amsmath}
\usepackage{amssymb}
\usepackage{latexsym}
\usepackage{amsmath}
\usepackage{amssymb}
\usepackage{amsthm}
\usepackage{latexsym}
\usepackage{graphicx}
\usepackage{color}
\usepackage{hyperref}
\usepackage{ushort}
\newcommand{\arxiv}[1]{\href{http://www.arXiv.org/abs/#1}{arXiv:#1}}

\renewcommand{\geq}{\geqslant}
\renewcommand{\leq}{\leqslant}
\renewcommand{\preceq}{\preccurlyeq}
\renewcommand{\succeq}{\succcurlyeq}

\makeatletter
\newcommand{\pushright}[1]{\ifmeasuring@#1\else\omit\hfill$\displaystyle#1$\fi\ignorespaces}
\newcommand{\pushleft}[1]{\ifmeasuring@#1\else\omit$\displaystyle#1$\hfill\fi\ignorespaces}
\newcommand{\specialcell}[1]{\ifmeasuring@#1\else\omit$\displaystyle#1$\ignorespaces\fi}
\makeatother

\newenvironment{adding}[1]{\marginpar{{\bf [}\footnotesize #1}\bf }{\marginpar{\bf ]}}
\newenvironment{addingwithoutmargin}{\bf }{}
\newcommand{\signh}{^{\text{sgn}}}
\newcommand{\signed}[1]{#1\signh}
\newcommand{\badd}[1]{\begin{adding}{#1}}
\newcommand{\eadd}{\end{adding}}
\newcommand{\baddw}{\begin{addingwithoutmargin}}
\newcommand{\eaddw}{\end{addingwithoutmargin}}

\catcode`\@=11\def\th@mydefinition{
  \thm@notefont{\bfseries}
}
\catcode`\@=12
\newtheorem{theorem}{Theorem}[section]
\newtheorem{prop}[theorem]{Proposition}
\newtheorem{lemma}[theorem]{Lemma}
\newtheorem{pty}[theorem]{Property}
\newtheorem{propdefinition}[theorem]{Proposition-Definition}
\newtheorem{corollary}[theorem]{Corollary}

\theoremstyle{mydefinition}
\newtheorem{example}[theorem]{Example}
\theoremstyle{definition}
\newtheorem{ex}[theorem]{Example}
\newtheorem{definition}[theorem]{Definition}
\newtheorem{remark}[theorem]{Remark}
\newtheorem{fact}[theorem]{Fact}

\newcommand{\val}{\operatorname{val}}
\newcommand{\cP}{\mathcal{P}}
\newcommand{\cD}{\mathcal{D}}
\newcommand{\Id}{\mathrm{I}}
\newcommand{\we}{\mathrm{w}}
\newcommand{\cof}{\operatorname{cof}}

\newcommand{\set}[2]{\{#1\mid\,#2\}}
\newcommand{\sgn}{\operatorname{sgn}}
\newcommand{\NEW}[1]{{\em #1\/}\index{#1}}
\newcommand{\adj}{^{\mathrm{adj}}}

\newcommand{\balance}{\,\nabla\, }
\newcommand{\bal}{\balance}

\newcommand{\detp}[1]{{\rm det\,}(#1)}

\def\allperm{\mathfrak{S}}

\newcommand{\bala}{\;|\!\nabla\!|\;}

\newcommand{\thin}{thin }
\newcommand{\thinp}{thin}
\newcommand{\constr}{the construction of monotone algorithms}
\newcommand{\conv}{the convergence of monotone algorithms}

\makeatletter
\@addtoreset{equation}{section}

\makeatother
\newcommand{\rmax}{\mathbb{R}_{\max}}

\newcommand{\smax}{\mathbb{S}_{\max}}

\def\SS{{\mathcal S}}

\def\TT{{\mathcal T}}
\newcommand{\GG}{{\mathcal G}}
\ushortCreate(0.65){uline}
\newcommand{\skewproduct}[2]{{#1}{\mathrel{\,\uline{\!\rtimes}}}{#2}}
\newcommand{\skewproductbar}[2]{{#1}{\mathrel{\bar{\rtimes}}}{#2}}
\newcommand{\skewproductstar}[2]{#1{\rtimes}{#2}}

\def\SSp{{\mathcal A}}
\def\MM{{\mathcal M}}
\def\e{\begin{equation}}
\def\ee{\end{equation}}
\def\RR{{\mathcal R}}
\def\M{{\mathcal M}_{n}}
\newcommand{\Z}{{\mathbb Z}}
\newcommand{\B}{{\mathbb B}}

\def\Ti{{\mathbb T}_{\mathrm 2}}
\def\To{{\mathbb T}}
\def\Se{\SS_{\mathrm e}}
\def\Re{{\mathbb R}}
\newcommand{\phase}{\mathsf{Ph}}
\newcommand{\K}{\mathbb{K}}
\def\S{\mathbb{S}_{\max}}

\def\BB{{\mathbb B}_{{\mathrm s}}}

\def\C{{\mathbb C}}
\newcommand{\N}{\mathbb{N}}
\newcommand{\Q}{\mathbb{Q}}

\newcommand{\per}{\operatorname{per}}
\newcommand{\Cext}{\skewproduct{\C}{\rmax}}

\DeclareMathAlphabet{\mathbbold}{U}{bbold}{m}{n}
\newcommand{\zero}{\mathbbold{0}}
\newcommand{\unit}{\mathbbold{1}}
\newcommand{\szero}{\zero}

\newcommand{\oo}{0}
\newcommand{\oi}{1}
\newcommand{\ooo}{\zero}
\newcommand{\ooi}{\unit}

\title{Tropical Cramer Determinants Revisited}
\author{Marianne Akian}
\author{St\'ephane Gaubert}
\author{Alexander Guterman}
\dedicatory{We dedicate this paper to the memory of our friend and colleague Grigory L. Litvinov.}
\date{\today}
\subjclass[2010]{14T05}
\keywords{Tropical algebra, max-plus algebra, tropical hyperplanes, optimal assignment, Cramer systems}
\thanks{The first two authors have been partially supported by the joint RFBR/CNRS grant 11-01-93106, and by the PGMO program of EDF and Fondation Math\'ematique Jacques Hadamard. The third author was partially supported by the invited professors program from INRIA Saclay and \'Ecole Polytechnique and by the grants
MD-2502.2012.1 and RFBR 12-01-00140a.}
\begin{document}
\begin{abstract}
We prove general Cramer type theorems for linear systems over various
extensions of the tropical semiring, in which tropical numbers are
enriched with an information of multiplicity, sign, or argument. We
obtain existence or uniqueness results, which extend or refine earlier
results of Gondran and Minoux (1978), Plus (1990), Gaubert (1992),
Richter-Gebert, Sturmfels and Theobald (2005) and Izhakian and Rowen
(2009). Computational issues are also discussed; in particular, some
of our proofs lead to Jacobi and Gauss-Seidel type algorithms to solve
linear systems in suitably extended tropical semirings.
\end{abstract}
\maketitle
\section{Introduction}
\subsection{Motivations}
\newcommand{\paramh}{a}
The max-plus or tropical semiring $\rmax$ is the set $\Re\cup\{-\infty\}$, equipped with the addition $a\oplus b = \max(a,b)$ and the multiplication $a\odot b= a+b$. We refer the reader for instance to~\cite{BCOQ92,maslovkolokoltsov95,litvinov00,itenberg,RGST} for introductory materials on max-plus or tropical algebra. 

We denote by $\rmax^n$ the $n$th-fold Cartesian product of $\rmax$, which can be thought of as the tropical analogue of a finite dimensional
vector space. A {\em tropical hyperplane} 
is a subset of $\rmax^n$ of the form
\begin{align}
H= \{x\in \rmax^n\mid \max_{i\in[n]} (\paramh_i+x_i) \text{ is attained at least twice}\} \enspace ,\label{e-def-h}
\end{align}
where $\paramh=(\paramh_1,\dots,\paramh_n)$ is a vector of $\rmax^n$, not identically $-\infty$, and $[n]:=\{1,\dots,n\}$. This 
definition is motivated by non-archimedean geometry. 
Indeed, let $\K=\C\{\{t\}\}$ denote the field
of complex Puiseux series in the variable $t$, and let $v$ denote
the valuation which associates
to a series the opposite of its smallest exponent. 
Consider now a hyperplane of $\K^n$,
\[
\mathbf{H} := \{ \mathbf{x} \in \K^n \mid \sum_{i\in[n]} \mathbf{\paramh}_i\mathbf{x}_i =0 \} \enspace ,
\] 
where $\mathbf{\paramh}=(\mathbf{\paramh}_1,\dots,\mathbf{\paramh}_n)$ is a vector of $\K^n$ lifting $\paramh$, meaning that  $v(\mathbf{\paramh}_i)= \paramh_i$, for all $i\in [n]$.
Then it is easily checked
that the image of $\mathbf{H}$ by the map which applies the valuation
$v$ entrywise is precisely the set of vectors of $H$ with rational
coordinates. This is actually a special case of a result of Kapranov characterizing the non-Archimedean amoeba of a hypersurface~\cite{kapranov}. 

In the present work, we will extend or refine a series of basic results
concerning the intersections of tropical hyperplanes, that we now review.
One of these results was established
by Richter-Gebert, Sturmfels and Theobald.
\begin{theorem}[{Tropical Cramer Theorem, {``complex''} version~{\cite{RGST}}}]\label{th-1}
Any $n-1$ vectors of $\Re^n$ in general position
are contained in a unique tropical hyperplane. 
\end{theorem}
This statement has also an equivalent
dual form: {\em the intersection of $n-1$ tropical hyperplanes in general
position contains a unique vector up to an additive constant}.

The parameters $\paramh_i$ of the hyperplane $H$ arising in Theorem~\ref{th-1}
can be obtained by solving the tropical analogue of a square linear
system. Its solution turns out to be determined by the tropical analogues
of Cramer determinants, which in this context are merely the value
of optimal assignment problems. Then
the data are said to be in {\em general position}
if each of the optimal assignment problems arising
in this way has a unique optimal solution. The role
of the general position notion is made clear by the following result.
\begin{theorem}[\cite{RGST}]\label{th-2}
A collection of $n$ vectors of $\Re^n$ is contained in a tropical hyperplane
if and only if the matrix having these vectors as columns is
{\em tropically singular}, meaning that the
assignment problem associated with this matrix has at least two optimal
solutions. 
\end{theorem}
More information on tropical singularity and related rank notions can be found in~\cite{DSS,AGG08,IzhRowen,AGGut10}.

Different but related results were obtained previously
by considering tropical numbers with signs.
Indeed, in the above results, the tropical semiring is essentially thought of as the image of the field of complex Puiseux series by the valuation. Alternatively, a tropical number with sign may be thought of as
the image of a {\em real} Puiseux series, i.e., of an element
of $\Re\{\{t\}\}$. 
Ideas of this nature were indeed essential
in the development by Viro of the patchworking method
(see the references in~\cite{viro}), as well 
as in the construction by Plus~\cite{Plus} 
of the symmetrized tropical (or max-plus)
semiring, $\smax$.

When considering tropical numbers with signs,
the notion of equation has to be replaced
by a notion of {\em balance}~\cite{Plus}: a tropical sum of terms
equipped with signs is said to be balanced if
the maximum of terms with positive signs 
coincides with the maximum of the terms 
with negative signs. 
Some results of~\cite{Plus,gaubert92a}
concerning systems of linear balances can be interpreted geometrically
using the signed variant of tropical hyperplanes,
considered by Joswig in~\cite{joswig04}. 
A {\em signed tropical hyperplane} 
is a subset of $\rmax^n$ of the form
\begin{align}\label{e-def-hsign}
H\signh= \{x\in \rmax^n\mid \max_{i \in I} (\paramh_i+x_i)  = 
\max_{j\in J} (\paramh_j +x_j)\}
\enspace ,
\end{align}
where $[n]= I\cup J$ is a non-trivial partition, and $a\in \rmax^n$ is a vector
non-identically $-\infty$.
Observe that $H\signh\subset H$.
Consider now a hyperplane of $(\Re\{\{t\}\})^n$,
\[
\mathbf{H}\signh := \{ \mathbf{x} \in(\Re\{\{t\}\})^n \mid \sum_{i\in I} \mathbf{\paramh}_i\mathbf{x}_i = \sum_{j\in J} \mathbf{\paramh}_{j}\mathbf{x}_j \} \enspace ,
\] 
where $\mathbf{\paramh}$ is any vector of $(\Re\{\{t\}\})^n$ lifting $\paramh$, meaning now that  $v(\mathbf{\paramh}_i)= \paramh_i$ and that $\mathbf{a}_i$ is nonnegative (recall that a real Puiseux series is nonnegative
if it is zero or if its leading coefficient is positive), 
for all $i\in [n]$. Then it can be easily checked that the 
vectors with rational entries of a signed tropical hyperplane are precisely the images by the valuation of the nonnegative vectors of the associated hyperplane over the field of {\em real} Puiseux series. 

With these observations in mind,  the following result established
by Plus appears to be a ``real'' analogue of Theorem~\ref{th-1}.
\begin{theorem}[Tropical Cramer theorem, ``real'' version, Corollary of~{\cite[Th.~6.1]{Plus}}]\label{th-3}
Any $n-1$ vectors of $\rmax^n$ in sign-general
position are contained in a unique signed tropical hyperplane. 
\end{theorem}
By comparison with Theorem~\ref{th-1}, we use here a milder
notion of general position.  Indeed,
the ``real'' tropical analogue of a determinant
consists of the value of an optimal assignment problem, together with the 
information of all the possible signs of optimal permutations.
Then, a tropical determinant
is said to be {\em sign-nonsingular}
if all optimal permutations have the same sign.
Finally, $n-1$ vectors of $\rmax^n$ are said
to be in {\em sign-general position} if all the 
associated tropical Cramer determinants are sign-nonsingular.
The vector $\paramh$ defining the signed hyperplane $H\signh$ of Theorem~\ref{th-3} is determined by the tropical Cramer determinants, the signs of which 
provide the sets $I,J$ in the partition.

Actually, a more general result
was stated in~\cite{Plus} in the language
of systems of balances over the symmetrized tropical semiring $\smax$.
Theorem~\ref{th-3} covers a special case
with a more straightforward geometric interpretation.
The details of the derivation of the latter theorem from Theorem~6.1 of~\cite{Plus} will be given in Section~\ref{sec-geom}, together with the dual result of Theorem~\ref{th-3}, concerning intersections of signed tropical hyperplanes.

A result of Gondran and Minoux, that
we restate as follows in terms of signed tropical hyperplanes, 
may be thought of as a ``real''
analogue of Theorem~\ref{th-2}.

\begin{theorem}[Corollary of \cite{gondran78}]\label{th-4}
A collection of $n$ vectors of $\rmax^n$ is contained in a signed tropical
hyperplane if and only if the matrix having these vectors as columns 
has a sign-singular tropical determinant.
\end{theorem}
A dual result, concerning the intersection of $n$ signed tropical hyperplanes,
was stated by Plus~\cite{Plus} and proved by Gaubert~\cite{gaubert92a}, we 
shall discuss it in Section~\ref{sec-geom}. 
As pointed out in~\cite{ButGa}, the notion
of sign-nonsingularity arising here is an extension
of the notion
with the same name arising in combinatorial matrix theory,
in particular in the study of the permanent problem of P\'olya,
see~\cite{shader} for more information.

Whereas the results concerning ``balances'' involve the extension of the tropical numbers with signs, 
we note that other extensions have been used more recently.
In particular, the incorporation of ``phase'' (instead of sign) information
in tropical constructions has played an important role in the
arguments of Mikhalkin ~\cite{mikhalkin}. Moreover, Viro introduced
a general notion of {\em hyperfield}~\cite{1006.3034} with a multivalued
addition, which he used in particular to capture the phase information. 
The semirings with symmetry introduced by the 
authors in~\cite{AGG08} provide another way to encode the sign or phase. 
Also, different extensions have been provided by the
``supertropical'' structures of Izhakian and Rowen~\cite{Izhsuper2010}
extending the bi-valued tropical semiring introduced
by Izhakian~\cite{Izhalone}; in the latter, the goal is not to encode ``sign'' or ``phase'', but the
fact that the maximum in an expression is achieved twice at least. 

\subsection{Main results}
Given the analogy between the ``complex'' and ``real'' versions of the tropical Cramer theorem (Theorems~\ref{th-1} and \ref{th-3} above), as well as between
the unsigned and signed notions of tropical singularity of matrices (Theorems~\ref{th-2} and \ref{th-4}), one may ask whether
all these results may be derived from common principles. One may also
ask whether results of this kind are valid
for more general tropical structures, encoding different kinds of sign
or phase information, in the light of recent constructions of extended
tropical semirings by Viro~\cite{1006.3034}, Izhakian and Rowen~\cite{Izhsuper2010}, and the authors~\cite{AGG08}. 

In this paper, we answer these questions and deal
with related algorithmic issues, by developing
a theory of elimination of linear systems over semirings, building
on ideas and results of~\cite{Plus,gaubert92a,AGG08}. This will
allow us to show that the earlier results are indeed special instances of general  Cramer type theorems, 
which apply to various extensions of tropical semirings. 
These theorems are established using axioms allowing one
to perform ``elimination'' of balances, in a way
similar to Gaussian elimination. In this way, we will generalize
and sometimes refine (handling degenerate cases)
earlier results. Also, 
some of our proofs are based on Jacobi or Gauss-Seidel
type iterative schemes, and lead to efficient algorithms.
In passing,
we revisit some of the results of~\cite{Plus,gaubert92a}, giving 
their geometric interpretation in terms of signed tropical
hyperplanes.

An ingredient of our approach is the introduction (in Section~\ref{sec-prelim}) of a rather general notion of {\em extension}
of the tropical semiring, together with
a general ``balance'' relation
$\bal$, which, depending on the details of the
extension, expresses the fact that the maximum is attained at least
twice in an expression, or that the maxima of two collections of terms
coincide.  Our constructions include as a special case
the symmetrized tropical semiring of Plus~\cite{Plus} and the bi-valued tropical semiring of Izhakian~\cite{Izhalone}, but also
a certain ``phase extension''
of the tropical semiring,
which is a variant of the complex tropical hyperfield of Viro.
Our notion also includes certain ``supertropical
semifields'' in the sense of Izhakian and Rowen~\cite{Izhsuper2010}.
Then auxiliary combinatorial results are presented in Section~\ref{sec-comb}.

The general affine Cramer
system reads $Ax \bal b$, whereas the homogeneous system reads
$Ax\bal \zero$, where $A$ is an $n\times n$ matrix 
and $b$ is a vector of dimension $n$, both 
with entries in an extension of the tropical semiring. To be interpreted
geometrically, the vector
$x$ which is searched will be
required to satisfy certain non-degeneracy conditions, typically that the coordinate of $x$ do not belong to the set of balanced non-zero elements of the extension. 

In this way, by developing methods of~\cite{AGG08},
we obtain in Section~\ref{sec-elim}
a general result, Theorem~\ref{th-cramer}, concerning
the unique solvability of non-singular Cramer system, which includes
Theorems~\ref{th-1} and~\ref{th-3} as special cases. 

Then we study in Section~\ref{sec-exists} the existence problem for the solution of
the affine Cramer system $Ax\bal b$, without making the non-singularity
assumptions needed in the previous uniqueness results. Theorem~\ref{theo-jacobi} below
gives a general existence theorem, with a constructive proof
based on the idea of the Jacobi algorithm in~\cite{Plus}. 
Theorem~\ref{theo-gauss-seidel} gives an alternative Gauss-Seidel
type algorithm. These results are valid in a large enough class of semirings,
including not only the symmetrized tropical semiring as in~\cite{Plus}, but 
also the phase extension of the tropical semiring.

In Section~\ref{sec-homogeneous}, we deal with the generalization
of Theorems~\ref{th-2} and~\ref{th-4} which concern singular linear systems
of $n$ equations in $n$ variables. 
In Theorem~\ref{newTGMa}, we characterize
the existence of non-degenerate solutions of 
$Ax \bal \zero$, 
recovering Theorems~\ref{th-2} and~\ref{th-4} as special cases.
This extends to more general semirings
a theorem of Gaubert~\cite{gaubert92a} dealing
with the case of the symmetrized tropical semiring $\smax$.
We note however that by comparison with the Jacobi/Gauss-Seidel type
results of Section~\ref{sec-exists}
the results of this section hold under more restrictive
assumptions on the semiring.

A geometrical interpretation of the previous results, in the
case of the symmetrized tropical semiring, is presented in Section~\ref{sec-geom}.

In Section~\ref{S_AC}, we start to address computational issues.
The $n+1$ Cramer determinants
of the system $Ax \bal b$ correspond to $n+1$ optimal assignment problems.
In Section~\ref{sec-perall}, we show that our approach
based on a Jacobi-type iterative method 
leads to an algorithm to compute the solution, as well as the 
Cramer determinants (up to signs), by solving a single (rather than $n+1$)
optimal assignment problem, followed by a single destination
shortest path problem.
For the sake of comparison, we revisit in Section~\ref{sec-transport}
the approach of Richter-Gebert, Sturmfels and Theobald~\cite{RGST}, building
on results of Sturmfels and Zelevinsky dealing with Minkowski sums of 
Birkhoff polytopes: it gives a reduction to a different transportation problem.
Although the original exposition of~\cite{RGST} is limited to instances
in general position, we show that some of their results remain valid
even without such an assumption.
The computation of the signs of tropical Cramer determinants is finally
briefly discussed in Section~\ref{sec-compute-det}.

\section{Semirings with a symmetry and a modulus}\label{sec-prelim}

\subsection{Definitions and first properties} 

\begin{definition}
A \NEW{semiring} is a set $\SS$ with two binary operations, addition, denoted by $+$,
and multiplication, denoted by $\cdot$ or by concatenation, such 
that:
\begin{itemize}
\item $\SS$ is an abelian monoid under addition (with neutral element denoted by $\oo$ and called zero);
\item $\SS$ is a monoid under multiplication (with neutral element denoted
by $\oi$ and called unit);
\item multiplication is distributive over addition on both sides;
\item $s \oo=\oo s=\oo$ for all $s\in \SS$.
\end{itemize}
\end{definition}
In the sequel, a semiring will mean a non-trivial semiring (different from
$\{\oo\}$).

Briefly, a semiring differs from a ring by the fact that 
an element may not have an additive inverse.
The first examples of semirings which are not rings that come to  mind are non-negative
integers $\N$, non-negative  rationals $\Q_+$ and
non-negative
reals ${\Re}_+$ with the usual addition and multiplication. There are
classical examples of non-numerical semirings as well. Probably the first such
example appeared in the work of Dedekind \cite{Dedekind} in connection
with the algebra of ideals of a commutative ring (one can add and multiply
ideals but it is not possible to subtract them). 

\begin{definition}
A semiring or an abelian monoid $\SS$ is called \NEW{idempotent} if $a+a=a$  for all $a\in\SS$,
 $\SS$ is called \NEW{zero-sum free} or \NEW{antinegative} if $a+b=0$ implies $a=b=0$ for all $a,b\in\SS$, and
$\SS$ is called \NEW{commutative}  if  $a\cdot b=b\cdot a$ for all $a,b\in\SS$.
\end{definition}
An idempotent semiring is necessarily zero-sum free.
We shall always assume that the semiring $\SS$  is commutative.

An interesting example of an idempotent semiring is the max-plus semiring
$$\rmax:=({\Re}\cup \{ -\infty \}, \oplus, \odot),$$
where $a\oplus b= \max\{a,b\}$ and $a\odot b= a+b$.  Here the zero element of the semiring is $-\infty$, denoted by $
\ooo$, and the unit of the semiring is 0, denoted by $\ooi$. 

The usual definition of matrix operations carries over to an arbitrary semiring.
We denote the set of  $m\times n$ matrices over $\SS$ by $\MM_{m,n}(\SS)$.
Also we denote $\MM_{n}(\SS)=\MM_{n,n}(\SS)$ 
and we identify $\SS^n$ with $\MM_{n,1}(\SS)$. Note that $\MM_{n}(\SS)$ is a semiring.

Some of the following notions were introduced in~\cite{AGG08},
where details and additional properties can be found.

\begin{definition}
Let $\SS$ be a semiring.
A map $\tau:\SS\to \SS$ is a symmetry if 
\begin{subequations}\label{sym}
\begin{align}
& \tau(a+b)=\tau(a)+\tau(b) \label{sym1}\\ 
& \tau(0)=0\label{sym2}\\ 
& \tau(a\cdot b)=a\cdot\tau(b)= \tau(a)\cdot b \label{sym3} \\
& \tau(\tau(a))=a .\label{sym4}
\end{align}
\end{subequations}
\end{definition}
\begin{example}
A trivial example of a symmetry is the identity map $\tau(a)=a$.
Of course, in a ring, we may take $\tau(a)=-a$. 
\end{example}

\begin{prop}\label{car-sym} A map $\tau$ is a symmetry of the semiring $\SS$
if and only if there exists $e\in\SS$ such that $e\cdot e=1$, and
$ \tau(a)= e\cdot a=a\cdot e$ for all $a\in\SS$ (hence $e$ commutes 
with all elements of $\SS$).
\end{prop}
\begin{proof}
Let  $\tau$ be a symmetry of $\SS$, and denote $e=\tau(1)$.
By~\eqref{sym3}, we get that
$\tau(a)=\tau(1\cdot a)=e \cdot a$, and 
$\tau(a)=\tau(a\cdot 1)=a \cdot e$ for all $a\in\SS$.
By~\eqref{sym4}, we deduce that $\tau(e)=\tau(\tau(1))=1$, and
since $\tau(e)=e\cdot e$, we get $e\cdot e=1$.
Conversely, if $\tau(a)= e\cdot a=a\cdot e$ for all $a\in\SS$, with 
$e\cdot e=1$, then $\tau$ satisfies all the conditions in~\eqref{sym}.
\end{proof}

In the rest of the paper we shall write $-a$ for $\tau(a)$. So, $a-a$ is 
not zero generally speaking, but is a formal sentence meaning $a+\tau(a)$.
Moreover,  for any integer $n\geq 0$, $(-1)^n$ will mean the 
$n$th power of $-1=\tau(1)$, hence the product of $n$ copies of $-1$.
Also $+a$ will mean $a$, in particular in the
formula $\pm a$. If the addition of $\SS$ is denoted by $\oplus$ 
instead of $+$, then $+a$ and $-a$ will be replaced by $\oplus a$ and
$\ominus a$.

\begin{definition}\label{circ}
For any $a\in \SS$, we set $a^\circ:=a-a$,
thus $-a^\circ=a^\circ=(-a)^\circ$, and we denote
\[
\SS^\circ := \{a^\circ \mid a\in \SS\}  \enspace 
\]
The elements of this set will be called \NEW{balanced elements} of $\SS$.
Moreover, we define the \NEW{balance} relation $\balance$ on $\SS$ by 
$a \balance b$ if $a-b\in \SS^\circ $.
\end{definition}
Note that $\SS^\circ$ is an ideal, hence 
the relation $\balance$ is reflexive and symmetric.
It may not be transitive.
Since $\SS^\circ$ is an ideal, it contains an invertible element of $\SS$
if and only if it coincides with $\SS$. 
We shall only consider in the sequel symmetries such that $\SS^\circ\neq \SS$.
This permits the following definition.

\begin{definition}
When $\SS^\circ\neq \SS$, 
we say that a subset $\SS^\vee$ of $\SS$ is \NEW{\thinp} if 
$\SS^\vee \subset (\SS\setminus \SS^\circ )\cup\{\oo\}$
and if it contains $\oo$ and all invertible elements of $\SS$.
When such a set $\SS^\vee$  is fixed, its elements will be 
called \NEW{\thin elements}.
\end{definition}
Note that $0$ is both a balanced and a \thin element. In the sequel,
we shall consider systems of linear ``equations'' (the equality relation will be replaced by balance), and we will require the variables to be \thinp.
By choosing appropriately the set of \thin elements, we shall
see that standard Gauss type elimination algorithms carry
over. Also,
in most applications, only the \thin solutions will have simple
geometrical interpretations.
Note that in~\cite{AGG08}, we only used the notation $\SS^\vee$ for the maximal
possible \thin set, that is $(\SS\setminus \SS^\circ )\cup\{\oo\}$;
however, it will be useful to consider also for instance the smallest
possible \thin set, which is the set of invertible
elements completed with $0$. 

Recall that  $(R,\cdot,\leq)$ is an \NEW{ordered semigroup} if
$(R,\cdot)$ is a semigroup, and $\leq$ is an order on $R$ such that
for all $a,b,c\in R$, $a\leq b$ implies $a\cdot c\leq b\cdot c$
and $c\cdot a\leq c\cdot b$.
An \NEW{ordered monoid} is a monoid that is an ordered semigroup.
An \NEW{ordered semiring} is a semiring $(\SS,+,\oo,\cdot,\oi)$ 
endowed with an order relation $\leq$ such that
$(\SS,+,\leq)$ and $(\SS,\cdot,\leq)$ are ordered semigroups.

\begin{definition} For any subsemiring $\TT$ of $\SS$, one defines the 
relations:
\[ a\preceq^\TT b \iff b\succeq^\TT a\iff b=a+c \text{ for some } c\in \TT
\enspace.\]
We shall write simply $\preceq$ instead of $\preceq^\SS$, and
$\preceq^\circ$ instead of  $\preceq^{\SS^\circ}$.
\end{definition}
These relations are preorders (reflexive and transitive),
compatible with the laws of $\SS$. They may not be 
antisymmetric. 
\begin{definition} The preorder $\preceq$ is
called the  \NEW{natural preorder} on $\SS$.
A semiring $\SS$ is said to be \NEW{naturally ordered} when
$\preceq$ (or equivalently $\succeq$) is an order relation, and 
in that case $\preceq$ is called the \NEW{natural order} on $\SS$, and
$\succeq$ is its opposite order. The notation $\prec$ and $\succ$ will
be used for the corresponding strict relations. 
\end{definition}
When $\preceq$ (or $\succeq$) is an order relation, so are $\preceq^\TT$,
$\preceq^\circ$, $\succeq^\TT$, and $\succeq^\circ$.
An idempotent semiring is necessarily naturally ordered, and 
a naturally ordered semiring is necessarily zero-sum free.
We also have:
\begin{align}\label{succeq-bal}
 a\preceq^\circ b \text{ or } b\preceq^\circ a \Rightarrow a\balance b
\enspace.
\end{align}
The converse is false in general.

\begin{definition}
Let $\SS$ be a semiring and $\RR$ be an idempotent semiring
in which the natural order is total, or for short, a
totally ordered idempotent semiring.
We say that a map $\mu:\SS\to \RR$ is a \NEW{modulus} if it is a 
surjective morphism of semirings.
In this case, we denote $\mu(a)$ by $|a|$ for all $a\in \SS$.
\end{definition}

We shall apply the notations $\balance$,  $\preceq^\TT$, $\succeq^\TT$
(so $\preceq$, $\succeq$,
 $\preceq^\circ$,  and $\succeq^\circ$), $|\cdot|$ to matrices
and vectors, understanding that the relation holds entrywise. 
We shall do the same for the notions of ``balanced'' or ``\thinp'' elements.

\begin{prop}\label{sym-tot}
If $(\RR,+,\ooo,\cdot,\ooi)$ is a totally ordered idempotent semiring, then 
the only symmetry on $\RR$ is the identity.
Moreover, any semiring $(\SS,+,0,\cdot,1)$ with a symmetry and a modulus $| \ |: \SS \to \RR$ satisfies
$|-1|=|1|=\ooi$.
\end{prop}
\begin{proof}
Let $\tau$ be a symmetry of $\RR$. Then by Proposition~\ref{car-sym},
there exists $e\in\RR$ such that $e\cdot e=\ooi$, and
$ \tau(a)= e\cdot a=a\cdot e$ for all $a\in\RR$.
Since $\RR$ is totally ordered, either $e\leq \ooi$ or $\ooi\leq e$.
Assuming that the inequality  $e\leq \ooi$ holds, then 
multiplying it by $e$, we get
$e\cdot e \leq e\cdot \ooi=e$, and since $e\cdot e=\ooi$ and
$\leq$ is an order relation, we deduce that $e=\ooi$.
The same is true when the inequality $e\geq \ooi$ holds. This shows 
that $e=\ooi$, hence $\tau$ is the identity map.

If now $\SS$ is a semiring with a symmetry and a modulus,
then $|-1|\cdot |-1|=|(-1)\cdot (-1)|=|1|=\ooi$, and by the above arguments,
we get that $|-1|=\ooi$.
\end{proof}

\subsection{Tropical extensions of semirings}
In the sequel, we shall consider semirings with a symmetry and a modulus.
The following construction allows one to obtain easily such semirings.
\begin{propdefinition}[Extension of semirings]\label{extension}
Let $(\SS,+,0,\cdot,1)$ be a semiring and let $(\RR,\oplus,\ooo,\cdot,\ooi)$
be a totally ordered idempotent semiring.
Then the semiring $\skewproductbar{\SS}{\RR}$ is defined as
the set $\SS\times \RR$ endowed with the operations
\[
(a,b)\oplus(a',b')=\begin{cases}
(a+a',b) & \text{if } b=b' \\
(a,b) & \text{if } b\succ b' \\
(a',b') & \text{if } b\prec b' 
\end{cases}
\qquad\text{and}\quad 
(a,b)\odot (a',b')=(a \cdot a',b\cdot  b') \enspace .
\]
Its zero element is $(\oo,\ooo)$ and its unit is $(\oi,\ooi)$.
For any subset $\SSp$ of a semiring $\SS$, we denote \[
\SSp^*:=\SSp\setminus\{0\} \enspace ,
\]
in particular, $\RR^*:=\RR\setminus\{\ooo\}$. 
Then the extension of $\RR$ by a subset $\SSp$ of $\SS$ is defined by
\begin{align}
\skewproduct{\SSp}{\RR}:=&(\SSp\times \RR^*) \cup \{(\oo,\ooo)\}\enspace .
\end{align}
If $\SSp$ is a subsemiring of $\SS$, then $\skewproductbar{\SSp}{\RR}$,
and $\skewproduct{\SSp}{\RR}$ are subsemirings of $\skewproductbar{\SS}{\RR}$.
If in addition $\SSp$ is zero-sum free and without zero divisors, 
then
\begin{align}
\skewproductstar{\SSp}{\RR}:=& \skewproduct{\SSp^*}{\RR}
\end{align}
is a subsemiring of $\skewproduct{\SSp}{\RR}$. \qed
\end{propdefinition}
This construction
bears some similarity with a semidirect product, which motivates
the notations ``$\bar{\rtimes}$'', $\ushort{\rtimes}$ and ``$\rtimes$''. 
We shall denote by $\zero$ and $\unit$, instead of $(\oo,\ooo)$ and
$(\oi,\ooi)$, the zero and unit of  $\skewproductbar{\SS}{\RR}$.

If $(R,\cdot,\ooi,\leq)$ is a totally ordered monoid, 
completing $R$ with a bottom element, denoted by $\ooo$,
we get the totally ordered idempotent semiring 
$(\RR:=R\cup\{\ooo\},\max,\ooo,\cdot,\ooi)$.
All the semirings $\RR$
satisfying the assumptions of Proposition~\ref{extension} are of this form. 
When $R=\Re$ is equipped with its usual order and addition,
we recover the max-plus semiring $\rmax$.
We may take more generally for $R$ any submonoid of $(\mathbb{R},+)$,
or take $\Re^d$ equipped with the lexicographic order
and entrywise addition.

The intuition of the construction of Proposition~\ref{extension}
is best explained by the following
example. 
\begin{example}[Complex extension of the tropical semiring]\label{ex-best}
Let $\C$ denote the field of complex numbers.
Then the semiring $\Cext$ will be called ``complex extension of the tropical semiring''. 
An element $(a,b)\in \Cext$
encodes the asymptotic expansion $a \epsilon^{-b}+o(\epsilon^{-b})$,
when $\epsilon$ goes to $0_+$
(when $(a,b)=\zero$, this is the identically $0$ expansion).
Indeed, the ``lexicographic'' rule
in the addition of $\Cext$ corresponds
precisely to the addition of asymptotic expansions, and the entrywise product
of $\Cext$ corresponds to the product of asymptotic expansions.
By taking the zero-sum free subsemiring $\Re_+\subset \C$ consisting
of the real nonnegative numbers,
we end up with the subsemiring $\skewproductstar{\Re_+}{\rmax}$, which encodes the asymptotic
expansions $a\epsilon^{-b}+o(\epsilon^{-b})$ with $a>0$ and $b\in\Re$,
together with the identically $0$ expansion. The latter semiring was
used in~\cite{Bapat} under the name of semiring of (first order)
jets, to study eigenvalue perturbation problems. 
\end{example}
We next list some simple
facts concerning extensions of semirings.

\begin{fact}\label{remark-extension0}
Let $\B:=\{\ooo,\ooi\}$ be the Boolean semiring, i.e., the idempotent semiring
with two elements, 
let $\SS$ be zero-sum free and without zero divisors. Then
$\skewproductstar{\SS}{\B}$ is isomorphic to $\SS$. 
In general $\SS$ is a subsemiring of $\skewproduct{\SS}{\RR}$ 
since the injective map 
$\jmath:\SS\to \skewproductbar{\SS}{\RR}$ defined by
\[\jmath(a)=\begin{cases}
(a,\ooi)  & \text{ if } a\neq \oo,\\
\zero  & \text{ if } a = \oo, \end{cases} \]
is a morphism of semirings. 
So the semirings $\skewproductbar{\SS}{\RR}$ and $\skewproduct{\SS}{\RR}$ 
 are semiring extensions of $\SS$. The same is true for
$\skewproductstar{\SS}{\RR}$ as soon as $\SS$ is zero-sum free 
without zero divisors.
Note that one can also consider the map
$\gamma:\skewproductbar{\SS}{\RR}\to\SS$, such that $\gamma(a,b)=a$,
for all $(a,b)\in \skewproductbar{\SS}{\RR}$.
This map yields a surjective multiplicative morphism
from $\skewproductbar{\SS}{\RR}$ (or from $\skewproduct{\SS}{\RR}$ or $\skewproductstar{\SS}{\RR}$) to $\SS$,
such that the composition $\gamma\circ \jmath$ equals the identity of
$\SS$.
\end{fact}

\begin{fact}\typeout{Mistake here, what is R?}
If $\SS$ and $\RR$ are commutative then so is $\skewproductbar{\SS}{\RR}$.
If $\SS$ is idempotent (resp., naturally ordered) then
so is $\skewproductbar{\SS}{\RR}$, and consequently $\skewproduct{\SS}{\RR}$ and  
$\skewproductstar{\SS}{\RR}$.
The natural preorder on $\skewproductbar{\SS}{\RR}$ is the
lexicographic preorder: $(a,b)\preceq (a',b')$ if and only if either
$b\prec b'$, or $b=b'$ and $a\preceq a'$.
\end{fact}

\begin{fact}\label{remark-extension} 
Let $\tau$ be a symmetry of $\SS$. We define the map $\tau'$ from $\skewproductbar{\SS}{\RR}$ 
to itself  by $\tau'((a,b))=(\tau(a),b)$ for all $a\in\SS$ and
$b\in\RR$. Then $\tau'$  is a symmetry of $\skewproductbar{\SS}{\RR}$,
which preserves $\skewproduct{\SS}{\RR}$ and  $\skewproductstar{\SS}{\RR}$.
We shall call $\tau'$ the extension of $\tau$.
Conversely, if $\tau'$ is a symmetry of $\skewproductbar{\SS}{\RR}$, then
it is the extension of some symmetry $\tau$ of $\SS$.

For these symmetries, 
we have $(a,b)^\circ=(a^\circ,b)$ for all $a\in\SS$ and $b\in\RR$, hence 
$(\skewproductbar{\SS}{\RR})^\circ=\skewproductbar{\SS^\circ}{\RR}$,
$(\skewproduct{\SS}{\RR})^\circ=\skewproduct{\SS^\circ}{\RR}=(\skewproductbar{\SS}{\RR})^\circ\cap (\skewproduct{\SS}{\RR})$,
$(\skewproductstar{\SS}{\RR})^\circ=\skewproductstar{\SS^\circ}{\RR}=(\skewproductbar{\SS}{\RR})^\circ\cap (\skewproductstar{\SS}{\RR})$.
If $\SS^\vee$ is a \thin set of $\SS$, then
$((\SS^\vee)^*\times \RR)\cup\{\zero\}$
is a \thin set of $\skewproductbar{\SS}{\RR}$,
and $\skewproductstar{\SS^\vee}{\RR}$ is a \thin set of $\skewproduct{\SS}{\RR}$,
$\skewproductstar{\SS}{\RR}$, and also of $\skewproductbar{\SS}{\RR}$.
\end{fact}
\begin{proof}
We only need to show that any symmetry $\tau'$ of $\skewproductbar{\SS}{\RR}$
is the extension of some symmetry of $\SS$, the other properties are immediate.
By Proposition~\ref{car-sym}, $\tau'((a,b))=e'\odot (a,b)= (a,b)\odot e'$,
for all $(a,b)\in\skewproductbar{\SS}{\RR}$,  
for some $e'=(e,f)\in \skewproductbar{\SS}{\RR}$ such that $(e,f)\odot (e,f)=(\oi,\ooi)$.
Denoting $\tau(a)=e\cdot a$ and $\sigma(b)=f\cdot b$,
we get that  $\tau'(a,b)=(\tau(a),\sigma(b))$.
The properties of $(e,f)$ imply that $\tau$ is a symmetry of $\SS$
and $\sigma$ is a symmetry of $\RR$.
Moreover, since $\RR$ is totally ordered, then
by Proposition~\ref{sym-tot}, $\sigma$ is the identity map, or $f=\ooi$.
This shows that $\tau'$ is necessarily the extension of the symmetry $\tau$ of $\SS$.
\end{proof}

\begin{fact} \label{remark-extension3} 
Let $\SS$ and $\RR$ be as in Proposition~\ref{extension}. Then
the map $\mu:\skewproductbar{\SS}{\RR}\to\RR$, $(a,b)\mapsto b$  is a modulus. Also
its restriction to $\skewproduct{\SS}{\RR}$ or $\skewproductstar{\SS}{\RR}$ is a modulus.
If $\SS$ is naturally ordered, then $\mu$ is order preserving.
Moreover, the map $\imath:\RR \to \skewproductbar{\SS}{\RR}$
defined by 
\[ \imath(b)=\begin{cases}(1,b) & \text{ if } b\in \RR^*, \\
 \zero & \text{ if } b=\ooo ,\end{cases}\]
is an injective and  multiplicative morphism, and its image is 
included in $\skewproductstar{\SS}{\RR} \subset\skewproduct{\SS}{\RR}$.
The composition $\mu\circ \imath$ equals the identity on $\RR$,
thus the image of $\imath$ is the set of fixed points of $\imath\circ \mu$.
The map $\imath$ is a semiring morphism if and only if $\SS$ is idempotent.
More generally, if $\SS$ is naturally ordered, then $\imath$
is an order preserving map from $(\RR,\preceq)$ to $\skewproductbar{\SS}{\RR}$ endowed with 
the natural order $\preceq$, and it satisfies: 
\begin{equation}\label{inj-quasi-morph}
\imath(x)\vee \imath(y) \preceq \imath(\max(x,y))\preceq 
\imath(x) \oplus \imath(y)\enspace,
\end{equation}
where $\vee$ is the supremum in the natural order of $\skewproductbar{\SS}{\RR}$.
\end{fact}

\subsection{Symmetrized max-plus semiring}
The {\em symmetrized max-plus semiring\/}, which is useful to deal with 
systems \
of linear equations over $\rmax$, was introduced in~\cite{Plus}.
It is also discussed in~\cite{gaubert92a}, \cite{bcoq} and~\cite{AGG08}.
Here we shall adopt a construction of this semiring using the above
extension. The resulting semiring is isomorphic to the one of~\cite{Plus},
as shown in~\cite{AGG08}. We start by defining the symmetrized Boolean
semiring. An alternative construction was given in~\cite[Section 5.1]{AGG08}.

\begin{definition}[Symmetrized Boolean semiring]\label{def-symb} 
Let $\BB$ be a set  with four elements denoted  
$\ooo, \ooi,\ominus \ooi$ and $\ooi^\circ$.
Define the laws $\oplus, \otimes$ on $\BB$ by:
\[ 
\begin{tabular}{|r|rrrr|}
\hline
$\oplus$ & $\ooo$ & $\ooi$ & $\ominus \ooi$ &$ \ooi^\circ$  \\ \hline
$\ooo$ & $\ooo$ & $\ooi$ & $\ominus \ooi$ &$ \ooi^\circ$  \\
$\ooi$ & $\ooi$ & $\ooi$ & $\ooi^\circ$ &$ \ooi^\circ$  \\ 
$\ominus \ooi$ &$\ominus \ooi$ & $\ooi^\circ$ & $\ominus \ooi$ &$ \ooi^\circ$  \\ 
$ \ooi^\circ$  & $ \ooi^\circ$&$ \ooi^\circ$ & $ \ooi^\circ$ &$ \ooi^\circ$  \\ \hline
\end{tabular}
\quad 
\begin{tabular}{|r|rrrr|}
\hline
$\otimes$ & $\ooo$ & $\ooi$ & $\ominus \ooi$ &$ \ooi^\circ$  \\ \hline
$\ooo$ & $\ooo$ & $\ooo$ & $\ooo$ &$\ooo$  \\
$\ooi$ & $\ooo$ & $\ooi$ & $\ominus \ooi$ &$ \ooi^\circ$  \\ 
$\ominus \ooi$ &$\ooo$ & $\ominus\ooi$ & $\ooi$ &$ \ooi^\circ$  \\ 
$ \ooi^\circ$  & $\ooo$ & $\ooi^\circ$ & $\ooi^\circ$ &$ \ooi^\circ$  \\ \hline
\end{tabular}
\]
Then $(\BB,\oplus,\otimes)$ is
an idempotent semiring with zero element $\ooo$ and
unit element $\ooi$,
and the map  $\tau:\BB\to\BB$, such that 
$\tau(a)=(\ominus \ooi)\otimes a$, for all $a\in\BB$,
is a symmetry of $\BB$, such that $\ooi^\circ=\ooi\oplus \tau(\ooi)$.
\end{definition}
Since $\BB$ is idempotent, it is naturally ordered. 
The order of $\BB$ satisfies:
\begin{center}
\setlength{\unitlength}{0.6pt}
\begin{picture}(100,110)
\put(45,102){$\ooi^\circ$}
\put(45,100){\line (-1,-1){38}}
\put(55,100){\line (1,-1){38}}
\put(1,48){$\ooi$}
\put(90,48){$\ominus \ooi$}
\put(5,43){\line (1,-1){38}}
\put(95,43){\line (-1,-1){38}}
\put(45,0){$\ooo$}
\end{picture}
\end{center}
The above properties imply that the notation $\ooi^\circ$
is coherent with the one of Definition~\ref{circ}.
We have $\BB^\circ:=(\BB)^\circ=\{\ooo,\ooi^\circ\}$,
and since $\ooi$ and $\ominus \ooi$ are invertible,
the only possible \thin set of $\BB$ is
$\BB^\vee:=\{\ooo,\ooi,\ominus \ooi\}$.

\begin{definition}[Symmetrized max-plus semiring]\label{def-smax}
The symmetrized max-plus semiring, $\smax$, is defined
to be $\skewproductstar{\BB}{\rmax}$, according to Proposition~\ref{extension}.
\end{definition}
Indeed, since $\BB$ is zero-sum free, with no zero divisors,
the extension $\skewproductstar{\BB}{\rmax}$ 
is a semiring. By Fact~\ref{remark-extension}, the symmetry $\tau$ of 
$\BB$ is extended into the symmetry $\tau'$ of $\skewproductstar{\BB}{\rmax}$.
The definition of the symmetrized max-plus semiring
given in~\cite{Plus} leads to a structure isomorphic
to $\skewproductstar{\BB}{\rmax}$, which was denoted by $\smax$ in this reference.
In the present paper, $\smax$ is directly defined as $\skewproductstar{\BB}{\rmax}$.

By Fact~\ref{remark-extension}, we have 
$\S^\circ:=(\S)^\circ=\skewproductstar{\BB^\circ}{\rmax}=(\{\ooi^\circ\}\times \Re)\cup\{\zero\}$.
Moreover $\S^\vee:=\skewproductstar{\BB^\vee}{\rmax}=(\{\ooi,\ominus \ooi\}\times \Re)\cup\{\zero\}$
is a \thin set of $\S$, and it is indeed the only possible one.
We shall also use the notations
$\S^\oplus:=(\{\ooi\}\times \Re) \cup\{\zero\}$ and 
$\S^\ominus:=(\{\ominus \ooi\}\times \Re)\cup\{\zero\}$,
thus $\S^\vee=\S^\oplus\cup  \S^\ominus$.
By Fact~\ref{remark-extension3}, $\imath$ is a morphism,
so that we can identify $\rmax$ with $\imath(\rmax)=\S^\oplus$.
We have $\S^\ominus=\ominus \S^\oplus$.
In~\cite{Plus}, the elements of $\S^\vee$ are called \NEW{signed},
thus ``signed'' in this particular semiring is equivalent
to ``\thinp''.

\begin{remark}
In the idempotent semiring $(\BB,\oplus,\otimes)$,
the elements of $(\BB^\vee)^*$  are 
not comparable in the natural order,
and $\zero$ and $\ooi^\circ$ are respectively the minimal and 
maximal elements of $\BB$.
\end{remark}

It is natural to extend the tropical semiring by capturing the phase information, rather than the sign. The next construction yields a coarse way to do so.
\begin{example}[Tropical extension of the torus, $\skewproductstar{\bar{\To}}{\rmax}$]
\label{groupext}
Let $(G,\otimes,\ooi)$ be a group,
equip it with the trivial order $\leq$ such that every two elements are 
incomparable (i.e., $a\leq b$ if and only if $a=b$),
and add a minimal and maximal element to $G$, denoted respectively
$\zero$ and $\ooi^\circ$, such that $\zero$ is absorbing for the
multiplication in $\bar{G}:=G\cup \{\zero,\ooi^\circ\}$,
and $\ooi^\circ$ is absorbing for the multiplication
in $G\cup \{\ooi^\circ\}$.
Then $(\bar{G},\vee,\zero,\otimes,\ooi)$ is an idempotent semiring
in which $a\vee b=\ooi^\circ$ for all $a,b\neq \zero$, such that
$a\neq b$.
Assume that there exists $e\in G\setminus\{\ooi\}$, such that
$e\otimes e=\ooi$ and $e$ commutes with all elements of $G$.
Then the map $a\in \bar{G}\mapsto e\otimes a$ is a non trivial symmetry
of $\bar{G}$, and since $a\neq e\otimes a$ for all $a\in G$,
we get that $\bar{G}^\circ=\{\zero,\ooi\}$
and that $\bar{G}^\vee=G\cup\{\zero\}$ is the only thin set of $\bar{G}$.
Since $\bar{G}$ is zero-sum free without zero divisors,
one can then construct $\skewproductstar{\bar{G}}{\rmax}$
with the thin set  $\skewproductstar{\bar{G}^\vee}{\rmax}$.
When $G$ is the group with two elements (of order $2$, so isomorphic
to the additive group $\Z_2$), we recover the semirings $\BB$ and $\S$.
When $G$ is the unit circle $\To$ of $\C$, we obtain a semiring
$\skewproductstar{\bar{\To}}{\rmax}$, with 
only one possible non trivial symmetry obtained with $e=-1$.
\end{example}
A more powerful semiring than $\skewproductstar{\bar{\To}}{\rmax}$ is obtained by the following 
construction which is a variant of the one
of the complex tropical hyperfield that Viro~\cite{1006.3034}
made, using a different set of axioms, see Remark~\ref{viro-remark} below.
\begin{example}[Phase extension of the tropical semiring]
\label{ex-viro}
Let $\phase$ (for ``phases'') denote the set 
of closed convex cones of $\C$ 
seen as a real 2-dimensional space, that is the set 
of angular sectors of $\C$ between two half-lines with angle less or equal to $\pi$ or
equal to $2\pi$, together with the singleton $\{0\}$
(the trivial cone). 
Consider the following laws on $\phase$:
the sum $\Phi+\Phi'$ of two elements $\Phi,\Phi'\in \phase$ is the closed convex hull of
$\Phi\cup \Phi'$, and the product $\Phi\cdot \Phi'$ is the closed convex hull
of the set of complex numbers $a\cdot a'$ with $a\in \Phi$ and $a'\in \Phi'$.
Then $\phase$ is an idempotent semiring. The zero is $\ooo:=\{0\}$,
the unit is the half-line of positive reals, and the invertible elements
are the half-lines. 
Taking $-\Phi$ equal to the set of $-a$ for $a\in \Phi$,
we obtain a symmetry of $\phase$, which is the only symmetry of $\phase$ 
different from identity.
In that case, $\phase^\circ$ is the subset of $\phase$ consisting of $\zero$, 
all lines, 
and the plane, and we can consider for the \thin set $\phase^\vee$ the set
of half-lines and $\zero$.
We can then construct the semiring $\skewproductstar{\phase}{\rmax}$
 equipped with the \thin set 
$\skewproductstar{\phase^\vee}{\rmax}$. We call this semiring the {\em phase
extension} of the tropical semiring.

Similarly to Example~\ref{ex-best}, an element $(\Phi,b)$ of 
$\skewproductstar{\phase}{\rmax}$ may be thought
of as an abstraction of the set of asymptotic expansions of the form
 $a \epsilon^{-b}+o(\epsilon^{-b})$,
when $\epsilon$ goes to $0_+$, where $a$ is required
to belong to the relative interior of $\Phi$, 
denoted by $\operatorname{relint}\Phi$. Recall
that the relative interior of a convex set is the interior
of this set with respect to the topology of the affine space
that it generates. For instance, the relative interior of a closed
half-line is an open half-line.
If $a\in \operatorname{relint}\Phi$,
and if $a'\in \operatorname{relint}{\Phi'}$
for some $\Phi,\Phi'\in \phase^*=\phase\setminus\{\zero\}$,
then, it can readily be checked that $a \epsilon^{-b}+o(\epsilon^{-b})
+ a' \epsilon^{-b'}+o(\epsilon^{-b'})
= a''\epsilon^{-b''} +o(\epsilon^{-b''})$,
where $a''\in \operatorname{relint}\Phi''$ and
$(\Phi'',b''):= (\Phi,b)\oplus (\Phi',b')$. Similarly,
the product of the semiring $\phase$ is consistent with the one of asymptotic
expansions. Note that when the cone $\Phi$ is either a line
or the whole set $\mathbb{C}$, $0$ is in the relative interior 
of $\Phi$. Then the corresponding asymptotic expansion
$a \epsilon^{-b}+o(\epsilon^{-b})$
may reduce to $o(\epsilon^{-b})$, as $a=0$ is allowed.
The elements of $\skewproductstar{\phase^\vee}{\rmax}$ correspond
to  asymptotic expansions with a well defined information on the angle, whereas
an element $(\Phi,b)$ such that $\Phi$ is a pointed cone
(a sector of angle strictly
inferior to $\pi$) correspond to asymptotic expansions 
having their leading term in a given angular sector.
\end{example}

\begin{remark}[Viro's complex tropical hyperfield]
\label{viro-remark}
A related encoding was proposed by Viro in~\cite{1006.3034} in a different setting, with his {\em complex tropical hyperfield} ${\TT}\C$.
A hyperfield is a set endowed with a multivalued addition
and univalued multiplication, that satisfy distributivity and invertibility 
properties similar to those of semifields.
The hyperfield ${\TT}\C$ is the set of complex numbers $\C$ 
endowed with a multivalued addition and the usual multiplication.
This allows one to see a non zero complex number $e^{i\theta+b}$ as an encoding
of asymptotic expansions of the form
 $r e^{i\theta} \epsilon^{-b}+o(\epsilon^{-b})$, when $\epsilon$ goes to $0_+$,
with $r> 0$. 
Hence, the phase extension
of the tropical semiring and the complex tropical hyperfield provide
two abstractions of the arithmetics of asymptotic expansions
(or of Puiseux series). 
The two abstractions differ, however,
in the handling of the element $x-x$. 
Indeed, if $x\in {\TT}\C$, then, $x-x$ is defined to be
$\set{y\in\C}{|y|\leq |x|}$ in ${\TT}\C$.
This set may be thought of as an encoding of all the expansions
in $\epsilon$ that are $O(\epsilon^{-|x|})$. 
If $x=(\Phi,b) \in \skewproductstar{\phase^\vee}{\rmax}$ and
$\theta$ is the angle of the half-line $\Phi$, then $x-x$ 
encodes all the asymptotic expansions 
$r e^{i\theta} \epsilon^{-b}+o(\epsilon^{-b})$ with $r\in\mathbb{R}$ (so we
get an extra bit of information by comparison with $O(\epsilon^{-b})$).
Note that we may also identify $\skewproductstar{\phase^\vee}{\rmax}$ with $\C$
and $\skewproductstar{\phase}{\rmax}$ with subsets of $\C$,
by means of the bijective map
$(\Phi,b)\mapsto e^{i\theta+b}$ where $\theta$ is as above,
and $(\zero,\zero)\mapsto 0$.
Again, in this identification, the multiplication and addition of
${\TT}\C$ and $\skewproductstar{\phase^\vee}{\rmax}$ coincide except 
for $x-x$. Moreover, in this way,
$\skewproductstar{\phase^\vee}{\rmax}$ is not a hyperfield
(since $0\not\in x-x$), and ${\TT}\C$ cannot be put in the form of
$\skewproductstar{\SS^\vee}{\rmax}$ for some semiring $\SS$.
\end{remark}

\subsection{The bi-valued tropical semiring}\label{sec-ext-izh} 
Izhakian introduced in~\cite{Izhalone} an extension of the tropical
semiring, which can be cast in the previous general construction.
We shall also see that some of the supertropical semifields of
Izhakian and Rowen~\cite{Izhsuper2010} can be 
reduced to the previous construction.
The following presentation is a simplified version of~\cite{AGG08}.

\begin{definition}\label{def-Ti}
Let $\N_{2}$ be the semiring which is the quotient of the semiring $\N$ of non-negative integers by the equivalence relation which identifies all numbers greater than or equal to $2$. The \NEW{bi-valued tropical semiring} is 
\[ \Ti:=\skewproductstar{\N_2}{\rmax} \]
in the sense of Proposition~\ref{extension}.
\end{definition}
Indeed, $\N_2$  is zero-sum free and
without zero divisors, thus $\Ti=\skewproductstar{\N_2}{\rmax}$ is a subsemiring of
$\skewproductbar{\N_2}{\rmax}$. Moreover $\Ti$ is isomorphic to the 
extended tropical semiring defined in~\cite{Izhalone}, see~\cite{AGG08} for details. (In the present paper, we prefer to use the term {\em bi-valued}
rather than {\em extended} since
other extensions of the tropical semiring are considered.)
Recall that this algebraic structure encodes whether the maximum
in an expression is attained once or at least twice.
The semiring $\N_2$ is not idempotent so the injection $\imath$ is
not a morphism. However,  $\N_2$ is naturally ordered (by
the usual order of $\N$), so is the semiring $\Ti$, and 
$\imath$ satisfies~\eqref{inj-quasi-morph}.
The only element $e$ of $\N_2$ such that $e\cdot e=1$ is equal to $1$, thus
the only symmetry of $\N_2$ is the identity map.
Since by Fact~\ref{remark-extension}, a symmetry of $\Ti$ is the
 extension of a symmetry of $\N_2$, the only symmetry of $\Ti$
is the identity map.
In $\N_2$, we have $2=1^\circ$, $\N_2^\circ=\{0,1^\circ\}$ and 
the only possible \thin set is $\N_2^\vee:=\{0,1\}$.
Then $\Ti^\circ=(\{1^\circ\}\times \Re)\cup \{\zero\}$ and
$\Ti^\vee=(\{1\}\times \Re)\cup \{\zero\}=\imath(\rmax)$
is a \thin set of  $\Ti$, which is the only possible one.

\begin{remark}[Supertropical semirings as semirings with symmetry]
\label{rem-supertropical}
In~\cite{IzhRowen}, in the particular
context of the semiring $\Ti$, 
the names ``reals'' and ``ghosts'' were given to what we call here
``\thin'' and ``balanced'' elements.
The construction of the bi-valued tropical semiring
has been generalized to the
notion of {\em supertropical} semifield or semiring in~\cite{Izhsuper2010}
and of {\em layered} semiring 
in~\cite{izhrowen2012}.
Supertropical semirings are special cases
of semirings with a symmetry and a modulus.
Indeed, one can show that the triple $(\SS,{\GG}_\zero,\nu)$ is
a {\em supertropical semiring} in the sense of~\cite{Izhsuper2010}
if and only if the following conditions hold.
(i) $\SS$ is a
naturally ordered semiring, endowed with the identity symmetry,
such that $\SS^\circ$ is a totally ordered idempotent semiring, the
map $\mu:\SS\to \SS^\circ, a\mapsto a^\circ $ is a modulus,
and $\SS$ satisfies the additional properties:
$a+b=b$ if $|a|\prec |b|$ and $a+b=|a|$ if $|a|=|b|$.
(ii) ${\GG}_\zero$ is an ideal of $\SS$
containing $\SS^\circ$, and $\nu$ is the map from $\SS$ 
to ${\GG}_\zero$ such that $\nu(a)=\mu(a)$.
Note that the idempotency of $\nu$ follows from the fact that $\mu$ is 
a morphism.
Similarly, a supertropical semiring is a {\em supertropical semifield} 
if the following additional conditions hold.
(iii) $\SS^\sharp:=\SS\setminus\SS^\circ$ is  a multiplicative 
commutative group such that the map $\mu$ is onto from $\SS^\sharp$ 
to $(\SS^\circ)^*$.
The latter properties mean that $\SS^\sharp\cup\{\zero\}$ is 
the unique \thin set $\SS^\vee$ of $\SS$, and that 
any element of $\SS^\circ$ can be written as $a^\circ$ with $a\in \SS^\vee$.
(iv) ${\GG}_\zero=\SS^\circ$ and $\nu=\mu$.
Since Conditions (ii) and (iv) concern only ${\GG}_\zero$ and $\nu$
and do not affect $\SS$,
we shall use in the sequel the name ``supertropical semifield''
for any semiring $\SS$ satisfying the above conditions (i) and (iii).
\end{remark}

\begin{example}\label{ex-supertropical}
The following construction gives the main example of a
supertropical semifield.
Consider a group $G$ and its extension $\bar{G}$ defined
as in Example~\ref{groupext}. Let us consider on $\bar{G}$ the additive
law $\oplus$ such that $a\oplus b=\ooi^\circ$ for all 
$a,b\in G\cup\{\ooi^\circ\}$. Then 
$\bar{G}^\circ=\{\zero,\ooi\}$ and $\skewproductstar{\bar{G}}{\rmax}$
is a supertropical semifield.
More generally, if $\RR$ is a totally ordered idempotent semifield, then
$\skewproductstar{\bar{G}}{\RR}$ is a supertropical semifield.
\end{example}

\section{Combinatorial properties of semirings}\label{sec-comb}
We next recall or establish some properties
of a combinatorial nature, 
which will be useful when studying Cramer systems over extended
tropical semirings. 

\subsection{Determinants in semirings with symmetry}

\begin{definition} \label{de:detS}
Let $(\SS,+,\cdot)$  be a semiring with symmetry and $A=(A_{ij})\in\M(\SS)$.
We define the \NEW{determinant} $\detp{A}$ of $A$ to be the element
of $\SS$ given by the usual formula 
$$ \detp{A}=\sum_{\sigma\in \allperm_n} \sgn(\sigma) A_{1\sigma(1)}\cdots A_{n\sigma(n)} ,$$ 
understanding that $\sgn(\sigma)=\pm 1$ depending on the even or odd parity of $\sigma$. 
\end{definition}

In a general semiring, the determinant is not a multiplicative morphism,
however the following identities hold.
\begin{prop}[\protect{\cite{gaubert92a},\cite[Section 4]{AGG08}}]
Let $\SS$ be a semiring with symmetry, then for all
 $A,B\in\M(\SS)$, we have:
\begin{align}\label{e-mult3}
\detp{AB}\balance \detp{A} \detp{B} \enspace ,
\end{align}
or more precisely:
\begin{align}\label{e-mult4}
\detp{AB}\succeq^\circ \detp{A} \detp{B} \enspace .
\end{align}
\end{prop}

\begin{definition} Let $\SS$ be a semiring.
A matrix $C\in \M(\SS)$ is \NEW{monomial} if it can be written as $C=DP^\sigma$
where $D$ is a diagonal matrix ($D_{ij}=0$ for $i\neq j$),
and $P^\sigma$ is the matrix of the permutation $\sigma\in\allperm_n$
(that is $P^\sigma_{ij}=1$ if $j=\sigma(i)$ and $P^\sigma_{ij}=0$ otherwise).
\end{definition}
Note that, in any semiring $\SS$, a permutation matrix is invertible
since $(P^{\sigma})^{-1}=P^{\sigma^{-1}}$. Hence $C$ can be written
as $C=DP^{\sigma}$ if and only if it can be written as $C=P^\sigma D'$,
by taking $D'=(P^{\sigma})^{-1}DP^{\sigma}$.
This also implies that $C=DP^\sigma$ is invertible in $\M(\SS)$
if and only if all diagonal entries of $D$ are invertible in $\SS$.

The following property is easy to check.
\begin{prop}\label{prop-det-mult}
Let $\SS$ be a semiring with symmetry, then for all
 $A,B\in\M(\SS)$, such that $A$ or $B$ is monomial, we have:
\begin{align}\label{det-mul}
\detp{AB}= \detp{A} \detp{B} \enspace .
\end{align}
\end{prop}

We introduce now the concept of adjugate matrices.
\begin{definition}\label{def-adjugate}
Let $\SS$ be a semiring with symmetry.
If $A\in\M(\SS)$, we denote
by $A(i,j)$ the $(n-1)\times (n-1)$ submatrix in which row $i$ and column $j$
are suppressed. Define the \NEW{$ij$-cofactor} of $A$ to
be $\cof_{ij}(A):=  (-1)^{i+j} \detp{A(i,j)}$ and 
the {\em adjugate matrix} of $A$ to be the $n\times n$ matrix $A\adj$ 
with $(i,j)$-entry:
\[
(A\adj)_{ij}:=\cof_{ji}(A)=  (-1)^{i+j}\detp{A(j,i)} \enspace .
\]
\end{definition}

Using Proposition~\ref{prop-det-mult} we also obtain
the following identities.
\begin{lemma} \label{llem-util07}
Let $\SS$ be a semiring with symmetry, then for all
 $A,B\in\M(\SS)$, such that $A$ or $B$ is monomial, we have:
\begin{equation}
(AB)\adj=B\adj 	A\adj \enspace .\label{lem-util07}
\end{equation}
\end{lemma}

\begin{lemma} \label{lhabitude}
Let $\SS$ be a semiring with symmetry,
and let $C\in \MM_n(\SS)$ be an invertible monomial matrix.
Then $\detp{C}$ is invertible and 
\begin{equation}
C^{-1}=(\det C)^{-1}C\adj \enspace .\label{habitude}
\end{equation}
\end{lemma}

\subsection{Diagonal scaling of matrices and Yoeli's theorem}
We next recall some properties concerning
the idempotent semiring $\RR$ arising as an ingredient
of the semiring extension. Although we shall only 
apply these properties when $\RR$ is totally ordered,
we state the properties in their full generality
as far as possible.

Let $(\RR,+,0,\cdot,1)$ be an idempotent semiring equipped with the trivial symmetry
and with the natural order.
Then $-1=1$, so  the determinant of a matrix $A$, $\detp{A}$ coincides
with its permanent, denoted by $\per (A)$.
The adjugate matrix $A\adj$ has a useful
interpretation in terms of maximal weights of paths.
To see this, let us first recall the definition and basic
properties of the {\em Kleene star} of a matrix
in an idempotent semiring. 
Recall that to an $n\times n$ matrix $A$ with entries
in a semiring $(\RR,+,0,\cdot,1)$, one associates a digraph $G(A)$
with nodes $1,\dots,n$ and an arc $i\to j$ if $A_{ij}\neq 0$. The
{\em weight} of a path $p=(i_0,\ldots,i_k)$ in $G(A)$ is defined as 
$\we_A(p):=A_{i_0i_1}\dots A_{i_{k-1}i_k}$, and its length $l(p)$ is equal 
to $k$ (the number of arcs).
This definition applies in particular
to {\em circuits}, which are closed paths, meaning that $i_k=i_0$. 
We denote by $I$ the identity matrix (with diagonal elements equal to $1$ and
non-diagonal elements equal to $0$).

\begin{propdefinition}[{Carr\'e~\cite{carre71}, Gondran~\cite{gondran73}, \cite[p. 72, Th.~1]{gondran79}}]\label{propdef}
Let $(\RR,+,0,\cdot,1)$ be any idempotent semiring equipped
with the natural order.
Let $A\in \MM_n(\RR)$ be a matrix such that every circuit of $A$ has a weight
less than or equal to $1$. Then
\begin{align*}
A^*&:=\sum_{i=0}^\infty A^i = I+ A+ A^2+  \dots+ A^k+ \cdots\\
&= I+A+A^2+\dots +A^{n-1}\enspace .
\end{align*}
\end{propdefinition}
When $\RR$ is idempotent, the sum is equivalent to the supremum 
for the natural order of $\RR$, hence $A^*_{ij}$ represents
the maximal weight of a path from $i$ to $j$.
Moreover, when every circuit of $A$ has a weight less than or equal to $1$,
$A^*_{ij}$ is also equal to the maximal weight of an
elementary path from $i$ to $j$ (when $i=j$, this means
a path with length $0$).

The following result generalizes a theorem of Yoeli~\cite{yoeli61},
which was stated when the idempotent semiring has
a maximal element equal to $1$. 
\begin{theorem}[{Compare with Yoeli~\cite[Theorem 4]{yoeli61}}]\label{th-yoeli}
Let $(\RR,+,0,\cdot,1)$ be an idempotent semiring equipped with the trivial
symmetry and the natural order denoted by $\leq$.
Let $A=(A_{ij})\in \MM_n(\RR)$ be a matrix such that $A_{11}=\dots =A_{nn}=1$ 
and $\per A=1$. Then every circuit of $A$ has a weight
less than or equal to $1$, and 
\[
A\adj  = A^* \enspace .
\]
\end{theorem}
\begin{proof}
Let $A=(A_{ij})\in \MM_n(\RR)$ be a matrix such that $A_{11}=\dots =A_{nn}=1$ 
and $\per A=1$. Let $c=(i_0,\ldots,i_k=i_0)$ be a circuit of $G(A)$.
One can construct the permutation $\sigma\in  \allperm_n$
containing this circuit and all the circuits
with one element not in $c$, that is
$\sigma(i_l)=i_{l+1}$ for $l=0,\ldots, k-1$, and
$\sigma(i)=i$ for $i\not\in\{i_0,\ldots, i_{k-1}\}$.
Since $A_{ii}=1$, the weight of this permutation $A_{1\sigma(1)}\cdots A_{n\sigma(n)}$
is equal to the weight $\we_A(c)$ of the circuit $c$. Since
$\per A=1$ is the sum, thus the supremum for the natural order,
of the weights of all permutations, we get that  $\we_A(c)\leq 1$,
which shows that every circuit of $A$ has a weight
less than or equal to $1$

To prove the last assertion of the theorem, we need to show that
$(A\adj)_{ij}=A^*_{ij}$ for all $i,j\in [n]$.
As remarked after Proposition-Definition~\ref{propdef}, 
since all circuits have a weight less than or equal to $1$,
$A^*_{ij}$ is equal to the maximal weight of an
elementary path from $i$ to $j$.
In particular, $A^*_{ii}=1$, and since $A(i,i)$ has the same properties
as $A$ (the diagonal coefficients are equal to $1$ and all circuits have 
a weight less than or equal to $1$), 
we deduce that $(A\adj)_{ii}=\per (A(i,i))=1=A^*_{ii}$.

Assume now that $i\neq j$. We have
\begin{equation}\label{eq-vision}
(A\adj)_{ij}=\per ( A(j,i)) 
= \sum_{\sigma} \prod_{l\in [n]\setminus \{j\}} A_{l\sigma(l)}\enspace,
\end{equation}
where the sum is taken over all bijections
from $[n]\setminus \{j\}$ to $[n]\setminus \{i\}$.
Since a map $\sigma:[n]\setminus \{j\}\to [n]\setminus \{i\}$
is a bijection  if and only if it can be completed into
a permutation of $[n]$ by taking $\sigma(j)=i$,
the above sum can be taken equivalently over all
$\sigma\in  \allperm_n$ such that $\sigma(j)=i$.

We say that a circuit $c=(i_0,\dots ,i_k=i_0)$ of $G(A)$ 
contains the arc $(i,j)$ if there exists $l=0,\ldots, k-1$ such that 
$i=i_l$, $j=i_{l+1}$. 
Then $p=(i_0,\ldots , i_k)$ is an elementary path from $i_0=i$
to $i_k=j$ if and only if $(p,i)=(i_0,\ldots , i_k,i)$ is 
an elementary circuit, containing the arc $(j,i)$.
Let $p$ be such a path and $c=(p,i)$.
Completing $c$ as above to a permutation $\sigma\in  \allperm_n$,
containing all circuits with one element not in $c$,
we get that $\sigma(j)=i$,
and that $\we_A(p)= \we_A(p)\cdot \prod_{l\not \in c} A_{ll}
=\prod_{l\in [n]\setminus \{j\}} A_{l\sigma(l)}$.
Since~\eqref{eq-vision} holds with a sum over all
$\sigma\in  \allperm_n$ such that $\sigma(j)=i$,
we obtain that $\we_A(p)\leq (A\adj)_{ij}$, and since this
holds for all  elementary paths from $i$ to $j$, we
deduce that $A^*_{ij}\leq (A\adj)_{ij}$.

To show the reverse inequality, let $\sigma \in \allperm_n$
be such that $\sigma(j)=i$.
Decomposing $\sigma$ into elementary cycles,
we get in particular a cycle $c$ containing the arc $(j,i)$.
Let $p$ be the elementary path from $i$ to $j$ such that $c=(p,i)$,
we deduce that 
$\prod_{l\in [n]\setminus \{j\}} A_{l\sigma(l)}=
\we_A(p)\cdot \prod_{c'} \we_A(c')$, where the last
product is taken over all cycles $c'$ of $\sigma$ different from $c$.
Since all cycles have  weights  less than or equal to $1$,
we get that $\prod_{l\in [n]\setminus \{j\}} A_{l\sigma(l)}\leq 
\we_A(p)\leq A^*_{ij}$.
By applying~\eqref{eq-vision} with a sum over all $\sigma\in  \allperm_n$
such that $\sigma(j)=i$, we obtain that $(A\adj)_{ij}\leq  A^*_{ij}$,
and so the equality holds. This finishes the proof of the theorem.
\end{proof}

The following proposition gives a semifield version of a
well known duality result concerning
the optimal assignment problem. It will allow convenient normalizations.
\begin{prop}[Hungarian scaling]\label{prop-hungarian}
Let $C$ be an $n\times n$ matrix with entries in a totally ordered idempotent semifield $\RR$, and assume that $\per C\neq 0$. Then there exist two
$n$-dimensional vectors $u,v$ with entries in $\RR\setminus \{0\}$ such that
\[
C_{ij} \leq u_i v_j, \qquad \forall i,j\in [n]
\]
and $C_{ij} = u_i v_j$ for all $(i,j)$ such that $j=\sigma(i)$ for every optimal
permutation $\sigma$, i.e., for every permutation $\sigma$ such that 
\[
\per C = \prod_{i\in [n]} C_{i\sigma (i)} \enspace .
\]
In particular,
\[
\per C = \prod_{i\in [n]} u_i \prod_{j\in [n]} v_j 
\]
\end{prop}
\begin{proof}
This is a byproduct of the termination of the Hungarian algorithm.
We refer the reader to~\cite{Schrijver03:CO_A} for more information on 
this algorithm. The latter
is usually stated for matrices with entries in the ordered group $(\Re,+)$
completed by the $-\infty$ element, or equivalently
for matrices with entries in the ordered group $(\Re_{+*},\times)$ 
of strictly positive real numbers, completed by the $0$ element. In the latter context, it allows one to compute $\max_{\sigma} \prod_{i} C_ {i\sigma(i)}$.
Let us call a {\em row} (resp.\ {\em column}) {\em scaling} the operation of multiplying by the inverse of a non-zero number a given row (resp.\ column) of a matrix. The Hungarian algorithm
performs a finite number of row and column scalings, reaching eventually
a matrix $B_{ij} = u_i^{-1}C_{ij}v_j^{-1}$ such that $B_{ij} \leq 1$ for all $i,j$
and $B_{ij}= 1$ for all $(i,j)$ in a collection of couples
of indices among which $n$ are 
{\em independent}, meaning that none of them belong to
the same row or column. Then these independent
$(i,j)$ define an optimal permutation, and all the conclusions of the proposition
are valid. The algorithm
can be readily checked to be valid when the entries of $C$ belong to any totally ordered semifield.
\end{proof}
We note that some generalizations of network flow problems to ordered
algebraic structures were studied in~\cite[Chap.~12]{Zimm81}, the previous result could also be derived from results there.

A variant of the following result has appeared
in the work of Butkovic~\cite{butkovip94}. It allows one to reduce
matrices to a ``normal form'' in which the diagonal consists
of unit elements and all other elements are not greater than the unit.
\begin{corollary}[{\cite[Th.~3.1]{butkovip94}}]\label{cor-but}
Let $C$ be an $n\times n$ matrix with entries in a totally ordered idempotent semifield $\RR$, and assume that $\per C\neq 0$. Then there exist two
diagonal matrices $D$ and $D'$ with invertible diagonal entries, and a permutation matrix $\Sigma$, such that $B=\Sigma DCD'$ satisfies
\[
B_{ij}\leq 1,\qquad \forall i,j\in [n],\qquad B_{ii}=1 ,\qquad \forall i\in [n] \enspace .
\]
\end{corollary}

\begin{remark}
Corollary~\ref{cor-but}
shows that in the special case of a totally ordered idempotent semifield,
as soon as $\per C\neq \zero$, we may reduce $C$ by diagonal scaling and permutation to a matrix $A$ satisfying the assumptions of Yoeli's theorem.
\end{remark}

\section{Elimination in semirings and Cramer theorem}\label{sec-elim}

\subsection{Elimination in semirings with symmetry}
In the sequel, we shall consider a semiring $\SS$ with symmetry and 
a thin set $\SS^\vee$ satisfying the following properties, which will allow
us to eliminate variables in order to solve tropical
linear systems.

\begin{pty}
\label{LRRR} 
For $x,y\in \SS^\vee$, we have that $x\balance y$ implies $x=y$. 
\end{pty}

\begin{pty} \label{LRRR1}
The set of non-zero thin elements $(\SS^\vee)^*$
is closed under multiplication. So is $\SS^\vee$, a fortiori.
\end{pty}

\begin{pty}[Weak transitivity of systems of balances]\label{pty-weaktrans}
For all $n,p\geq 1$, $a\in \SS^\vee$, $C\in \MM_{n,p}(\SS)$, 
$b\in \SS^p$, and $d\in\SS^n$, we have
\begin{equation}\label{pty-weaktrans2}
(x\in (\SS^\vee)^p, \; ax\balance b \text{ and } Cx \balance d)\implies Cb\balance ad \enspace .
\end{equation}
\end{pty}

\begin{definition}\label{def-tropextension}
If a semiring $\SS$ with a symmetry and a \thin set $\SS^\vee$
satisfies Properties~\ref{LRRR1} and~\ref{pty-weaktrans}, we will say that 
it \NEW{allows weak balance elimination}.
If it satisfies also Property~\ref{LRRR}, then we will say that 
it \NEW{allows strong balance elimination}.
\end{definition}

It was pointed out in~\cite{Plus} (see also~\cite[Section 6]{AGG08})
that the symmetrized max-plus semiring $\smax$ satisfies 
Properties~\ref{LRRR} -- \ref{pty-weaktrans}.  It was also
observed in~\cite{AGG08} that so does the bi-valued tropical semiring $\Ti$.
Note that in the latter reference, Property~\ref{LRRR1} was replaced by the stronger property 
that the set $(\SS^\vee)^*=\SS\setminus \SS^\circ$ is exactly the 
set of all invertible elements in $\SS$, but this stronger property will
not always be needed.
In~\cite{AGG08} a proof of Property~\ref{pty-weaktrans} was given
specially for $\smax$, we shall give now some sufficient conditions
for the above properties to hold, which allow one to 
check them easily for the semirings $\BB$ and $\N_2$
and deduce them for $\smax$ and $\Ti$.
In particular, we shall also consider the following properties:

\begin{pty}[Weak transitivity of balances]\label{pty-weaktrans0}
For all $b,d\in \SS$, we have 
\[ (x\in \SS^\vee, \; b\balance x \text{ and } x \balance d)\implies b\balance d \enspace . \]
\end{pty}

\begin{pty}[Weak transitivity of scalar balances]\label{pty-weaktrans1}
For all $b,c,d\in \SS$, we have 
\begin{equation}\label{elimin-scalaire}
 (x\in \SS^\vee, \; x\balance b \text{ and } cx \balance d)\implies cb\balance d \enspace ,
\end{equation}
which is~\eqref{pty-weaktrans2} for $n=p=1$ and $a=\oi$.
\end{pty}

\begin{pty}\label{pty-thinexact} 
$(\SS^\vee)^*=\SS\setminus \SS^\circ$.
\end{pty}

\begin{pty} \label{LRRR2}
The set $\SS$ is additively generated by $\SS^\vee$, which means that
any element of $\SS$ is the sum of a finite number of
elements of $\SS^\vee$.
\end{pty}

\begin{lemma}\label{lemma-pty-weaktrans1}
 $\SS$ allows weak balance elimination if and only if
Properties~\ref{LRRR1} and~\ref{pty-weaktrans1} hold together.
\end{lemma}
\begin{proof}
Since Property~\ref{pty-weaktrans} implies in particular 
Property~\ref{pty-weaktrans1}, we get the ``only if'' part of the assertion
of the lemma. For the ``if'' part, let us
assume that Properties~\ref{LRRR1} and~\ref{pty-weaktrans1} hold,
and show that Property~\ref{pty-weaktrans} holds.
Let $n,p\geq 1$, $a\in \SS^\vee$, $C\in \MM_{n,p}(\SS)$, 
$b\in \SS^p$, $d\in\SS^n$, and $x\in (\SS^\vee)^p$ be such that 
$ax\balance b$ and $Cx \balance d$. Let us show that $Cb\balance ad$.
Since,  by Property~\ref{LRRR1}, $\SS^\vee$ is stable 
under product, $a x= (a x_i)_{i\in [p]}\in (\SS^\vee)^p$.
Since $\SS^\circ$ is an ideal of $\SS$, 
multiplying the equation $Cx \balance d$ by $a$, we get that
$Ca x\balance ad$. Then  it remains to show the above
implication for $1$, $ax$ and $ad$ instead of $a$, $x$ and $d$ respectively.
Without loss of generality we can assume that $a=\oi$.
Moreover, since  $Cx \balance d$ is equivalent to $C_{i\cdot} x\balance d_i$
for all $i\in [n]$, where $C_{i\cdot}$ denotes the $i$th row of $C$,
it is sufficient to prove the above implication for each row of $C$ 
instead of $C$. We can thus assume that $n=1$.

Let $C=(c_1,\ldots, c_p)\in \MM_{1,p}(\SS)$, and
assume that $x\balance b$ and $Cx \balance d$.
The relation $Cx\balance d$ is equivalent to $c_1 \cdot x_1\balance
d-c_2\cdot x_2\cdots - c_p\cdot x_p$.
From~\eqref{elimin-scalaire}, which holds by
Property~\ref{pty-weaktrans1}, and $x_1\balance b_1$, we deduce that 
 $c_1 \cdot b_1\balance d-c_2\cdot x_2\cdots - c_p\cdot x_p$.
Now exchanging the sides of $c_1 \cdot b_1$ and $c_2\cdot x_2$,
and applying~\eqref{elimin-scalaire} with $x_2\balance b_2$,
we can replace $x_2$ by $b_2$ in the previous balance equation.
Doing this inductively on all $x_i$, we obtain $Cb\balance d$,
which concludes the proof.
\end{proof}

\begin{lemma}\label{lemma-pty-weaktrans2}
Properties~\ref{LRRR} and~\ref{pty-thinexact} together imply
 Property~\ref{pty-weaktrans0}.
\end{lemma}
\begin{proof}
Assume Property~\ref{LRRR} holds.
Let $b,d\in \SS$, $x\in \SS^\vee$ be such that $b\balance x$ and $x \balance d$.
We need to show that $b\balance d$.
Since $x\in \SS^\vee$, if $b\in\SS^\vee$ then Property~\ref{LRRR} implies
that $b=x$, so that $b\balance d$. Similarly, $b\balance d$ if $d\in\SS^\vee$.
Otherwise, $b$ and $d\not\in \SS^\vee$, which implies that $b$ and
$d\in \SS^\circ$ by Property~\ref{pty-thinexact}.
Hence $b-d\in\SS^\circ$, which means that $b\balance d$.
\end{proof}

\begin{lemma}\label{lemma-balance}
Properties~\ref{pty-weaktrans0},~\ref{LRRR1},
and~\ref{LRRR2} all together imply Property~\ref{pty-weaktrans1},
hence, Property~\ref{pty-weaktrans} and 
that $\SS$ allows weak balance elimination.
\end{lemma}
\begin{proof} Assume Properties~\ref{pty-weaktrans0},~\ref{LRRR1}, 
and~\ref{LRRR2} hold. Then~\eqref{pty-weaktrans2}
holds when $n=p=1$ and $a=C=1$.
Let us show that~\eqref{pty-weaktrans2}
also holds when $n=p=1$, $a=1$,  and $C\in\SS$.
This will mean that Property~\ref{pty-weaktrans1} holds and 
will imply by Lemma~\ref{lemma-pty-weaktrans1} that
Property~\ref{pty-weaktrans} holds.
So let $C,b,d \in \SS$, and $x\in \SS^\vee$ be such that 
$x\balance b$ and $Cx \balance d$, and let us show that $Cb\balance d$.
If $C\in\SS^\vee$, then by Property~\ref{LRRR1}, $Cx \in\SS^\vee$.
Multiplying $x\balance b$ by $C$, we get that $Cx \balance C b$,
and since $Cx \balance d$ and $Cx \in \SS^\vee$, 
 Property~\ref{pty-weaktrans0} implies $Cb\balance d$.
This shows that~\eqref{pty-weaktrans2}
holds when $n=p=1$, $a=1$,  and $C\in\SS^\vee$.
Assume now that $C\not \in\SS^\vee$. Then by Property~\ref{LRRR2},
there exist $E_1,\ldots, E_k\in \SS^\vee$ such that
$C=E_1+\cdots + E_k$. 
The relation $Cx\balance d$
is then equivalent to $E_1\cdot x\balance -E_2\cdot x\cdots - E_k\cdot x+d$.
Since $E_1\in\SS^\vee$, $x\in\SS^\vee$ and $x\balance b$, 
applying~\eqref{pty-weaktrans2} with $n=p=1$, $a=1$,
$E_1\in\SS^\vee$ instead of $C$ and $-E_2\cdot x\cdots - E_k\cdot x+d$
 instead of $d$ (the implication~\eqref{pty-weaktrans2} is already known to hold in that case), we get that 
$E_1\cdot b\balance -E_2\cdot x\cdots - E_k\cdot x+d$.
Now exchanging the sides of $E_1\cdot b$ and $E_2\cdot x$,
and applying~\eqref{pty-weaktrans2} with $n=p=1$, $a=1$,
$E_2\in\SS^\vee$ instead of $C$, we can replace $E_2\cdot x $ by $E_2\cdot b$
in the previous balance equation.
Doing this inductively, we obtain $Cb\balance d$,
hence~\eqref{pty-weaktrans2}
holds when $n=p=1$, $a=1$,  and for all $C\in\SS$.
\end{proof}

\begin{corollary}\label{cor-balance}
Properties~\ref{LRRR},~\ref{LRRR1},
~\ref{LRRR2}, together with either Property~\ref{pty-weaktrans0}
or~\ref{pty-thinexact}, imply that 
$\SS$ allows strong balance elimination.
\end{corollary}

\begin{prop}\label{extension-balance}
Let $(\SS,+,\cdot)$ be a semiring with symmetry,
let $\SS^\vee$ be a \thin set of $\SS$, and let
$(\RR,\oplus,\ooo,\cdot,\ooi)$ be a totally ordered 
idempotent semiring.

Denote by $\Se$ any of the semiring extensions $\skewproductbar{\SS}{\RR}$, 
$\skewproduct{\SS}{\RR}$ or $\skewproductstar{\SS}{\RR}$ defined in Proposition \ref{extension}
endowed with the extension of the symmetry of $\SS$
defined in Fact~\ref{remark-extension}. 
Here, $\Se$ is assumed to be a subsemiring of $\skewproductbar{\SS}{\RR}$.

Consider the \thin set $\Se^\vee=\skewproductstar{\SS^\vee}{\RR}$ 
if $\Se=\skewproduct{\SS}{\RR}$ or $\Se=\skewproductstar{\SS}{\RR}$, 
and $\Se^\vee=((\SS^\vee)^*\times \RR)\cup\{\zero\}$ 
if $\Se=\skewproductbar{\SS}{\RR}$.
Then the following properties hold:
\begin{itemize}
\item[(a)] $\Se$ satisfies Property~\ref{LRRR} (resp.\ \ref{LRRR1},
resp.\ \ref{pty-weaktrans1})
if and only if $\SS$ does.

\item[(b)]  $\Se$ allows weak balance elimination if and only if $\SS$ does.
\item[(c)]  $\Se$ allows strong balance elimination if and only if $\SS$ does.
\end{itemize}
\end{prop}
\begin{proof} 
Assertions (b) and (c) follow from Assertion (a), 
Lemma~\ref{lemma-pty-weaktrans1}, and the definition of strong
balance elimination.

1. Let us first prove the ``only if'' part of Assertion (a) 
 of the proposition.
Assume that $\SS$, $\SS^\vee$, $\Se$, $\Se^\vee$ are as
in the statement.
 By Fact~\ref{remark-extension0}, $\SS$ is isomorphic by $\jmath$ to
a subsemiring of $\Se$. 
Moreover, by definition, the map $\jmath$ is compatible with 
the symmetries of $\SS$ and $\Se$, hence 
$\jmath(\SS^\circ)=\Se^\circ\cap\jmath(\SS)$ then on $\SS$
the balance relation of $\SS$ coincides with the one of $\Se$.
Also, by definition of $\Se^\vee$,
we have $\Se^\vee\cap \jmath(\SS)=\jmath(\SS^\vee)$,
and $(\Se^\vee)^*\cap \jmath(\SS)=
\jmath((\SS^\vee)^*)$.
From this, we deduce that Property~\ref{LRRR} (resp.\ \ref{LRRR1},
resp.\ \ref{pty-weaktrans1})
for $\Se$ implies the same for $\SS$.

2. Let us now show the ``if'' part of the Assertion (a).
By definition, 
$(\Se^\vee)^*= (\SS^\vee)^*\times
\RR^*$ or $(\SS^\vee)^*\times \RR$.
Since $\RR^*$ and $\RR$ are closed by multiplication, it is clear that
Property~\ref{LRRR1} for $\SS$ implies that the same property is valid
for $\Se$.

Assume now that Property~\ref{LRRR} holds for $\SS$ and let us prove it for $\Se$.
Remark that for any semiring $\SS$ with symmetry,
the assertion of Property~\ref{LRRR} is equivalent to
the same assertion with $\SS^\vee$ replaced by 
$(\SS^\vee)^*$.
Indeed, when $x,y\in\SS^\vee$ with $x=\oo$ or $y=\oo$, the equation 
$x\balance y$ implies that $x,y\in\SS^\circ$. Thus
$x,y\in \SS^\circ\cap \SS^\vee=\{\oo\}$. Hence $x=y=\oo$.
Then it is sufficient to prove the assertion
of  Property~\ref{LRRR} for $(\Se)^*$ instead of
$\Se$.
Let $x=(a,b)$ and $y=(a',b')\in (\Se^\vee)^*$
 be such that $x\balance y$.
Since the equation $x\balance y$ is equivalent to $x-y\in \Se^\circ$
and since $x-y= x$ if $b>b'$,  $x-y=-y$ if $b<b'$, and
$x-y=(a-a',b)$ if $b=b'$, 
we deduce that $b=b'$ and $(a-a',b)\in\Se^\circ$.
Since $\Se^\circ=\Se\cap (\skewproductbar{\SS^\circ}{\RR})$, we deduce that $a\balance a'$.
Since $(\Se^\vee)^*\subset
(\SS^\vee)^*\times \RR$,
we also get that $a,a'\in (\SS^\vee)^*$.
Hence if Property~\ref{LRRR} holds for $\SS$, the above properties 
imply $a=a'$, so $x=y$,
which shows that Property~\ref{LRRR} also holds for $\Se$.

Assume now that Property~\ref{pty-weaktrans1} holds for $\SS$ and
let us prove it for $\Se$.
Let $x=(a,b)\in \Se^\vee$, and
$y=(a',b'), z=(a'',b''), w=(a''',b''')\in \Se$
 be such that $x\balance y$ and $wx\balance z$.
We need to show that $wy\balance z$.
It is easy to show that this holds when $x=\zero$,
since then $y,z\in\Se^\circ$ so that $wy-z\in\Se^\circ$.
Also this holds when $y=\zero$, since then $x\in\Se^\circ$ 
and thus again $x=\zero$ (since $x\in\Se^\vee$),
and when $w=\zero$ since then $wx=wy$.
So we can assume that $x\in (\Se^\vee)^*$, and
$y, w\in (\Se)^*$.
Recall that $(\Se^\vee)^*=(\SS^\vee)^*\times \RR$ when
$\Se=\skewproductbar{\SS}{\RR}$
and $(\Se^\vee)^*=(\SS^\vee)^*\times \RR^*$ otherwise.
It follows that $a\in(\SS^\vee)^*$. Moreover,
we have $b\in\RR^*$ when $\Se=\skewproduct{\SS}{\RR}$
or $\skewproductstar{\SS}{\RR}$. 
Now $w=(a''',\ooi)\odot (1,b''')$.
Since $w\neq \zero$,
and $(\Se)^*\subset \SS\times \RR^*$,
we get that $b'''\neq \zero$
when $\Se=\skewproduct{\SS}{\RR}$
or $\skewproductstar{\SS}{\RR}$. 
Then in all cases $x':=(1,b''')\odot x=(a,b'''\cdot b)\in (\Se^\vee)^*$.
Multiplying both terms of the relation $x\balance y$ by $(1,b''')$,
and replacing $x$ by $x'$, $y$ by $y'=(1,b''')\odot y$ and
$w$ by $w':=(a''',\ooi)$, we are reduced to the case where
$b'''=\ooi$.
Then by the same arguments as above we obtain from $x\balance y$ that 
either $b<b'$ and  $y-x=y\in \Se^\circ $, or 
$b'=b$ and $a'\balance a$.
Similarly, from $wx\balance z$, we obtain that
either $b<b''$ and $z\in \Se^\circ$, or 
$b=b''$ and $a'''a\balance a''$,
or $b>b''$ and $wx\in \Se^\circ$ so that $a'''a\in\SS^\circ$.

When $b<b'$ and $b<b''$, we get that $y,z\in \Se^\circ $,
so $wy-z\in \Se^\circ $ and $wy\balance z$.
When $b<b'$ and $b\geq b''$, we get that $wy-z=wy\in \Se^\circ $,
and again $wy\balance z$.
When $b=b'$ and $b<b''$, we get that $wy-z=-z\in \Se^\circ $,
and again $wy\balance z$.
When $b=b'$ and $b>b''$, we get that $wy-z=wy=(a'''a',b)$.
Since $a\in(\SS^\vee)^*$, $a'\balance a$ and $a'''a\balance 0$,
we deduce from Property~\ref{pty-weaktrans1} for $\SS$, 
that $a'''a'\balance 0$, hence $wy-z\in\Se^\circ$ again.
When $b=b'=b''$, we get that $a'\balance a$ and
$a'''a\balance a''$. Since we also have $a\in  (\SS^\vee)^*$, 
we deduce from Property~\ref{pty-weaktrans0} for $\SS$, that
$a'''a'\balance a''$. This implies again  $wy\balance z$.
Since in all cases $wy\balance z$, this shows 
Property~\ref{pty-weaktrans1} for $\Se$.

\end{proof}

\begin{remark}
Note that one can also prove similar equivalences to the ones of item (a) 
of Proposition~\ref{extension-balance} for Properties 
\ref{pty-weaktrans0} and~\ref{pty-thinexact}.
Also, one can do the same for Property~\ref{LRRR2} if 
$\Se=\skewproductstar{\SS}{\RR}$.
We do not detail these equivalences here, since we are interested 
mostly in 
``balance elimination''.
\end{remark}

We are now able to give examples of semirings allowing balance elimination,
as a consequence of Proposition~\ref{extension-balance},
Lemmas~\ref{lemma-pty-weaktrans2}
and~\ref{lemma-balance}, and Corollary~\ref{cor-balance}.

\begin{fact} The semirings $\SS=\BB$ and $\SS=\N_2$
satisfy Properties~\ref{pty-thinexact},~\ref{LRRR},~\ref{LRRR1} and~\ref{LRRR2}.
Then these semirings allow
strong balance elimination. Hence so do
$\S=\skewproductstar{\BB}{\rmax}$  and $\Ti=\skewproductstar{\N_2}{\rmax}$.
The semiring $\SS=\phase$
of Example~\ref{ex-viro}, i.e., the set of closed convex cones of $\C$,
does not satisfy Property~\ref{pty-thinexact}, but it does
satisfy Properties~\ref{pty-weaktrans0},~\ref{LRRR},~\ref{LRRR1},
 and~\ref{LRRR2}. So, $\phase$ allows strong balance
elimination, and so does $\skewproductstar{\phase}{\rmax}$ (the phase
extension of the tropical semiring).
\end{fact}
Note that for the same semirings $\SS$, 
the semirings $\skewproductbar{\SS}{\rmax}$ and $\skewproduct{\SS}{\rmax}$ 
also allow strong balance elimination.

\begin{fact}
Let $\SS$ be an integral domain (a ring without zero divisors).
Since $\SS^\circ=\{\oo\}$, the balance relation
reduces to the equality relation. Taking
$\SS^\vee=\SS$, we get that $\SS$ allows
trivially strong balance elimination. 
Hence, the semirings $\skewproductbar{\SS}{\RR}$ and $\skewproduct{\SS}{\RR}$
also allow strong balance elimination.
However $\skewproductstar{\SS}{\RR}$ is not a semiring, since $\SS$ is not zero sum free.
\end{fact}
Recall that when $\SS=\C$ or $\SS=\Re$, an element $(a,b)\in \skewproduct{\SS}{\RR}$
is equivalent to the asymptotic expansion $a \epsilon^{-b}+o(\epsilon^{-b})$,
when $\epsilon$ goes to $0_+$ (Example~\ref{ex-best}).

\begin{fact}
Let $\SS$ be a supertropical semifield (see Remark~\ref{rem-supertropical}).
Then Properties~\ref{pty-weaktrans0},~\ref{LRRR1} and~\ref{LRRR2} hold,
thus $\SS$ allows weak balance elimination, by Lemma~\ref{lemma-balance}.
However, if the map $\mu$ from $\SS^\vee$ to $\SS^\circ$  is not injective, then
Property~\ref{LRRR} does not hold, thus $\SS$ does not
allow strong balance elimination.
\end{fact}
\begin{proof} The first assertion can be checked easily.
For the second one, if $\mu$ is not injective, 
there exist $a,b\in\SS^\vee$, $a\neq b$ such that $|a|=|b|$.
Since $a-b=a+b=|a|\in\SS^\circ$, we get that $a\balance b$,
thus Property~\ref{LRRR} does not hold.
\end{proof}

\subsection{Cramer formul\ae\ in semirings allowing balance elimination}
\label{sec-cramer}

We state here a general Cramer theorem. In the special
case of the symmetrized max-plus semiring, this was established by Plus~\cite{Plus}, see also \cite{gaubert92a} and \cite[Theorem 6.4]{AGG08}.
It is remarked
in the latter reference (see the paragraph before Theorem 6.6 of 
\cite{AGG08}) 
that the proof of \cite[Theorem 6.4]{AGG08} 
is valid in any semiring satisfying
Properties~\ref{LRRR},~\ref{pty-weaktrans}, 
and the property that $(\SS^\vee)^*=\SS\setminus \SS^\circ$
is the set of invertible elements.
However looking at the proof of \cite[Theorem 6.4]{AGG08} 
more carefully, one see that the latter property can be replaced
by Property~\ref{LRRR1}, and that 
the first part does not use Property~\ref{LRRR}.
This leads to the following general result.

\begin{theorem}[{Cramer theorem, compare with~\cite[Theorem 6.1]{Plus},~\cite[Chap.\ III, Theorem 3.2.1 and Proposition 3.4.1]{gaubert92a} and~{\cite[Theorem 6.4]{AGG08}} for $\S$ and~{\cite[Theorem 6.6]{AGG08}} for $\Ti$}]\label{th-cramer} 
Let $\SS$ be a semiring with a symmetry and a \thin set $\SS^\vee$,
allowing weak balance elimination
(Definition~\ref{def-tropextension}).
Let $A\in {\MM}_n(\SS)$ and $b\in \SS^n$, then
\begin{enumerate}
\item \label{cramer1} Every \thin solution $x$ (such
that $x\in (\SS^\vee)^n$) of the linear system
\[
Ax\balance b 
\]
satisfies the relation
$$ (\det A) x \balance A\adj  b \enspace .$$
\item  
Assume also that $\SS$ allows strong balance elimination
(Definition~\ref{def-tropextension}),
that the vector $A\adj  b$ is \thin and that $\det A$ is invertible in 
$\SS$. Then 
$$\hat x:={(\det A)}^{-1} A\adj b$$
is the unique \thin solution of $Ax\balance b$.
\end{enumerate}
\end{theorem}
\begin{proof}
The proof follows exactly the same lines as the proof of
 \cite[Theorem 6.4]{AGG08}, so we do not reproduce it.
Let us just remark that it relies on an elimination argument, 
in which ``equations'' 
involving balances rather than equalities are considered.
\end{proof}

The above uniqueness part can be reformulated 
equivalently in the following homogeneous form.
For a matrix $A\in \MM_{n,m}(\SS)$ over a semiring $\SS$,
and $k\in[m]$, we denote by $A_{|k)}$ the $n\times (m-1)$ matrix obtained from
$A$ by deleting the $k$th column, and by $A_{\cdot  k}$ the $k$th column of $A$.

\begin{corollary}\label{cor-cramer}
Let $\SS$ be a semiring allowing strong balance elimination.
Let  $A\in \MM_{n,n+1}(\SS)$, and let $\hat{x}\in\SS^{n+1}$
 be such that $\hat{x}_k=(\ominus \unit)^{n-k+1}\det A_{|k)}$,
for all $k\in [n+1]$.
Then if $\hat{x}$ is \thin and has at least one
invertible entry (which is the case when $\hat{x}$ is non-zero and
$(\SS^\vee)^*$ is exactly the set of
invertible elements of $\SS$), then
any \thin  
 solution of $Ax\bal \zero$ 
is a \thin multiple of $\hat{x}$.
\end{corollary}
\begin{proof}
Let $\hat{x}$ be as in the statement of the theorem and 
assume that $\hat{x}$ is \thin and has at least one
invertible entry, for instance $\hat{x}_{n+1}$ is invertible.
Let $x$ be a \thin solution of $A x\bal \zero$.
Then taking $M:=A_{|n+1)}$ and $b=\ominus x_{n+1} A_{\cdot  n+1}$,
we get that $\det M=\hat{x}_{n+1}$ is invertible and
$(M\adj b)_i= x_{n+1} \hat{x}_i$ is \thin for all $i\in [n]$.
Let $y\in \SS^n$ be such that $y_i=x_i$ for $i\in [n]$.
Then $y$ is a \thin solution of $My\bal b$,
and applying the uniqueness part of Theorem~\ref{th-cramer} to this system,
we get that $x_i=y_i=(\det M)^{-1} (M\adj b)_i=
(\hat{x}_{n+1})^{-1}x_{n+1}\hat{x}_i$, for all  $i \in [n]$.
Hence $x=\lambda \hat{x}$,
with $\lambda=(\hat{x}_{n+1})^{-1}x_{n+1}\in \SS^\vee$.
We can show that the same conclusion holds for any $i\in [n+1]$,
such that $\hat{x}_{i}$ is invertible, which implies the corollary.
\end{proof}

\section{Existence of solutions of tropical linear systems}\label{sec-exists}
In  this section we consider a semiring $(\TT,\oplus,\zero,\odot,\unit)$
with a symmetry, a \thin set $\TT^\vee$, and a modulus taking its values
in a totally ordered semiring $\RR$. For instance $\TT=\skewproductstar{\SS}{\RR}$
where $\SS$ is a zero-sum free semiring with a symmetry
and without zero divisors, and $\TT^\vee=\skewproductstar{\SS^\vee}{\RR}$.
We shall study the square affine systems 
\[
Ax\balance b 
\]
where $A$ is an $n\times n$ matrix and $b$ is a vector of dimension $n$,
all with entries in $\TT$.
We shall look for the solutions $x\in \TT^n$ with {\em \thinp} entries.

\subsection{Monotone algorithms in semirings with symmetry}
The existence results that we shall state in the next sections
extend the ones proved in~\cite{Plus} for $\smax$.
As the latter results,
they will be derived as a byproduct of the convergence of 
an iterative Jacobi-type algorithm to solve the system $Ax\balance b$.
Recall that the usual Jacobi algorithm constructs a sequence
which is known to converge to the solution of a linear system
under a strict diagonal dominance property.

Here we shall first transform the initial system 
to meet a diagonal dominance property. 
Then we shall construct a monotone sequence of thin vectors
satisfying balance relations which are similar to the
equations used in the definition of
the usual Jacobi algorithm.
To show that these thin vectors do exist
and that the resulting sequence does converge, 
we shall however need some new definitions
and properties concerning the semiring $\TT$, which are somehow
more technical than the properties used to establish the
Cramer theorem, Theorem~\ref{th-cramer}.

\begin{definition}
We define on $\TT$ the relation $\bala$, which is finer
than the balance relation $\bal$:\index{$\bala$}
\[ x\bala y\quad\Leftrightarrow \quad x\bal y\; \text{and }
|x|=|y| \enspace .\]
\end{definition}

Note that when $\TT$ is a semiring extension as in
Proposition \ref{extension} and $\gamma$ is as in
Fact~\ref{remark-extension0}, we have:
\begin{equation}\label{bala}
 x\bala y\quad\Leftrightarrow\quad |x|=|y|\; \text{and }
\gamma(x)\bal \gamma(y) \; (\text{in } \SS)\enspace .
\end{equation}

\begin{pty} \label{pty_order} 
$\TT$ is naturally ordered and for all $x,y\in \TT$ we have:
\[ (x\in \TT^{\vee}\text{ and } x\preceq y)
\implies\;\exists z\in \TT^{\vee} \text{ such that }
x\preceq z\preceq y \text{ and } z \bala y\enspace.  \]
\end{pty}

\begin{pty}\label{pty_order-finite} $\TT$ is naturally ordered
and for all $a\in \RR$ any totally ordered subset of
$\TT^\vee_a=\set{x\in \TT^\vee}{|x|=a}$ is finite.
\end{pty}

The following property is stronger.
\begin{pty}\label{pty_order-equal} $\TT$ is naturally ordered
and for all $x,y\in \TT$ we have:
\[ (x,y\in \TT^{\vee},\;  x \preceq  y \text{ and } |x|=|y|) \implies x=y
\enspace. \]
\end{pty}

\begin{pty} \label{pty_inv} For all $d\in \TT^\vee$ such that
$|d|$ is invertible in $\RR$, 
there exists $\tilde{d}\in \TT$ such that
for all $x,y\in\TT$,
\[ dx \bala y \Longleftrightarrow x\bala \tilde{d} y\enspace .  \]
\end{pty}

\begin{definition}\label{monotone-def}
Let $\TT$  be a semiring with a symmetry, a \thin set $\TT^\vee$,
and a modulus taking its values in a totally ordered semiring $\RR$.
We shall say that $\TT$ \NEW{allows \constr}
if it satisfies Properties~\ref{pty_order} and~\ref{pty_inv}.
If it satisfies in addition Property~\ref{pty_order-finite},
we shall say that it \NEW{allows \conv}.
\end{definition}

From~\eqref{bala} and the property that $\SS$ can be seen as a subsemiring of
$\skewproductstar{\SS}{\RR}$, we obtain easily the following results.
\begin{lemma} \label{l_balance} 
Let $\SS$ be a naturally ordered semiring with a symmetry
and  without zero-divisors.
Then $\skewproductstar{\SS}{\RR}$ satisfies Property~\ref{pty_order} if and only if
the following holds for all $a,a'\in \SS$
\[ (a\in \SS^{\vee}\text{ and } a\preceq a')
\implies \; \exists a''\in \SS^{\vee} \text{ such that }
a\preceq a''\preceq a'\text{ and } a'' \bal a'\enspace.\] 
Moreover, $\skewproductstar{\SS}{\RR}$ satisfies Property~\ref{pty_order-equal}
if and only if the following holds for all $a,a'\in \SS$
\[  (a,a'\in(\SS^{\vee})^* \text{ and } a \preceq a' ) \implies a= a'
\enspace. \]
Also $\skewproductstar{\SS}{\RR}$ satisfies Property~\ref{pty_order-finite}
if and only if any totally ordered subset of $\SS^\vee$ is finite. \qed
\end{lemma}

\begin{lemma} \label{l3.6}
Let $\SS$ be as in Lemma~\ref{l_balance}.
Then  $\skewproductstar{\SS}{\RR}$ satisfies Property~\ref{pty_inv} if and only if
for all $d\in (\SS^\vee)^*$
there exists $\tilde{d}\in \SS$ such that
for all $a,a'\in\SS$
\[\hspace{3cm} da \bal a' \Longleftrightarrow a\bal \tilde{d} a'\enspace . 
\hspace{5.75cm}\qed\]
\end{lemma}

The conditions of Lemma~\ref{l_balance}  are
easily satisfied when $\SS=\BB$ and $\SS=\N_2$.
For instance, for the first condition of Lemma~\ref{l_balance} 
one can take $a''=a$ if $a\neq \zero$,  
$a''=a'$ if $a=\zero$ and $a'\in \SS^{\vee}$, and any 
$a''\in (\SS^{\vee})^*$  otherwise.
Then by Lemma~\ref{l_balance}, $\TT=\S$ and $\TT=\Ti$
satisfy Properties~\ref{pty_order}~--~\ref{pty_order-equal}.
The same holds for $\phase$ and $\skewproductstar{\phase}{\rmax}$ as in Example~\ref{ex-viro}
(the phase extension of the tropical semiring).

The condition of Lemma~\ref{l3.6} holds as soon as $(\SS^\vee)^*$
is the set of invertible elements of $\SS$, since then
$\tilde{d}=d^{-1}$ is a solution ($\SS^{\circ}$ is an ideal).
This is the case for $\SS=\BB$,  $\SS=\N_2$, and for the semiring
$\phase$ of Example~\ref{ex-viro}.
Then by Lemma~\ref{l3.6} the semirings $\TT=\S$, $\TT=\Ti$, and $\skewproductstar{\phase}{\rmax}$ 
satisfy Property~\ref{pty_inv}.
Note that similarly Property~\ref{pty_inv} holds as
soon as $(\TT^\vee)^*$, or at least the
set 
of elements $d$ of $\TT^\vee$ such that $|d|$ 
is invertible in $\RR$,
is the set of invertible elements of $\TT$.

The above properties also hold for each of the extensions
$\bar{G}$ of a group $G$  defined in
Examples~\ref{groupext} and \ref{ex-supertropical},
(with a nontrivial symmetry in the first case and the identity symmetry in
the second one). Thus for each of these examples,
$\skewproductstar{\bar{G}}{\rmax}$ satisfies Properties~\ref{pty_order}--\ref{pty_inv}.
This is in particular the case for the tropical extension
of the torus $\skewproductstar{\bar{\To}}{\rmax}$
of Example~\ref{groupext}  and the supertropical semifield of
Example~\ref{ex-supertropical}. We can also prove directly 
that Properties~\ref{pty_order}--\ref{pty_inv} hold
for any supertropical semifield.

These examples can be summarized as follows.

\begin{fact}
All the following semirings allow the convergence of monotone algorithms:
the bi-valued tropical semiring $\Ti$, 
the symmetrized max-plus semiring $\S$, 
the phase extension of the tropical semiring
$\skewproductstar{\phase}{\rmax}$ (Example~\ref{ex-viro}),
the tropical extension of the torus $\skewproductstar{\bar{\To}}{\rmax}$
or that
of any group with a non trivial symmetry $\skewproductstar{\bar{G}}{\rmax}$
(Example~\ref{groupext}), and any supertropical semifield
(see Remark~\ref{rem-supertropical}).
\end{fact}

We also have:
\begin{prop}\label{prop-equal}
Let $\TT$ 
satisfy
Properties~\ref{pty_order}  and~\ref{pty_order-equal}.
Then  for all $x,y\in \TT$  we have:
\[ (y\in \TT^{\vee},\;  x \preceq  y \text{ and } |x|=|y|) \implies x=y
\enspace. \]
\end{prop}
\begin{proof}
Let $x,y\in \TT$ be such that $y\in \TT^{\vee}$,
$x \preceq  y$  and $|x|=|y|$.
Then by Property~\ref{pty_order} applied to $\zero$ and $x$,
there exists $x'\in\TT^{\vee}$, such that $x'\preceq x$, 
$x\balance x'$ and $|x'|=|x|$.
Then $x'\preceq y$ and $|x'|=|y|$. By Property~\ref{pty_order-equal}
we get that $x'=y$. Since $x'\preceq x\preceq y$, we deduce that $x=y$.
\end{proof}

\subsection{Existence theorems}
The following result 
shows that the existence part of Theorem~\ref{th-cramer} does not require
the condition that all the Cramer determinants are \thinp. 

\begin{theorem}[{Compare with~\cite[Th.~6.2]{Plus}}]\label{th-jac}
Let $\TT$ be a semiring allowing \conv, see Definition~\ref{monotone-def},
let $A\in \M(\TT)$, and assume that $|\det A|$ is invertible in $\RR$
(but possibly $\det{A}\balance \szero$). 
Then for every $b\in \TT^n$ there exists a  \thin
solution $x$ of $Ax\bal b$, which can be chosen in such a way that
$|x|=|\det A|^{-1}|A\adj b|$.
\end{theorem}
This result will be proved in Section~\ref{subsec-stationn}
as a corollary of Theorem~\ref{theo-jacobi} below,
which builds the solution using a Jacobi-type algorithm.

Applying Theorem~\ref{th-jac} to the tropical extension $\S$
(in which $|x|$ is invertible if and only if $|x|\neq \szero$ 
or equivalently $x\neq \zero$), we 
recover the statement of~\cite[Th.~6.2]{Plus}.
\begin{corollary}[{\cite[Th.~6.2]{Plus}}]\label{th-jac-smax}
Let $A\in \M(\S)$. Assume that $\det A\neq \szero$ (but possibly $\det{A}\balance \szero$). Then for every $b\in \S^n$ there exists a  \thin
solution $x$ of $Ax\bal b$, which can be chosen in such a way that
$|x|=|\det A|^{-1}|A\adj b|$.\qed
\end{corollary}

A sketch of the proof of the latter result appeared in~\cite{Plus}; the complete proof appeared in~\cite{gaubert92a}.
The proof of Theorem~\ref{th-jac} that we next give generalizes
the former proof to the present setting.
We also derive as a corollary the following
analogous result in the bi-valued tropical semiring $\Ti$.
\begin{corollary}\label{new-th}
Let $A\in \M(\Ti)$, and assume that $\det A\neq \zero$ (but $\det{A}$ may be balanced).
Then for every $b\in \Ti^n$ the \thin 
vector (``Cramer solution'')
\[
x:=\imath(|\det A|^{-1}|A\adj b|)
\]
satisfies $Ax\bal b$.
\end{corollary}
\begin{proof}
An element of $\Ti$ is such that $|x|$ is invertible if and only if $|x|\neq \szero$ 
or equivalently $x\neq \zero$.
It is easy to see that an element $x\in \Ti$ is \thin 
if and only if $x=\imath(|x|)$.
Hence, the only possible \thin vector $x$ such that
$|x|=|\det A|^{-1}|A\adj b|$ is given by $x=\imath(|\det A|^{-1}|A\adj b|)$.
Applying Theorem~\ref{th-jac} to the tropical extension $\Ti$
we get the corollary.
\end{proof}

The previous results can be reformulated equivalently in the 
following homogeneous forms.

\begin{theorem}\label{th-jac-homogeneous}
Let $\TT$ be a semiring allowing \conv, see Definition~\ref{monotone-def}, and $A\in \MM_{n,n+1}(\TT)$.
Let $\hat{x}\in\TT^{n+1}$ be such that $\hat{x}_k=
(\ominus \unit)^{n-k+1}\det A_{|k)}$
for all $k\in [n+1]$.
Assume that either $\hat{x}=\zero$, or 
at least
one entry of $|\hat{x}|$ is invertible in $\RR$. 
Then there exists a \thin solution $x$ of
$Ax\bal \zero$ such that $|x|=|\hat{x}|$.
\end{theorem}
\begin{proof}
Let $\hat{x}$ be as in the statement of the theorem and 
let us show that there exists $x\in (\TT^\vee)^{n+1}$, such that 
$Ax\bal \zero$ and $|x|=|\hat{x}|$.
If all the entries of $\hat{x}$ are $\zero$,
then $x=\zero$ satisfies trivially the above conditions.
Hence, assume without loss of generality that $|\hat{x}_{n+1}|$ is
invertible in $\RR$.
We set $M:=A_{|n+1)}$ and $b=\ominus A_{\cdot  n+1}$.
By applying Theorem~\ref{th-jac} to the system $My\bal  b$, we get
a \thin solution $y\in\TT^n$ of $My\bal b$ such that
$|y_i| = |\det M|^{-1}|(M\adj b)_i|
= |\hat{x}_{n+1}|^{-1} |\hat{x}_i|$ for all $i\in [n]$.
Let $x\in\TT^{n+1}$ be such that 
$x_{n+1}=\imath(|\hat{x}_{n+1}|)$ and
$x_i=x_{n+1}\odot y_i$ for $i\in [n]$.
Then $|x_i|=|\hat{x}_i|$ for all $i\in [n+1]$ and
multiplying the equation $My\bal b$ by $x_{n+1}$, we get that
$A x\bal \zero$.
\end{proof}

Note that when $\RR$ is a semifield, $\hat{x}$ always satisfies the
condition of Theorem~\ref{th-jac-homogeneous}.
The uniqueness of the solution is obtained from
Corollary~\ref{cor-cramer}.

In the particular case of the
bi-valued tropical semiring $\Ti$ we obtain 
more precise result.

\begin{corollary}\label{th-new-homogeneous}
Let $A\in \MM_{n,n+1}(\Ti)$. 
Let $\hat{x}\in\Ti^{n+1}$ be such that $\hat{x}_k= \det A_{|k)}$
for all $k\in [n+1]$.
Then the \thin vector $x=\imath(|\hat{x}|)\in \Ti^{n+1}$
satisfies $Ax\bal \zero$.
Moreover, if  $\hat{x}$ is \thin and non-zero,
then for any \thin vector $x$ which is a solution of $Ax\bal \zero$,
we have that $x$ is a \thin multiple of $\hat{x}$.
\end{corollary}
\begin{proof}
An element $x\in \Ti$ is \thin if and only if $x=\imath(|x|)$
and the set $(\Ti^\vee)^*$ is exactly the set of
invertible elements of $\Ti$. Then applying Theorem~\ref{th-jac-homogeneous}
to $\Ti$, we get the first assertion.
The second assertion follows similarly from Corollary~\ref{cor-cramer}.
\end{proof}
\begin{remark}
The special case of the existence result, Theorem~\ref{th-jac},
concerning $\SS = \To_2$ or $\smax$ could
be derived alternatively from the existence and uniqueness result
in the Cramer theorem, Theorem ~\ref{th-cramer}. To do this a perturbation argument can be used
since the matrix $A$ and the vector $b$ can always be ``approximated''
by matrices satisfying the condition of Item (2) of the latter theorem.
\end{remark}

\subsection{The tropical Jacobi algorithm}
The  \thin solution $x$ in Theorem~\ref{th-jac} will be established 
constructively by means of the Jacobi algorithm of~\cite{Plus}.
The following notion of diagonal dominance is inspired
by the notion with the same name which is 
classically used in numerical analysis.
A real matrix $A=(A_{ij})$ is classically
said to have a dominant diagonal if $A_{ii}\geq \sum_{j\neq i} |A_{ij}|$ holds
for all $i\in[n]$.
The tropical analogue of this condition is
\begin{align}
A_{ii}\geq \max_{j\neq i} |A_{ij}|, \qquad \forall i\in[n]
\enspace .\label{e-tropdiag}
\end{align}
We shall use the following related condition.
\begin{definition}\label{def-domdiag}
We shall say that $A=(A_{ij})\in \MM_n(\TT)$ has a \NEW{dominant diagonal}
if 
\[
|\det A|=|A_{11}\cdots A_{nn}|\;\text {and is invertible in }\RR\enspace .
\]
\end{definition}
Corollary~\ref{cor-but} implies that, when $\RR$ is a semifield, any 
matrix which has a dominant diagonal in this sense
is diagonally similar to a matrix satisfying 
the tropical analogue~\eqref{e-tropdiag}
of the usual condition of diagonal dominance, in particular,
Definition~\ref{def-domdiag} is weaker than~\eqref{e-tropdiag}.

The following decomposition is
similar to the one used in the classical (relaxed) Jacobi algorithm.
\begin{prop}\label{exist-jacobi-dec}
Let $\TT$ be a semiring satisfying Property~\ref{pty_order}.
Then any matrix $A\in \MM_n(\TT)$ with a dominant diagonal
can be decomposed into the sum
\[
A=D\oplus N
\]
of matrices $D$ and  $N\in \MM_n(\TT)$ such that
$D$ is a diagonal matrix with diagonal entries in
$\TT^\vee$, and $|\det D|=|\det A|$.
The latter decomposition will be called a \NEW{Jacobi-decomposition} of $A$.
\end{prop}
\begin{proof}
Let $A=(A_{ij})\in \MM_n(\TT)$  be a matrix with a 
dominant diagonal. 
Let  $i\in [n]$. From Property~\ref{pty_order} applied to $x=0$
and $y=A_{ii}$, there exists $\delta_i\in\TT^\vee$
such that $\zero\leq \delta_i\preceq A_{ii}$  and $\delta_i\bala A_{ii}$.
Since $\delta_i\preceq A_{ii}$, there exists $\delta'_i\in \TT$ such that
$\delta_i\oplus \delta'_i= A_{ii}$.
Taking for $D$ the diagonal matrix such that
$D_{ii}=\delta_i$ and for $N$ the matrix such that
$N_{ii}=\delta'_i$ and $N_{ij}=A_{ij}$ for all $i\neq j$, we
get that $A=D\oplus N$.
Moreover, $|D_{ii}|=|A_{ii}|$ for all $i\in [n]$, so $|\det D|=
|A_{11}\cdots A_{nn}|$. Since $A$ has a dominant diagonal,
we obtain $|\det D|=|\det A|$.
\end{proof}

\begin{remark}\label{rem-dom}
We may always reduce the problem $Ax \balance b$ to the case where 
$|\det A|=|A_{11}\cdots A_{nn}|$.
Indeed, since $\mu:x\mapsto |x|$ is a morphism and $|\ominus \unit|=\ooi$.
we have $|\det A|=\per |A|$ for all $A\in\MM_n(\TT)$.
Since $\RR$ is a totally ordered idempotent semiring,
computing the permanent of the matrix $|A|$ is equivalent to solving an 
optimal assignment problem, which furnishes an optimal permutation.
Permuting the rows of the matrix $A$, we can transform
the system into a system $A' x\bal b'$ such that the optimal permutation
of $A'$ is the identity, which implies
$|\det A'|=\per |A'|=|A'_{11}|\cdots |A'_{nn}|= |A'_{11}\cdots A'_{nn}|$.
Since $|\det A'|=|\det A|$, then  $A'$ has a dominant diagonal
as soon as $|\det A|$ is invertible in $\RR$.
Moreover, since $\mu$ is a morphism, 
we get that $| A\adj b |=|A|\adj |b|$ and $|\detp{A} |=\per |A|$.
Using Lemmas~\ref{lhabitude} and~\ref{llem-util07}, we deduce that
proving Theorem~\ref{th-jac} for $Ax \balance b$ is equivalent to
proving it for  $A' x\bal b'$.
\end{remark}

\begin{theorem}[Tropical Jacobi Algorithm, compare with~{\cite[Th.~6.3]{Plus}} and {\cite[III, 6.0.2]{gaubert92a}} for $\S$]\label{theo-jacobi}
Let $\TT$ be a semiring allowing \constr,  see Definition~\ref{monotone-def}.
Let $A\in \MM_n(\TT)$ have a dominant diagonal,
and let $A=D\oplus N$ be a Jacobi-decomposition. Then
\begin{enumerate}
\item\label{jacobi-1} One can construct a sequence $\{x^{k}\}$
of \thin vectors satisfying:
\begin{enumerate}
\item\label{jacobi-i}$\zero=x^{0}\preceq x^{1}\preceq\dots\preceq x^{k}\preceq\dots$;
\item\label{jacobi-ii}$Dx^{k+1}\bala \ominus N x^{k}\oplus b$.
\end{enumerate}
\item \label{jacobi-2} The sequence $|x^k|$ is stationary after at most $n$
iterations, meaning that $|x^{k}|=|x^{n}|$ for all $k\geq n$,
and we have 
\[ |x^{n}|= |\det A|^{-1}|A\adj b|\enspace . \]
\item \label{jacobi-4}  When $\TT$ allows \conv\
 (Definition~\ref{monotone-def}),
the sequence $x^k$ is stationary,
meaning that there exists $m\geq 0$ such that $x^{k}=x^{m}$ for all $k\geq m$.
Moreover, the limit $x^{m}$ is a solution of $Ax\bal b$.
\item\label{jacobi-5} When $\TT$ satisfies also Property~\ref{pty_order-equal},
one can take $m=n$ in the previous assertion.
\end{enumerate}
\end{theorem}

Applying the previous result to the tropical extension $\S$
(in which $|x|$ is invertible if and only if $|x|\neq \szero$ 
or equivalently $x\neq \zero$), we recover~\cite[Th.~6.3]{Plus}.
Applying the same result to the case of the bi-valued tropical semiring
$\TT=\Ti$, and using the same arguments as for Corollary~\ref{new-th}
of Theorem~\ref{th-jac}, we obtain immediately the following result.

\begin{corollary}[Jacobi Algorithm in the Bi-Valued Tropical Semiring]\label{theo-jacobinew}
Let $A\in \MM_n(\Ti)$ have a dominant diagonal
and let $A=D\oplus N$ be a Jacobi-decomposition. Then
the sequence $\{x^{k}\}$ of 
\thin  vectors defined by $x^0=\zero$,
\[
x^{k+1}= \imath(|D|^{-1} |N x^{k}\oplus b|)
\]
is stationary after at most $n$ iterations,
meaning that $x^{n}=x^{n+1}$. Moreover,
$x^{n}$ is a \thin solution of $Ax\bal b$ and
\[ x^{n}=
\imath(|\det A|^{-1}|A\adj b|) \enspace .\qed
\]
\end{corollary}

Before proving Theorem~\ref{theo-jacobi}, we illustrate it by
an example.
\begin{ex}\label{ex-jacobi}
We take $\TT=\smax$, and apply the tropical Jacobi algorithm to the linear system
\begin{equation} \label{eq:ex-jacobi}
\begin{bmatrix} 
5&\ominus 0&3\\
1&3&\ominus 1\\
3&\ominus 2& 1^{\circ}
\end{bmatrix}
\begin{bmatrix} x_{1}\\ x_{2} \\ x_{3}
\end{bmatrix}
\bal 
\begin{bmatrix}
\ominus 1\\ 4^{\circ} \\0
\end{bmatrix}\enspace .
\end{equation}
Denoting by $A$ the matrix of the system and by $b$ its right-hand side,
we get:
\[
  |\det A|^{-1}|A\adj b |=\begin{bmatrix} 0\\ 1\\ 2
\end{bmatrix} \enspace .\]
Let us choose the Jacobi-decomposition $A=D\oplus N$ with:
\begin{equation}\label{ex-jacobi-dec}
D= \begin{bmatrix} 
5& \zero&\zero\\
\zero &3&\zero\\
\zero&\zero& 1\end{bmatrix},\quad 
N= \begin{bmatrix} 
\zero&\ominus 0&3\\
1&\zero&\ominus 1\\
3&\ominus 2& \ominus 1
\end{bmatrix}\enspace.\end{equation}
Applying the Jacobi algorithm, we get the following sequence
starting from $x^0=\zero$:
\[
\left\{\begin{array}{l}
5x^{1}_{1}\bala 0x^{0}_{2}\ominus 3x^{0}_{3}\ominus 1=\ominus 1\\
3x^{1}_{2}\bala \ominus 1x^{0}_{1}\oplus 1x^{0}_{3}\oplus 4^{\circ}=4^{\circ}\\
1x^{1}_{3}\bala \ominus 3x^{0}_{1} \oplus 2x^{0}_{2}\oplus 1x^{0}_{3}\oplus 0=0
\end{array}\right. 
\Rightarrow\; 
\left\{\begin{array}{l}
x^{1}_{1}=\ominus -4\\
x^{1}_{2}=1 \text{ or } \ominus 1,\;\text{we choose } x^{1}_{2}:=1\\
x^{1}_{3}= -1 
\end{array}\right.
\]
\[
\left\{\begin{array}{l}
5x^{2}_{1}\bala 0x^{1}_{2}\ominus 3x^{1}_{3}\ominus 1=\ominus 2\\
3x^{2}_{2}\bala \ominus 1x^{1}_{1}\oplus 1x^{1}_{3}\oplus 4^{\circ}=4^{\circ}\\
1x^{2}_{3}\bala \ominus 3x^{1}_{1} \oplus 2x^{1}_{2}\oplus 1x^{1}_{3}\oplus 0=3
\end{array}\right. 
\;\text{and}\quad
x^{2}_{2}\succeq x^{1}_{2}\Rightarrow
\left\{\begin{array}{l}
x^{2}_{1}=\ominus -3\\
x^{2}_{2}=1\\
x^{2}_{3}=2 
\end{array}\right. 
\]
\[
\left\{\begin{array}{l}
5x^{3}_{1}\bala 0x^{2}_{2}\ominus 3x^{2}_{3}\ominus 1=\ominus 5\\
3x^{3}_{2}\bala \ominus 1x^{2}_{1}\oplus 1x^{2}_{3}\oplus 4^{\circ}=4^{\circ}\\
1x^{3}_{3}\bala \ominus 3x^{2}_{1} \oplus 2x^{2}_{2}\oplus 1x^{2}_{3}\oplus 0=3
\end{array}\right.
\;\text{and}\quad
x^{3}_{2}\succeq x^{2}_{2}\Rightarrow
\left\{
\begin{array}{l}
x^{3}_{1}=\ominus 0\\
x^{3}_{2}=1\\
x^{3}_{3}=2 \enspace .
\end{array}\right. 
\]
Choosing $x^{1}_{2}= \ominus 1$ yields another solution:
$x^{3}=(0,\ominus 1,\ominus 2)^\top$ (where $x^\top$ denotes the transpose of
$x$), whereas taking the opposite sign in
$D_{33}$ and $N_{33}$ would have lead to the other possible solutions 
$x^{3}=(0,1,\ominus 2)^\top$, and $x=(\ominus 0,\ominus 1, 2)^\top$.
\end{ex}

\subsection{Proof of Theorems~\ref{theo-jacobi} and~\ref{th-jac}}
\label{subsec-stationn}
For the proof of Theorem~\ref{theo-jacobi} we shall need 
the following lemmas which are derived from the results of 
Section~\ref{sec-cramer} and from Theorem~\ref{th-yoeli}.

\begin{lemma} \label{lem-dom}  Let $A\in \M(\TT)$ be a matrix with a 
dominant diagonal and let $A=D\oplus N$ be a Jacobi-decomposition. 
Then $|D|$ is an invertible diagonal matrix which coincides with the
diagonal submatrix of $|A|$: $|D_{ii}|=|A_{ii}|$ for all $i\in [n]$.
\end{lemma}
\begin{proof}
Since $A=D\oplus N$ is a Jacobi-decomposition and
$\mu$ is a morphism, we have $\per |D|=|\det D|=|\det A|$.
This implies that $\per |D|$ is invertible in $\RR$,
since $A$ has a dominant diagonal.
$|D|$ is a diagonal matrix. Thus  it is necessarily a 
monomial matrix with entries in $\RR$.
From $|D_{11}|\cdots |D_{nn}|=\per |D|$ and the property that
$\per |D|$ is invertible in $\RR$, we get that all 
diagonal entries $|D_{ii}|$ of $|D|$ are invertible in $\RR$. Thus
$|D|$ is invertible in $\M(\RR)$.
From $|D|\oplus |N|= |A|$ we deduce that
$|D_{ii}|\leq |A_{ii}|$ for all $i\in [n]$.

Let us show the reverse inequalities, which will imply the equalities.
Since $A$ has a dominant diagonal, we have 
$|D_{11}|\cdots |D_{nn}|=\per |D|=|\det A|=
|A_{11}\cdots A_{nn}|= |A_{11}|\cdots |A_{nn}|$.
Hence all the $|A_{ii}|$ are invertible in $\RR$.
We have for instance $|A_{11}|\cdots |A_{nn}|
=|D_{11}|\cdots |D_{nn}|\leq |D_{11}||A_{22}|\cdots  |A_{nn}|$.
Since $|A_{22}|,\dots,|A_{nn}|$ are invertible, we deduce that
$|A_{11}|\leq |D_{11}|$. The same argument shows
that $|A_{ii}|\leq |D_{ii}|$ holds for all $i\in [n]$.
\end{proof}

\begin{lemma}\label{lem-sec-cir} Let $A\in \M(\TT)$ be a matrix with a 
dominant diagonal and let $A=D\oplus N$ be a Jacobi-decomposition. Then 
$\per (|D|^{-1} |N|) \leq \per (|D|^{-1} |A|)=\ooi$.
Hence every circuit of $|D|^{-1}|N|$ has a weight
less than or equal to $\ooi$.
\end{lemma}
\begin{proof} 
By Lemma~\ref{lem-dom}, $|D|$ is invertible.
Let us show that $\per (|D|^{-1} |A|)=\ooi$.
Since $|D|$ is a monomial matrix of $\M(\RR)$,
we have  $\per (|D|^{-1} |A|)=(\per|D|)^{-1} (\per |A|)$. \sloppy
Since $\per|D|=|\det A|=\per |A|$,
we obtain that $\per (|D|^{-1} |A|)=\ooi$.

Let us denote $M=|D|^{-1}|N|$. Since $|D|^{-1} |A|= I\oplus M$, we get that
$M\leq |D|^{-1} |A|$ for the natural order
of $\RR$. This implies that $\per M \leq \per (|D|^{-1} |A|)=\ooi$
and that the weight of a circuit for $M$ is 
less than or equal to its  weight for $|D|^{-1} |A|$.
From Theorem~\ref{th-yoeli} we get that
the weight of a circuit of $|D|^{-1} |A|$ is 
less than or equal to $\ooi$, which implies the same for $M$.
\end{proof}

\begin{lemma} \label{lem-meta-yoeli}  Let $A\in \M(\TT)$ be a matrix with a 
dominant diagonal and let $A=D\oplus N$ be a Jacobi-decomposition. 
Then $(|D|^{-1}|N|)^{*}|D|^{-1}=(\per|A|)^{-1} |A|\adj=|\det A|^{-1}|A\adj|$.
\end{lemma}
\begin{proof} 
Let $M=|D|^{-1}|N|$ and $Q=|D|^{-1}|A|$.
Since $\mu$ is a morphism, we have easily
$(\per|A|)^{-1} |A|\adj=|\det A|^{-1}|A\adj|$.
Since $A=D\oplus N$, we also have $Q= I\oplus M$.
Then $M^*=Q^*$ and it remains to show that
$Q^*|D|^{-1}=(\per|A|)^{-1} |A|\adj$.

By Lemma~\ref{lem-sec-cir} $\per (Q)=\ooi$.
Since by Lemma~\ref{lem-dom}
$|D|$ is the diagonal of $|A|$, the diagonal
entries of $Q$ are equal to $\ooi$. Thus
$Q$ satisfies the 
conditions of Theorem~\ref{th-yoeli}. Hence
$ Q^*=Q\adj$. Since $|D|$ is a monomial matrix,
Lemmas~\ref{llem-util07} and~\ref{lhabitude} imply that 
 $Q\adj=|A|\adj (|D|^{-1})\adj=|A|\adj (\per |D|^{-1}) |D|=
(\per |A|)^{-1} |A|\adj |D|$, thus
$Q^*|D|^{-1}=Q\adj |D|^{-1}= (\per |A|)^{-1} |A|\adj$, which finishes the proof.
\end{proof}

\begin{proof}[Proof of Theorem~\ref{theo-jacobi}]
We first prove Item~\eqref{jacobi-1}.
Since $A=D\oplus N$ is a Jacobi-decomposition, $D$ is a diagonal
matrix with entries in $\TT^\vee$.
By Lemma~\ref{lem-dom}, $|D|$ is an invertible diagonal matrix,
which implies that all its diagonal entries are invertible in $\RR$.
Hence, one can construct from Property~\ref{pty_inv}, 
the diagonal $n\times n$ matrix $\widetilde{D}$ such that
$\widetilde{D}_{ii}=\widetilde{D_{ii}}$.
Then Condition \eqref{jacobi-ii} of Theorem~\ref{theo-jacobi}  is equivalent to:
\begin{equation}\label{jacobi-ii'}
x^{k+1}\bala \widetilde{D}(\ominus N x^{k}\oplus  b  )
\enspace.
\end{equation}

Let us show by induction on $k$ that there exist \thin
$n$-dimensional vectors 
$\zero=x^{0}\preceq x^{1}\preceq\dots\preceq x^{k}\preceq x^{k+1}$
satisfying~\eqref{jacobi-ii'}, together with
\begin{equation}\label{eq-etoui}
x^{k+1}\preceq \widetilde{D}(\ominus N x^{k}\oplus  b  )
\enspace.
\end{equation}
When $k=0$, the above conditions
are equivalent to $x^{1}\in (\TT^{\vee})^{n}$,
$\zero\preceq x^1\preceq  \widetilde{D} b$ and $x^{1}\bala \widetilde{D} b$.
Thus  $x^1$  can be constructed by
applying Property~\ref{pty_order} entrywise. Hence,
the properties of the induction hold for $k=0$.

Assume now that these properties hold for some $k\geq 0$.
Since the map $x\mapsto  \widetilde{D}(\ominus N x\oplus 
 b  )$ is non-decreasing and $x^{k}\preceq x^{k+1}$,
we have from~\eqref{eq-etoui} that 
$x^{k+1} \preceq \widetilde{D}(\ominus N x^{k+1}\oplus  b  )$.
Then applying Property~\ref{pty_order} entrywise,
we construct $x^{k+2}\in (\TT^{\vee})^{n}$ such that
$x^{k+1} \preceq x^{k+2}\preceq  \widetilde{D}(\ominus N x^{k+1}\oplus  b  )$ and
$x^{k+2} \bala \widetilde{D}(\ominus N x^{k+1}\oplus  b  )$.
This shows that the induction hypothesis holds for $k+1$.
Thus Item~\eqref{jacobi-1} of the theorem is proven.

We next prove Item~\eqref{jacobi-2}.
Denote $\hat{x}^{k}:=|x^{k}|$. Since $\mu$ is a morphism,
Condition~\eqref{jacobi-ii}  of Theorem~\ref{theo-jacobi} implies
 that $|D| \hat{x}^{k+1}=|N| \hat{x}^{k}\oplus |b|$.
Since $|D|$ is invertible in $\M(\RR)$, we obtain that 
$ \hat{x}^{k+1}=M\hat{x}^{k}\oplus |D|^{-1}|b|$ for $M=|D|^{-1}|N|$.
Hence, for all $k\geq 0$ we get that
\[\hat{x}^{k+1}=(\Id\oplus M\oplus \dots \oplus M^{k})|D|^{-1}|b|\enspace .\]
Lemma~\ref{lem-sec-cir} allows us to apply the 
theorem of Carr\'e and Gondran (Proposition-Definition~\ref{propdef})
to the matrix $M$.
It follows that
\[
|x^{k}|=\hat{x}^{k}=M^{*}|D|^{-1}|b|\quad \text{ for all } k\geq n \enspace .
\]
The last assertion of Item~\eqref{jacobi-2}
follows from the previous equation and Lemma~\ref{lem-meta-yoeli}.

We now prove Item~\eqref{jacobi-4}.
Since all $x^{k}$ are \thinp, $|x^{k}|=|x^{n}|$ for all
$k\geq n$, and $x^{n}\preceq  x^{n+1}\preceq\cdots$.
Applying  Property~\ref{pty_order-finite} entrywise implies
that $x^n$ is stationary after some finite time $m$.
Then by Condition~\eqref{jacobi-ii} of the theorem we get that
$ Dx^{m}\bal \ominus N x^{m}\oplus b $, which is
equivalent to $Ax^{m}\bal b$. 

We finally prove Item~\eqref{jacobi-5}.
When $\TT$ satisfies also  Property~\ref{pty_order-equal}, 
the above properties imply that $x^n=x^{n+1}$. Thus 
one can take $m=n$ in the previous conclusions.
\end{proof}

\begin{proof}[Proof of Theorem~\ref{th-jac}]
Applying Theorem~\ref{theo-jacobi} after the transformation described in
Remark~\ref{rem-dom} and using Proposition~\ref{exist-jacobi-dec},
we deduce Theorem~\ref{th-jac}.
\end{proof}

\subsection{The tropical Gauss-Seidel algorithm}
We now introduce a Gauss-Seidel type algorithm. 
It  is a variant of the Jacobi algorithm in which
the information is propagated more quickly.

\begin{prop}\label{exist-gauss-seidel-dec}
Let $\TT$ be a semiring satisfying Property~\ref{pty_order}.
Then any matrix $A\in \MM_n(\TT)$ with a dominant diagonal
can be decomposed into the sum
\[
A=D\oplus L\oplus U\enspace,
\]
of matrices $D$, $L$ and $U\in \MM_n(\TT)$ such that
$D$ is a diagonal matrix with diagonal entries in
$\TT^\vee$, $|\det D|=|\det A|$,  all the entries on the main diagonal and above the main diagonal of
$L$ equal to $\zero$, 
and  all the entries below the main diagonal of $U$ equal to $\zero$.
The latter decomposition will be called a \NEW{Gauss-Seidel-decomposition} 
of $A$.
\end{prop}
\begin{proof}
By Proposition~\ref{exist-jacobi-dec} there exists
a Jacobi-decomposition $A=D\oplus N$.
Taking for $L$ the (strict) lower diagonal part of $N$ and for
$U$ the upper diagonal and diagonal part of $N$, we get the result.
\end{proof}

\begin{theorem}\label{theo-gauss-seidel}\index{Gauss-Seidel (algorithm)}
Let $\TT$ be a semiring allowing \constr,  see Definition~\ref{monotone-def}.
Let $A\in \MM_n(\TT)$ have a dominant diagonal
and let $A=D\oplus L\oplus U$ be a Gauss-Seidel-decomposition. Then
\begin{enumerate}\label{gauss1}
\item One can construct a sequence $\{x^{k}\}$
of \thin vectors satisfying:
\begin{enumerate}
\item \label{gauss11}
$\zero=x^0 \preceq x^1\preceq \dots\preceq x^k\preceq \dots$;
\item  \label{gauss12}$Dx^{k+1}\bala \ominus Lx^{k+1}\ominus U x^{k}\oplus b$;
\item  \label{gauss13}$|x^{k+1}|= |\det A|^{-1}|(D\oplus L)\adj|| \ominus U x^{k}\oplus  b|$.
\end{enumerate}
\item  \label{gauss2}The sequence $|x^k|$ is stationary after at most $n$
iterations, meaning that $|x^{k}|=|x^{n}|$ for all $k\geq n$,
and we have 
\[ |x^{n}|= |\det A|^{-1}|A\adj b| \enspace . 
\]
\item  \label{gauss3} When $\TT$ allows \conv\ (Definition~\ref{monotone-def}),
 the sequence $x^k$ is stationary,
meaning that there exists $m\geq 0$ such that $x^{m}=x^{m+1}$. Moreover,
the limit $x^{m}$ is a solution of $Ax\bal b$.
\item \label{gauss4} When $\TT$ satisfies also Property~\ref{pty_order-equal}
one can choose $m=n$ in the previous assertion.
\end{enumerate}
\end{theorem}

\begin{ex} We consider the system of Example~\ref{ex-jacobi},
with the Gauss-Seidel decomposition $A=D\oplus L\oplus U$
such that $L\oplus U=N$ with $D$ and $N$ as in~\eqref{ex-jacobi-dec}.
Then starting from $x^0=\zero$ we obtain the following sequence:
\[
\left\{\begin{array}{lcl}
5x^{1}_{1}\bala 0x^{0}_{2}\ominus 3x^{0}_{3}\ominus 1=\ominus 1
&\Rightarrow &x^{1}_{1}=\ominus -4 \\
3x^{1}_{2}\bala \ominus 1x^{1}_{1}\oplus 1x^{0}_{3}\oplus 4^{\circ}=4^{\circ} 
&\Rightarrow 
& x^{1}_{2}=1\text { or } \ominus 1, \text{ we choose } x^{1}_{2}=1 \\
1x^{1}_{3}\bala \ominus 3x^{1}_{1} \oplus 2x^{1}_{2}\oplus 1x^{0}_{3}\oplus 0=3 
& \Rightarrow &x^{1}_{3}=2
\end{array}\right. 
\]
\[
\left\{\begin{array}{lcl}
5x^{2}_{1}\bala 0x^{1}_{2}\ominus 3x^{1}_{3}\ominus 1=\ominus 5
&\!\Rightarrow\! &x^{2}_{1}=\ominus 0 \\
3x^{2}_{2}\bala \ominus 1x^{2}_{1}\oplus 1x^{1}_{3}\oplus 4^{\circ}=4^{\circ}
&\!\Rightarrow\!
& x^{2}_{2}=1\text { or } \ominus 1, \text{ and } x^{2}_{2}\succeq x^{1}_{2}
\!\Rightarrow\!  x^{2}_{2}=1 \\
1x^{2}_{3}\bala \ominus 3x^{2}_{1} \oplus 2x^{2}_{2}\oplus 1x^{1}_{3}\oplus 0=3
& \!\Rightarrow\! &x^{1}_{3}=2\end{array}\right. 
\]
We find the
solution $(\ominus 0,1,2)^\top$ after 2 iterations only,
whereas the Jacobi algorithm required 3 iterations.
\end{ex}

\begin{lemma}\label{lem-star-u}
Let $A\in \M(\TT)$ be a matrix with a 
dominant diagonal and let $A=D\oplus L\oplus U$ be a
Gauss-Seidel-decomposition. 
Let $M=|\det A|^{-1}|(D\oplus L)\adj|\in \M(\TT)$. 
Then $|D|^{-1}|L|$ has no circuit
(all circuits have a zero weight) and we have $M=(|D|^{-1}|L|)^* |D|^{-1}$.
\end{lemma}
\begin{proof} Since the entries of $L$ are $\zero$ on and above
the diagonal, the graph of $|D|^{-1}|L|$ has no circuit.
Let $A'=D\oplus L$. Then 
$|\det A'|=|\det D|= |\det A|$. Taking $D'=D$ and
$N'=L$, we get a Jacobi-decomposition $D'\oplus N'$ 
of $A'$. Applying Lemma~\ref{lem-meta-yoeli}
to it, we get the last assertion of the lemma.
\end{proof}
\begin{lemma}\label{lem-2-nonpos} 
Let $A\in \M(\TT)$ be a matrix with a 
dominant diagonal. Assume $A=D\oplus L\oplus U$ is a
Gauss-Seidel-decomposition. Let
$M=|\det A|^{-1}|(D\oplus L)\adj|\in \M(\TT)$. Then
\begin{itemize}
\item[(a)] Every circuit of $M|U|$ has a weight less than or equal to $\unit$.
\item[(b)] $(M|U|)^* M =|\det A|^{-1}|A\adj|$.
\end{itemize}
\end{lemma}
\begin{proof}
We  start by proving item (b).
Indeed, adding a top element to $\RR$ allows one to define
the Kleene star $B^*$ of a matrix $B$ with entries in $\RR$. Then
every circuit of $B$ has a weight less than or equal to $\unit$
if and only if $B^*$ has all its entries in $\RR$ (meaning that they are all different from the top element).

Let $N=L\oplus U$. 
It is easy to see that $D\oplus N$ is a Jacobi-decomposition of $A$. Thus
by Lemma~\ref{lem-sec-cir}, every circuit of $|D|^{-1}|N|$
has a weight less than or equal to $\unit$.
By Lemma~\ref{lem-meta-yoeli}, we have 
$(|D|^{-1}|N|)^{*}|D|^{-1}=|\det A|^{-1}|A\adj|$.

Denote $L'=|D|^{-1}|L|$ and $U'=|D|^{-1}|U|$.
Then by Lemma~\ref{lem-star-u}, we have $M=(L')^* |D|^{-1}$, so $M|U|=(L')^* U'$,
and $|D|^{-1}|N|=L'\oplus U'$.
This implies that the assertion of item (b) is equivalent to the equality
\begin{equation} \label{LU}
((L')^* U')^* (L')^* 
=(L'\oplus U')^*,
\end{equation}
 which is indeed a well known 
{\em unambiguous rational identity}
(by expanding the Kleene star and products in
both expressions, we arrive at the sum of all words in the letters
$L'$ and $U'$).

\noindent 
(a) Let us note that the identity \eqref{LU} also shows 
that $(M|U|)^*=((L')^* U')^* \leq (L'\oplus U')^*=(|D|^{-1}|N|)^*$
and since every circuit of $|D|^{-1}|N|$
has a weight less than or equal to $\unit$,
the latter expression has all its entries in $\RR$,
so has $(M|U|)^*$.
This implies that every circuit of $M|U|$ has a weight less
than or equal to $\unit$.
\end{proof}

\begin{proof}[Proof of Theorem~\ref{theo-gauss-seidel}]
The proof follows the same lines as the one of Theorem~\ref{theo-jacobi}.
In particular, constructing the same matrix $\widetilde{D}$
as in this proof, we get that Condition \eqref{gauss12} of
Theorem~\ref{theo-gauss-seidel} is equivalent to
$x^{k+1}\bala \widetilde{D}(\ominus Lx^{k+1}\ominus U x^{k}\oplus b)$,
which corresponds to the system:
\begin{align*}
x^{k+1}_1&\bala\widetilde{D}_{11}(\ominus U_{11} x^{k}_1\cdots\ominus 
U_{1n} x^k_n \oplus b_1)\\
&\;\;\vdots\\
x^{k+1}_i&\bala \widetilde{D}_{ii}(\ominus L_{i1} x^{k+1}_1\cdots\ominus 
L_{i,i-1} x^{k+1}_{i-1}\ominus  U_{ii} x^{k}_i\cdots\ominus 
U_{in} x^k_n \oplus b_i)\\
&\;\;\vdots  \\
x^{k+1}_n&\bala \widetilde{D}_{nn}(\ominus L_{n1} x^{k+1}_1\cdots\ominus 
L_{n,n-1} x^{k+1}_{n-1}\ominus U_{nn} x^{k}_n \oplus b_n)\enspace.\\
\end{align*}
Using Property~\ref{pty_order}  for each $k$ and each $i\in[n]$ one chooses
 $x^{k+1}_i$ 
such that it satisfies the $i$th equation of this system, 
together with the conditions $x^{k}_i\preceq x^{k+1}_i$ and 
\[ x^{k+1}_i\preceq \widetilde{D}_{ii}(\ominus L_{i1} x^{k+1}_1\cdots\ominus 
L_{i,i-1} x^{k+1}_{i-1}\ominus  U_{ii} x^{k}_i\cdots\ominus 
U_{in} x^k_n \oplus b_i)\enspace .\]
Then the sequence satisfies Conditions~\eqref{gauss11} and~\eqref{gauss12}
of the theorem.
In particular it satisfies
$|x^{k+1}|=|D|^{-1}|L||x^{k+1}|\oplus |D|^{-1}|\ominus U x^{k}\oplus b|$.
Since $|D|^{-1}|L|$ has no circuit, the theorem of Carr\'e and Gondran
(Proposition-Definition~\ref{propdef}) implies that
$|x^{k+1}|=(|D|^{-1}|L|)^*|D|^{-1}|\ominus U x^{k}\oplus b|$,
which by Lemma~\ref{lem-star-u} is equivalent to Condition~\eqref{gauss13}
of the theorem.

Again the latter condition implies
that $|x^{k+1}|=M |U| |x^{k}|\oplus M |b|$ with $M$ as in 
Lemma~\ref{lem-star-u}. Then 
by Lemma~\ref{lem-2-nonpos} and the theorem of Carr\'e and Gondran,
we get that $|x^{k}|=(M|U|)^* M |b|$  for all $k\geq n$.
Moreover from the second assertion of Lemma~\ref{lem-2-nonpos}, we have 
$|x^{n}|=|\det A|^{-1}|A\adj| |b|=|\det A|^{-1}|A\adj b|$
which shows Item~\eqref{gauss2} of the theorem.

Items~\eqref{gauss3} and~\eqref{gauss4} of the theorem are 
obtained by the same arguments as for Theorem~\ref{theo-jacobi}.
\end{proof}

The particular case of Theorem~\ref{theo-gauss-seidel}
concerning the tropical extension $\S$ was obtained in~\cite{gaubert92a}.
One can also apply the same result to the case of the bi-valued tropical semiring
$\TT=\Ti$. This leads to the same solution
$x=\imath(|\det A|^{-1}|A\adj b|)$ as in Corollary~\ref{theo-jacobinew},
by using the Gauss-Seidel algorithm instead of the Jacobi algorithm.

\section{Homogeneous systems: the generalized Gondran-Minoux theorem}\label{sec-homogeneous}
The following result was stated in~\cite{Plus}.

\begin{theorem}[{\cite[Th.~6.5]{Plus}}]
\label{TGMa}
Let $A \in\M(\smax)$. Then there exists $x\in (\smax^\vee)^n\setminus \{\zero\}$ such that $Ax\balance \zero$ if and only if $\det A \balance \ooo$.
\end{theorem}
The special case in which $A\in \M(\rmax)$ is equivalent to the theorem of Gondran and Minoux quoted in the introduction (Theorem~\ref{th-4}).

The ``only if'' part is obtained by taking $b=\zero$ in the first 
assertion of Theorem~\ref{th-cramer}. 
The ``if'' part was proved in~\cite{gaubert92a}. An analogous
result was proved by Izhakian and Rowen~\cite{IzhRowen}, when the 
symmetrized tropical semiring $\smax$ is replaced by the bi-valued tropical semiring $\Ti$. We next provide a general result which includes
Theorem~\ref{TGMa} as well as the result of~\cite{IzhRowen} as special cases.

Let $\TT$  be a semiring with a symmetry, a \thin set $\TT^\vee$,
and a modulus taking its values in a totally ordered semiring $\RR$.
We shall need the following additional properties.

\begin{pty}\label{p-temp-0}
$\RR$ is an idempotent semifield and the \thin set $\TT^\vee$ is such that
the set of invertible elements of $\TT$ is $(\TT^\vee)^*$ 
and that it coincides with $\TT\setminus \TT^\circ$.
\end{pty}

This property is satisfied when $\TT=\skewproductstar{\SS}{\RR}$, 
with the thin set $\TT^\vee=\skewproductstar{\SS^\vee}{\RR}$,
$\RR$ is an idempotent semifield and $(\SS^\vee)^*=\SS\setminus\SS^\circ$
is the set of invertible elements.

\begin{pty}\label{p-temp-5}
For all $x\in \TT$  we have $x=\zero\Leftrightarrow |x|=\zero$.
\end{pty}
This property is satisfied when $\TT=\skewproductstar{\SS}{\RR}$.

\begin{pty}\label{p-temp-2}
For all $x_1,\dots,x_k\in \TT$ such that
$x:=x_1\oplus \dots \oplus x_k\in \TT^\circ$,
either there exists a single index $i\in [k]$ such that
$x_i\in \TT^\circ$ and $|x_i|=|x|$,
or there exist two different indices $i,j\in [k]$ 
such that $x_i \oplus x_j \in \TT^\circ$
and $|x_i|=|x_j|=|x|$.
\end{pty}
This property is satisfied when $\TT=\skewproductstar{\SS}{\RR}$, 
with the thin set $\TT^\vee=\skewproductstar{\SS^\vee}{\RR}$, and 
for all $a_1,\dots,a_k\in \SS$ such that
$a:=a_1+ \dots + a_k\in \SS^\circ$,
either there exists a single index $i\in [k]$ such that
$a_i\in \SS^\circ$, or there exist two different indices $i,j\in [k]$ 
such that $a_i + a_j \in \SS^\circ$.

\begin{pty}\label{p-temp-3}
If $x\in \TT^\circ$ and $|y|\leq |x|$ then $x \oplus y\in \TT^\circ$.
\end{pty}
This  property is satisfied when Property~\ref{p-temp-0} and
the result of Proposition~\ref{prop-equal} hold. We can then obtain
the following assertion.
\begin{fact} 
All the following semirings satisfy Properties~\ref{p-temp-0}--\ref{p-temp-3}:
the symmetrized max-plus semiring $\S$, 
the bi-valued tropical semiring $\Ti$, 
the tropical extension of the torus $\skewproductstar{\bar{\To}}{\rmax}$
or that
of any group with a non trivial symmetry $\skewproductstar{\bar{G}}{\rmax}$
(Example~\ref{groupext}), and any supertropical semifield
(see Remark~\ref{rem-supertropical}).

However the phase extension of the tropical semiring
$\skewproductstar{\phase}{\rmax}$ (Example~\ref{ex-viro})
does not satisfy  Properties~\ref{p-temp-0}, \ref{p-temp-2}, nor
\ref{p-temp-3}.
\end{fact}

We note the following consequences, the first one being easy.

\begin{prop}\label{p-temp-1}
Let $\TT$ satisfy Property~\ref{p-temp-0}. For all $x,y\in \TT$ if 
$xy\in \TT^\circ$ then $x\in \TT^\circ$ or $y\in \TT^\circ$. \qed
\end{prop}

\begin{prop}\label{p-temp-4}
Let $\TT$ satisfy Properties~\ref{p-temp-0} and~\ref{p-temp-3}. 
If $\unit \oplus xy\in \TT^\circ$ where  $|x|\leq \unit$ and $|y|\leq \unit$,
then $\unit \oplus x \in \TT^\circ$ or $\unit \oplus y\in \TT^\circ$.
\end{prop}
\begin{proof}
We have $(\unit \oplus x)(\unit \oplus y) = \unit \oplus xy \oplus x \oplus y
\in \TT^\circ$ by Property~\ref{p-temp-3}. By Proposition~\ref{p-temp-1}
we must have $\unit \oplus x\in \TT^\circ$ or $\unit \oplus y \in \TT^\circ$.
\end{proof}

\begin{theorem}[Homogeneous balances]
\label{newTGMa}
Let $\TT$ be a semiring allowing weak balance elimination 
(Definition~\ref{def-tropextension}) and
\conv\ (Definition~\ref{monotone-def}),
 and satisfying Properties~\ref{p-temp-0}--\ref{p-temp-3}.
Let $A \in\M(\TT)$. Then there exists $x\in (\TT^\vee)^n\setminus\{\zero\}$
such that $Ax\balance \zero$ if and only if $\det A \balance \ooo$.
\end{theorem}
\begin{proof}
The necessity of the condition
$\det A \balance \ooo$ follows by taking $b=\zero$ in the
first assertion of Theorem~\ref{th-cramer}.
Indeed, $Ax\bal \zero$ and  $x\in (\TT^\vee)^n$
imply $(\det A) x\bal \zero$. At least one entry of $x$ 
belongs to $(\TT^\vee)^*$. By Property~\ref{p-temp-0}, this entry is
invertible and so $\det A \bal \zero$. 

Let us prove that  $\det A \bal \zero$ implies the existence of $x\in (\TT^\vee)^n\setminus\{\zero\}$
such that $Ax\balance \zero$. 

{\it Case 1}: We first deal with the degenerate case in which $\det A = \zero$.
Then $\per|A|=|\det A|=\zero$, so applying
the Frobenius-K\"onig's theorem to $|A|$, and using 
Property~\ref{p-temp-5},
there exists a reordering of rows and columns such that the matrix
$A$ has the following form: 
\[
A=\begin{bmatrix} \zero_{p\times q}& B\\ C& D\end{bmatrix}
\enspace ,
\]
where $p+q=n+1$, $C\in \TT^{(n-p)\times q}$, and $\zero_{p\times q}$ denotes
the $p\times q$ zero matrix.
It suffices to look for a solution $x$
such that $x_{i}=\zero$ for all $q+1\leq i\leq n$. Denoting
by $y$ the vector with entries $x_1,\dots,x_q$, it remains
to solve the system $Cy\bal\zero$ which has $q$ unknowns 
and $n-p=q-1$ equations.
Thus if this system has a non-zero $(q-1)\times (q-1)$-minor, 
an application of Theorem~\ref{th-jac-homogeneous}
provides a non-zero \thin solution of the system $Cy\bal\zero$. 
Otherwise, we may assume by induction
that the sufficiency in the theorem (or at least its restriction 
to the case $\det A=\zero$) is already proved for systems of lower dimension.
Then we apply the induction to a square subsystem
$C'y'\bal\zero$ obtained by setting to zero one coordinate of $y$.
This completes the treatment of the degenerate case.

{\it Case 2}: We now assume that $\det A \neq \zero$, so $\per |A|=|\det A|\neq \zero$.
By Corollary~\ref{cor-but} applied to the matrix $C=|A|$, there exist
two diagonal matrices $D$ and $D'$ with invertible diagonal entries
in $\RR$
 and a permutation matrix $\Sigma$, such that $C'=\Sigma DCD'$ satisfies
$C'_{ij}\leq 1$ and $C'_{ii}=1$  for all $i,j\in [n]$.
Then applying the injection $\imath$ to the matrices $D$, $D'$ and $\Sigma$,
and using the fact that $\imath$ is a multiplicative morphism, we obtain
a matrix $A'=\imath(\Sigma) \imath(D) A \imath(D')$
such that $|A'|=C'$ and so satisfies the above properties.
Since a diagonal scaling of $A$ does not change the balanced character
of the determinant, nor the existence of a \thin solution of $Ax\bal \zero$,
we may always assume that $A=A'$. Thus $A$ satisfies:
\begin{equation}\label{anorm}
|A_{ij}|\leq 1,\qquad |A_{ii}|=1,\qquad \forall i,j\in [n] \enspace .
\end{equation}

{\it Subcase 2.1}: We shall first consider the subcase in which
 there is a permutation $\sigma$ such that 
\begin{align}\label{e-sat-perm}
 \bigodot_{i\in [n]} A_{i\sigma(i)} \bal \zero \text{ and }
\per |A|= |\det A| = |\bigodot_{i\in [n]} A_{i\sigma(i)}|
=\prod_{i\in [n]} |A_{i\sigma(i)}| \enspace .
\end{align}
Assume, possibly after permuting the rows of $A$, that $\sigma$ is the
identity permutation (this does not change Property~\eqref{e-sat-perm}).
Since $ \bigodot_{i\in [n]} A_{ii} \bal \zero $,
by Proposition~\ref{p-temp-1}, we must have $A_{jj}\bal \zero$ for some
$j\in [n]$ and we may always assume that $j=1$. Then $A$ can be written
in block form as 
\[
A= \begin{pmatrix} A_{11} & c \\ b & F \end{pmatrix} \enspace .
\]
We set $x_1:=\unit$ and define $y:=(x_2,\dots,x_n)^\top$ to be
a \thin solution of $b \oplus F y \bal \zero$ provided by Theorem~\ref{th-jac}.
Thus $|y|= |F\adj b|= |F|^*b$, and so, $|y_j|\leq \unit$ for
all $j$. Since $A_{11}\bal\zero$ and $|A_{11}| = \unit$,
it follows from Property~\ref{p-temp-3} and~\eqref{e-sat-perm}
that $A_{11} \oplus  c y \bal \zero$.
Hence, $x:=(x_1,y_1,\dots,y_{n-1})^\top$ is a non-zero \thin solution of $Ax\bal \zero$.

{\it Subcase 2.2}: It remains to consider the subcase in which no permutation $\sigma$ satisfies~\eqref{e-sat-perm}. 
Since $\det A \bal \zero$, by Property~\ref{p-temp-2},
there must exist two distinct permutations $\sigma$ and $\pi$ such that
\[
|\det A |= 
 |\bigodot_{i\in [n]} A_{i\sigma(i)} | = |\bigodot_{i\in [n]} A_{i\pi(i)} | \enspace,
\]
and
\[
\operatorname{sgn}(\sigma)\bigodot_{i\in [n]} A_{i\sigma(i)}
\oplus \operatorname{\sgn}(\pi)\bigodot_{i\in [n]} A_{i\pi(i)}  \bal \zero.
\]
We may always assume, possibly after permuting the rows of $A$,
that $\sigma$ is the identity permutation 
(this does not change Property~\eqref{e-sat-perm}).
Moreover, since the identity permutation $\sigma$ does not
 satisfy~\eqref{e-sat-perm},
but satisfies $|\bigodot_{i\in [n]} A_{ii} |=|\det A|$,
we deduce that $\bigodot_{i\in [n]} A_{ii} \not\in \TT^\circ$.
Hence by Property~\ref{p-temp-0} all  $A_{ii}$ are invertible.
Multiplying the system by the inverse of the diagonal submatrix of $A$,
we get a new matrix $A$ with all the above properties and
such that $A_{ii}=\unit$ for all $i\in [n]$. Thus
\[
\unit \oplus \operatorname{\sgn}(\pi)\bigodot_{i\in [n]} A_{i\pi(i)}  \bal \zero.
\]
Let us decompose $\pi$ as a product of disjoint cycles $c^1,\dots , c^k$, 
with supports $I_1,\dots,I_k$ of cardinalities $p_1,\dots,p_k$, respectively.
Then 
\[
\unit \oplus \bigodot_{m\in[k]} (\ominus \unit)^{p_m-1} \bigodot_{i\in I_m} A_{i, c^m(i)}  \bal \zero.
\]
It follows from~\eqref{anorm} and Proposition~\ref{p-temp-4} 
that there exists
a cycle $c^m$ such that
\begin{align}\label{e-compat}
\unit \oplus (\ominus \unit)^{p_m-1} \bigodot_{i\in I_m} A_{i, c^m(i)}  \bal \zero.
\end{align}
We may assume, without loss of generality, that $I_m= \{1,\ldots,p_m\}$,
with $c^m(1)=2,\dots,c^m(p_m-1)=p_m,c^m(p_m)=1$. Then we define 
inductively the entries $z_{p_m},\dots,z_1$ of the
vector $z\in (\TT^\vee)^{p_m}$ by
\[
z_{p_m}=\unit, \qquad  z_{p_{m}-1} \oplus A_{p_m -1 ,p_m} z_{p_m} \bal \zero,
\qquad 
\dots \qquad 
z_{1} \oplus A_{1 2 } z_{2} \bal \zero . 
\]
Since the permutation $\pi$ does not  satisfy~\eqref{e-sat-perm},
but satisfies $|\bigodot_{i\in [n]} A_{i\pi(i)} |=|\det A|$,
the entries $A_{p_m-1,p_m}$, \dots, $A_{12}$ are all invertible
in $\TT$. Hence, the former relations define the vector $z$ uniquely. Actually,
\begin{equation}\label{z_p_m}
z_{p_m}=\unit, \qquad  z_{p_{m}-1}=\ominus A_{p_m -1 ,p_m} z_{p_m} ,
\qquad 
\dots  , \qquad 
z_{1} =\ominus A_{1 2 } z_{2} 
\end{equation}
and we observe that $|z_i| = 1$ for all $i\in[p_m]$.
Moreover, from~\eqref{e-compat} and~\eqref{z_p_m} we deduce that
\begin{equation} \label{z_p_m_1}
z_{p_{m}} \oplus A_{p_m ,1} z_{1} \bal \zero.
\end{equation}
Let $G$ denote the $p_m \times p_m$ top-left submatrix of $A$. 
It follows from Formulas~\eqref{z_p_m} and~\eqref{z_p_m_1} and from
Property~\ref{p-temp-3} that
$(Gz)_i\bal \zero$ holds for all $i\in[p_m]$.

Let us now write $A$ in the block form
\[
A= \begin{pmatrix} G & * \\ V & F \end{pmatrix}
\]
where $V$ and $F$ are of sizes $q\times p_m$ and $q\times q$, 
respectively, with $q:=n-p_m$,
and look for a solution of the form $x= (z_1,\dots,z_{p_m}, y_1,\dots, y_q)^\top$.
Then we may choose for $y=(y_1,\ldots, y_q)^\top$ a \thin vector solution of $Vz\oplus Fy \bal \zero$ given by Theorem~\ref{th-jac}. Thus  $|y|= |F\adj Vz|= |F|^*|Vz|$. Hence, $|y_j|\leq \unit$ for
all $j\in [q]$. It follows from
Property~\ref{p-temp-3} that $A x \bal \zero$.
\end{proof}

\begin{corollary}\label{cor-dominant-unique}
Let $\TT$ be a semiring allowing strong balance elimination 
(Definition~\ref{def-tropextension}) and
\conv\ (Definition~\ref{monotone-def}), and satisfying Properties~\ref{p-temp-0}--\ref{p-temp-3}.
Let $A \in\M(\TT)$ such that $\det A$ is invertible and let $b\in \TT^n$. 
Then the system $Ax\bal b$ has a unique \thin
solution if and only if $A\adj b$ is \thinp.
\end{corollary}
\begin{proof}
The sufficient condition follows from the second
assertion of Theorem~\ref{th-cramer}.
For the necessary condition, let us assume that $A\adj b$ is not
\thinp, and prove that there exist at least two different solutions.
Since $A\adj b$ is not \thinp, there exists $i\in [n]$ such that
$(A\adj b)_i\in \TT^\circ\setminus\{\zero\}$.
Assume without loss of generality that $i=n$.

From Theorem~\ref{th-jac}, since $\det A\neq \zero$, 
 there exists a solution $x'$ of $Ax\bal b$ such that
$|x'|=|A\adj b|$, thus $|x'_n|=|(A\adj b)_n|\neq 0$ hence $x'_n\neq \zero$.
Now let us construct a solution $x''\neq \zero$ such that $x''_n=\zero$.
Let $B$ be the block matrix $B=(A_{|n)} \ominus b)$.
We have $\det B=(A\adj b)_n\bal \zero$, so by Theorem~\ref{newTGMa}, there 
exists $y\in (\TT^\vee)^n\setminus\{\zero\}$ such that $B y\balance \zero$.
Let us show that $y_n\neq \zero$. 
Indeed, if $y_n=\zero$ then $y$ yields a thin solution of $Ax \bal \zero$,
and since $\det A$ is invertible, this implies that $y=\zero$, a contradiction.
So $y_n\neq \zero$ and since $y_n\in \TT^\vee$, $y_n$ is invertible. So
multiplying $y$ by $y_n^{-1}$ we obtain a solution
of $B y\balance \zero$ such that $y_n=\unit$.
Then the vector $x''=(y_1,\ldots, y_{n-1}, \zero)^\top$ is a solution
of $A x\bal b$ such that $x''_n=\zero$.
This shows that the system  $A x\bal b$ has at least two non-zero solutions.
\end{proof}

\begin{remark}
Proposition~8.8 of~\cite{AGG08} gives a $6 \times 7$ matrix with entries
in $\rmax$, such that there is no signed no-zero row vector $x$ such that $xA\bal \zero$, but all maximal determinants taken from $A$ are balanced. This shows
that Theorem~\ref{newTGMa} cannot be extended to the rectangular case.
\end{remark}

\section{Systems of balances and intersections of signed hyperplanes}\label{sec-geom}
We now give a geometrical interpretation of the previous results.
We first consider, as in Section~\ref{sec-elim},
a semiring  $\SS$ with symmetry and a thin set
$\SS^{\vee}$. 
We call {\em hyperplane} of $\SS^{n}$ a set of the form
\begin{align}\label{e-def-genh}
H= \{x\in (\SS^\vee)^n\mid \bigoplus_{i\in [n]} \paramh_i x_i \balance \zero\}
\end{align}
where $\paramh\in (\SS^\vee)^n$ is a non-zero vector.

\begin{example}\label{ex-htrop}
When $\SS$ is the bi-valued tropical semiring $\To_2$,
$\SS^\vee$
can be identified to $\rmax$, then
$H\cap \Re^n$ coincides with
the tropical hyperplane~\eqref{e-def-h}.
\end{example}

\begin{example}\label{ex-hsign}
Assume now that $\SS$ is the symmetrized tropical semiring $\smax$, so that
$\paramh \in (\smax^\vee)^n$, with $\smax^\vee = \smax^\oplus \cup \smax^\ominus$. Identifying $\smax^\oplus$ with $\rmax$, and
setting $I:= \{i\in[n]\mid \paramh_i \in \smax^\oplus\}$ and $J:=[n]\setminus I$,
it is readily seen that $H \cap \rmax^n  = H\signh$ is the signed
tropical hyperplane defined in~\eqref{e-def-hsign}. 
\end{example}

The following result is a simple consequence of Theorem~\ref{th-cramer}.
We say that $n-1$ vectors $v^1,\dots,v^{n-1}$ of $(\SS^\vee)^n$ are
in {\em general position}
if every $(n-1)\times (n-1)$ minor of the $n\times (n-1)$
matrix $M$ with columns $v^1,\dots,v^{n-1}$ is thin and non-zero
($\in (\SS^\vee)^*$).
Similarly, we say that $n-1$ hyperplanes of $\SS^n$ are
in {\em general position} if the vectors of parameters
of these hyperplanes are in general position.
\begin{theorem}[Geometric form of Cramer theorem]\label{th-cramer-geom}
Let $\SS$ be a semiring with a thin set
$\SS^{\vee}$, allowing strong balance elimination
(Definition~\ref{def-tropextension}). Assume
that $(\SS^\vee)^*$ is the set of invertible elements of $\SS$. Then
\begin{itemize}
\item[Primal.]
Any $n-1$ vectors of $(\SS^{\vee})^{n}$ in general position 
are contained in a unique hyperplane.
\item[Dual.] Any $n-1$ hyperplanes of $(\SS^{\vee})^{n}$ in general position 
contain a non-zero vector which is unique
up to an invertible scalar multiple. 
\end{itemize}
\end{theorem}
\begin{proof}
We prove the primal statement
(the dual statement follows along the same lines).
Assume that the vectors $v^1,\dots,v^{n-1}$ are included in the hyperplane
$H$ of~\eqref{e-def-genh}.
Then the vector $\paramh$ of parameters of this hyperplane,
thought of as a row vector, satisfies $\paramh M\balance \zero$
where $M$ is as above.
Up to a transposition, and to the replacement of $n$ by $n+1$, this
system is of the type considered in Corollary~\ref{cor-cramer},
and the conclusion follows from the latter corollary.
\end{proof}
It follows from Examples~\ref{ex-htrop} and~\ref{ex-hsign} that 
Theorems~\ref{th-1} and~\ref{th-3} stated
in the introduction can be re-obtained
by specializing the primal form in Theorem~\ref{th-cramer-geom}
to $\SS=\To_2$ or $\SS=\smax$.

Similarly, Theorem~\ref{newTGMa} admits the following
geometric interpretation. The derivation is straightforward.
\begin{theorem}[Singular matrices]\label{th-hom-geom}
Let $\TT$ be a semiring allowing weak balance elimination 
(Definition~\ref{def-tropextension}) and
\conv\ (Definition~\ref{monotone-def}),
 and satisfying Properties~\ref{p-temp-0}--\ref{p-temp-3}.
Then 
\begin{itemize}
\item[Primal.]
A collection of $n$ vectors $v^1,\dots,v^n$ of $\TT^n$ is contained
in a hyperplane if and only if the determinant of the matrix
$(v^1,\dots,v^n)$ is balanced;
\item[Dual.]
A collection of $n$ hyperplanes of $\TT^n$
\[ H^j=\{x\in (\TT^\vee)^n\mid \bigoplus_{i\in[n]}a^{j}_i x_i \balance \zero\},
\qquad j\in[n] \enspace,
\]
contains a non-zero
vector if and only if the determinant of the matrix
$(a_i^j)_{i,j\in [n]}$ is balanced.
\qed
\end{itemize}
\end{theorem}

When $\SS$ or $\TT$ are equal to $\smax$, the dual
statements in Theorems~\ref{th-cramer-geom} and~\ref{th-hom-geom} 
turn out to have a geometric interpretation which 
can be stated elementarily, without introducing
the symmetrized tropical semiring.

This interpretation relies on the notion of 
{\em sign-transformation} of a signed 
hyperplane. Such a transformation is specified by a sign
pattern $\epsilon\in \{\pm 1\}^n$, it corresponds, in loose terms,
to putting variables $x_i$ such that $\epsilon_i=-1$ on the other side of the equality. Formally, the sign-transformation of pattern $\epsilon$
transforms the signed-hyperplane $\signed{H}$ to
\[
\signed{H}(\epsilon) = \{x\in \rmax^{n} \mid \max_{\substack{i\in I, \epsilon_i = 1 \text{ or}\\ i\in J,\;\epsilon_i =-1}}
(a_i + x_i) =
\max_{\substack{j\in J, \epsilon_j = 1 \text{ or}\\ j\in I,\;\epsilon_j =-1}}
(a_j + x_j) 
\} \enspace .
\]
Figure~\ref{fig-transform} gives an illustration of this notion.
\begin{figure}[htbp]
\begin{picture}(0,0)%
\includegraphics{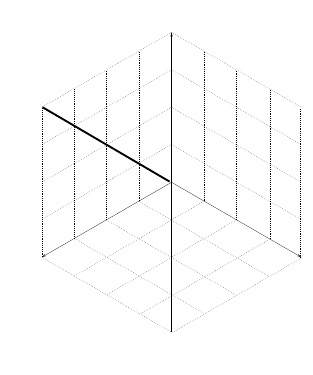}%
\end{picture}%
\setlength{\unitlength}{592sp}%
\begingroup\makeatletter\ifx\SetFigFont\undefined%
\gdef\SetFigFont#1#2#3#4#5{%
  \reset@font\fontsize{#1}{#2pt}%
  \fontfamily{#3}\fontseries{#4}\fontshape{#5}%
  \selectfont}%
\fi\endgroup%
\begin{picture}(9930,11946)(811,-11023)
\put(5926,164){\makebox(0,0)[lb]{\smash{{\SetFigFont{10}{12.0}{\rmdefault}{\mddefault}{\updefault}{\color[rgb]{0,0,0}$x_3$}%
}}}}
\put(10726,-7636){\makebox(0,0)[lb]{\smash{{\SetFigFont{10}{12.0}{\rmdefault}{\mddefault}{\updefault}{\color[rgb]{0,0,0}$x_2$}%
}}}}
\put(1651,-10711){\makebox(0,0)[lb]{\smash{{\SetFigFont{10}{12.0}{\rmdefault}{\mddefault}{\updefault}{\color[rgb]{0,0,0}$x_1=\max(x_2,x_3)$}%
}}}}
\put(826,-7411){\makebox(0,0)[lb]{\smash{{\SetFigFont{10}{12.0}{\rmdefault}{\mddefault}{\updefault}{\color[rgb]{0,0,0}$x_1$}%
}}}}
\end{picture}%
\ \ \ \ \ \ \ \  
\begin{picture}(0,0)%
\includegraphics{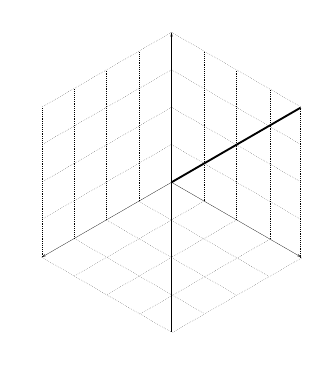}%
\end{picture}%
\setlength{\unitlength}{592sp}%
\begingroup\makeatletter\ifx\SetFigFont\undefined%
\gdef\SetFigFont#1#2#3#4#5{%
  \reset@font\fontsize{#1}{#2pt}%
  \fontfamily{#3}\fontseries{#4}\fontshape{#5}%
  \selectfont}%
\fi\endgroup%
\begin{picture}(9930,11946)(811,-11023)
\put(5926,164){\makebox(0,0)[lb]{\smash{{\SetFigFont{10}{12.0}{\rmdefault}{\mddefault}{\updefault}{\color[rgb]{0,0,0}$x_3$}%
}}}}
\put(10726,-7636){\makebox(0,0)[lb]{\smash{{\SetFigFont{10}{12.0}{\rmdefault}{\mddefault}{\updefault}{\color[rgb]{0,0,0}$x_2$}%
}}}}
\put(1651,-10711){\makebox(0,0)[lb]{\smash{{\SetFigFont{10}{12.0}{\rmdefault}{\mddefault}{\updefault}{\color[rgb]{0,0,0}$x_2=\max(x_1,x_3)$}%
}}}}
\put(826,-7411){\makebox(0,0)[lb]{\smash{{\SetFigFont{10}{12.0}{\rmdefault}{\mddefault}{\updefault}{\color[rgb]{0,0,0}$x_1$}%
}}}}
\end{picture}%
\ \ \ \ \ \ \ \  
\begin{picture}(0,0)%
\includegraphics{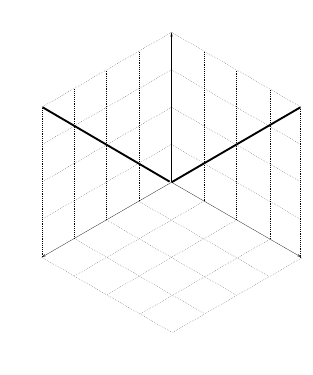}%
\end{picture}%
\setlength{\unitlength}{592sp}%
\begingroup\makeatletter\ifx\SetFigFont\undefined%
\gdef\SetFigFont#1#2#3#4#5{%
  \reset@font\fontsize{#1}{#2pt}%
  \fontfamily{#3}\fontseries{#4}\fontshape{#5}%
  \selectfont}%
\fi\endgroup%
\begin{picture}(9930,11946)(811,-11023)
\put(5926,164){\makebox(0,0)[lb]{\smash{{\SetFigFont{10}{12.0}{\rmdefault}{\mddefault}{\updefault}{\color[rgb]{0,0,0}$x_3$}%
}}}}
\put(10726,-7636){\makebox(0,0)[lb]{\smash{{\SetFigFont{10}{12.0}{\rmdefault}{\mddefault}{\updefault}{\color[rgb]{0,0,0}$x_2$}%
}}}}
\put(1651,-10711){\makebox(0,0)[lb]{\smash{{\SetFigFont{10}{12.0}{\rmdefault}{\mddefault}{\updefault}{\color[rgb]{0,0,0}$x_3=\max(x_1,x_2)$}%
}}}}
\put(826,-7411){\makebox(0,0)[lb]{\smash{{\SetFigFont{10}{12.0}{\rmdefault}{\mddefault}{\updefault}{\color[rgb]{0,0,0}$x_1$}%
}}}}
\end{picture}%
\caption{Sign-transformation of a signed hyperplane}\label{fig-transform}
\end{figure}
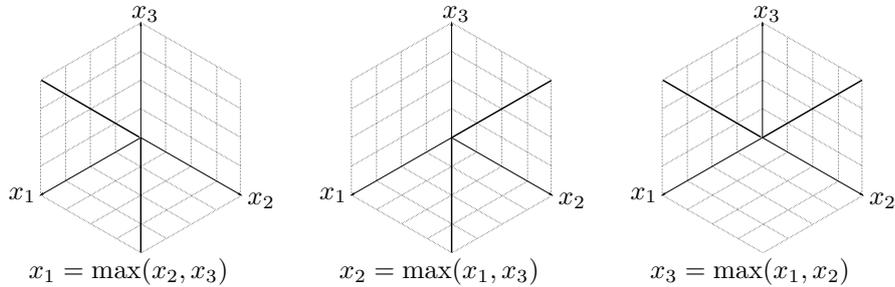
The following theorem follows readily from the dual statement
in Theorem~\ref{th-cramer-geom}, it can also be derived from
\cite[Th.~6.1]{Plus}.
\begin{theorem}\label{thnew-signt}
Given $n-1$ signed tropical hyperplanes $H_1\signh,\ldots,H_{n-1}\signh$,
in general position, there is a unique sign pattern $\epsilon$
such that the transformed hyperplanes $H_1\signh(\epsilon),\ldots,H_{n-1}\signh(\epsilon)$, meet at a non-zero
vector. Moreover, such a vector is unique up to an additive constant.
\end{theorem}
\begin{proof}
Let $H$ be the hyperplane defined by~\eqref{e-def-genh}, 
and let $H\signh= H\cap \rmax^n$ as in Example~\ref{ex-hsign}.
It is easily seen that
a vector $x\in (\smax^\vee)^n$ belongs to $H$ if and only if the vector
$|x|$ belongs to the transformed signed tropical hyperplane 
$H\signh(\epsilon)$ where $\epsilon$ is the sign vector of $x$.
Then the theorem follows from
the dual form of Theorem~\ref{th-cramer-geom}.
\end{proof}
\begin{figure}[htbp]
\begin{picture}(0,0)%
\includegraphics{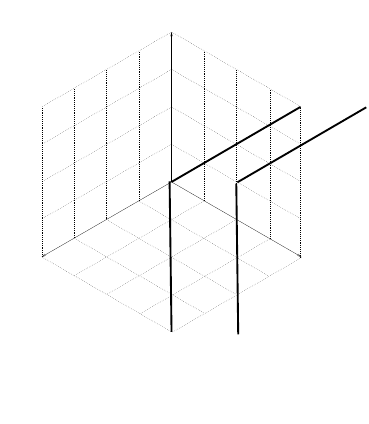}%
\end{picture}%
\setlength{\unitlength}{592sp}%
\begingroup\makeatletter\ifx\SetFigFont\undefined%
\gdef\SetFigFont#1#2#3#4#5{%
  \reset@font\fontsize{#1}{#2pt}%
  \fontfamily{#3}\fontseries{#4}\fontshape{#5}%
  \selectfont}%
\fi\endgroup%
\begin{picture}(11776,13521)(811,-12598)
\put(5926,164){\makebox(0,0)[lb]{\smash{{\SetFigFont{10}{12.0}{\rmdefault}{\mddefault}{\updefault}{\color[rgb]{0,0,0}$x_3$}%
}}}}
\put(10726,-7636){\makebox(0,0)[lb]{\smash{{\SetFigFont{10}{12.0}{\rmdefault}{\mddefault}{\updefault}{\color[rgb]{0,0,0}$x_2$}%
}}}}
\put(1651,-10711){\makebox(0,0)[lb]{\smash{{\SetFigFont{10}{12.0}{\rmdefault}{\mddefault}{\updefault}{\color[rgb]{0,0,0}$x_2=\max(x_1,x_3)$}%
}}}}
\put(4426,-12286){\makebox(0,0)[lb]{\smash{{\SetFigFont{10}{12.0}{\rmdefault}{\mddefault}{\updefault}{\color[rgb]{0,0,0}$x_2-2=\max(x_1,x_3-1)$}%
}}}}
\put(826,-7411){\makebox(0,0)[lb]{\smash{{\SetFigFont{10}{12.0}{\rmdefault}{\mddefault}{\updefault}{\color[rgb]{0,0,0}$x_1$}%
}}}}
\end{picture}%
\ \ \ \ \ \ \ \ \ \ \ \  
\begin{picture}(0,0)%
\includegraphics{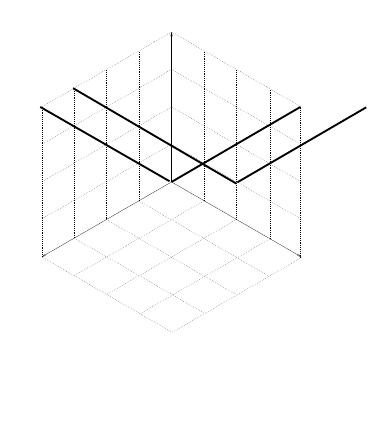}%
\end{picture}%
\setlength{\unitlength}{592sp}%
\begingroup\makeatletter\ifx\SetFigFont\undefined%
\gdef\SetFigFont#1#2#3#4#5{%
  \reset@font\fontsize{#1}{#2pt}%
  \fontfamily{#3}\fontseries{#4}\fontshape{#5}%
  \selectfont}%
\fi\endgroup%
\begin{picture}(11776,13521)(811,-12598)
\put(5926,164){\makebox(0,0)[lb]{\smash{{\SetFigFont{10}{12.0}{\rmdefault}{\mddefault}{\updefault}{\color[rgb]{0,0,0}$x_3$}%
}}}}
\put(10726,-7636){\makebox(0,0)[lb]{\smash{{\SetFigFont{10}{12.0}{\rmdefault}{\mddefault}{\updefault}{\color[rgb]{0,0,0}$x_2$}%
}}}}
\put(1651,-10711){\makebox(0,0)[lb]{\smash{{\SetFigFont{10}{12.0}{\rmdefault}{\mddefault}{\updefault}{\color[rgb]{0,0,0}$x_3=\max(x_1,x_2)$}%
}}}}
\put(4426,-12286){\makebox(0,0)[lb]{\smash{{\SetFigFont{10}{12.0}{\rmdefault}{\mddefault}{\updefault}{\color[rgb]{0,0,0}$x_3-1=\max(x_1,x_2-2)$}%
}}}}
\put(826,-7411){\makebox(0,0)[lb]{\smash{{\SetFigFont{10}{12.0}{\rmdefault}{\mddefault}{\updefault}{\color[rgb]{0,0,0}$x_1$}%
}}}}
\end{picture}%

\ \ \ \ 
\caption{Illustration of the dual form of the Cramer theorem in the symmetrized tropical semiring (Theorem~\ref{thnew-signt})}
\end{figure}
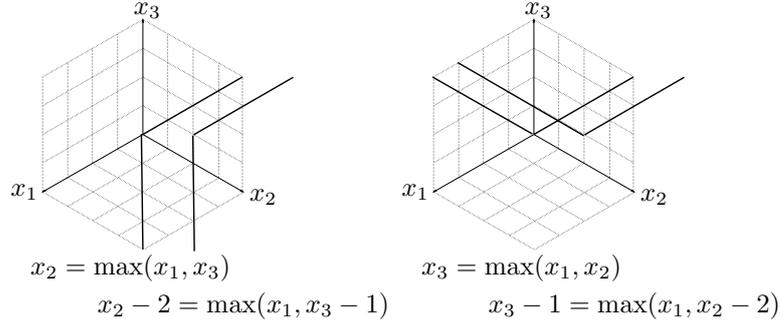
The interpretation of the dual form of Theorem~\ref{th-hom-geom},
which could also be derived from
the result of~\cite[Th.~6.5]{Plus} proved in~\cite[Chap.~3, Th.~9.0.1]{gaubert92a}, can be stated as follows. 
\begin{theorem}\label{thnew-5}
Given $n$ signed hyperplanes $H_1\signh,\dots,H_n\signh$, there exists
a sign pattern $\epsilon$ such that the transformed hyperplanes
$H_1\signh(\epsilon),\dots,H_n\signh(\epsilon)$
contain a common non-zero vector 
if and only if the matrix having as rows the vectors
of parameters of $H_1\signh,\dots,H_n\signh$ has a balanced determinant.
\end{theorem}
\begin{proof}
Argue as in the proof of Theorem~\ref{thnew-signt}.
\end{proof}

\section{Computing all Cramer Permanents: tropical Jacobi versus transportation approach} \label{S_AC}

\subsection{Computing all Cramer permanents by the tropical Jacobi algorithm}\label{sec-perall}
The present approach via the tropical Jacobi algorithms
leads to an algorithm to compute all the Cramer permanents.

\begin{corollary}[Computing all the Cramer permanents]\label{compute-cramer}
Let $A\in \MM_n(\rmax)$ and $b\in\rmax^n$. Assume
that $\per A\neq \zero$. Then the vector
$A\adj b$, the entries of which are the $n$ Cramer permanents of the system with matrix $A$ and right-hand side $b$,
can be computed by solving a single optimal assignment problem,
followed by a multiple origins-single destination shortest path problem.
\end{corollary}
\begin{proof}
Let $\sigma$ denote an optimal permutation for the matrix $A$.
After permuting the rows of $A$ and $b$, we may assume that this
permutation is the identity. Dividing every row of $A$ and $b$ by $A_{ii}$, we may assume that $A_{ii}=\unit$ for all $i\in [n]$.
Then using Yoeli's theorem we get $A\adj b=A^*b$. 
Computing the latter vector is equivalent to solving 
a shortest path problem from all origins to a single destination.
\end{proof}
\begin{remark}
The Hungarian algorithm of Kuhn runs in time $O(n(m+n\log n))$,
where $m$ is the number of finite entries of the matrix $A$. 
The subsequent shortest path problem can be solved 
for instance by the Ford-Bellman algorithm, in time
$O(mn)$. Hence, we arrive at the strongly polynomial bound
$O(n(m+n\log n))\leq O(n^3)$, for the time needed to compute all
Cramer permanents. Alternative (non-strongly polynomial) optimal assignment
algorithms may be used~\cite{Assignmentprobs},  leading to incomparable bounds.
\end{remark}

\subsection{The transportation approach of~\cite{RGST}}\label{sec-transport}
Richter-Gebert, Sturmfels, and Theobald developed in~\cite{RGST} a different
approach, based on earlier results of Sturmfels and Zelevinsky~\cite{SZ}. 
It provides an alternative method in which all the Cramer determinants
are obtained by solving a single transportation problem.
A beauty of this approach is that it directly
works with homogeneous coordinates, preserving
the symmetry which is broken by the Jacobi approach,
of an ``affine'' nature.
So, we next revisit the method of~\cite{RGST}, in order
to compare the results obtained in this way with the present ones. In passing, we shall
derive some refinements of results in~\cite{RGST} concerning
the case in which the data are not in general position. 
First, we observe that one of the results of~\cite{RGST} can be recovered
as a corollary of the present elimination approach.

\begin{definition}
Let $A\in \MM_{n-1,n}(\rmax)$. The {\em tropical Cramer permanent\/} $\per A_{|k)}$ is the permanent 
of the matrix obtained from $A$ by deleting the $k$'th column.
\end{definition}

\begin{theorem}[Compare with~\cite{RGST}, Corollary~5.4]
Assume that $A\in \MM_{n-1,n}(\rmax)$ is such that at least one of the tropical
permanents $\per A_{|k)}$ is finite. Then the vector $x=(x_k)$ with
$x_k=\per A_{|k)}$ is such that in the expression
\[
Ax=\bigoplus_{k\in[n]} A_{\cdot k} x_k 
\]
the maximum is attained at least twice in every row. Moreover, if all
the tropical Cramer permanents $\per A_{|k)}$ are non-singular, then the vector
$x$ having the latter property is unique up to an additive constant.
\end{theorem}
\begin{proof}
This is a special case of Theorem~\ref{th-new-homogeneous} in which the matrix $A$ is \thinp.
\end{proof}
Corollary~5.4 of~\cite{RGST} gives an explicit construction of the solution $x_k$ when the tropical Cramer permanents are non-singular. Then it proceeds
by showing 
that the solution of 
a certain transportation problem is unique.
Here the uniqueness of the solution
of the tropical equations is obtained by the elimination argument
used in the proof of Theorem~\ref{th-cramer}. 
We now
present the method of~\cite{RGST} in some detail.

Let $T_{n-1,n}$ denote the transportation polytope 
consisting of all  nonnegative $(n-1)\times n$ matrices $y=(y_{ij})$ such that 
\begin{align*}
\sum_{j\in[n]} y_{ij} = n, \quad i\in[n-1],\qquad 
\sum_{i\in[n-1]} y_{ij} = n-1,  \quad  j\in[n] \enspace .
\end{align*}
For $k\in [n]$ define $\Pi^k_{n-1,n}$ to be the Birkhoff
polytope obtained as the convex hull of the $(n-1)\times n$ matrices 
with $0,1$-entries, representing matchings between $[n-1]$ and
$[n]\setminus \{k\}$. Define the Minkowski sum:
\[
\Pi_{n-1,n}:= \sum_{k\in [n]} \Pi^k_{n-1,n} \enspace .
\]
Here (partial) matchings refer to the complete bipartite graph $K_{n-1,n}$
with $n-1$ nodes in one class (corresponding to the rows of $y$)
and $n$ nodes in the other class (corresponding to the columns of $y$).
Every vertex of $\Pi_{n-1,n}$ can be written
as 
\(
y=y^1+\cdots+y^n ,
\) 
where for all $k\in [n]$, $y^k\in \MM_{n-1,n}(\{0,1\})$ 
represents a matching between $[n-1]$ and $[n]\setminus \{k\}$,
meaning that the set of edges $\set{(i,j)}{y^k_{ij}=1}$ constitutes
a matching between  $[n-1]$ and $[n]\setminus \{k\}$.

More generally,  Sturmfels and Zelevinsky~\cite{SZ}
considered the Newton polytope
$\Pi_{m,n}$ of the product of maximal minors of 
any $m\times n$ matrix such that
$m\leq n$. The entries of this matrix are thought of as pairwise distinct indeterminates. The former polytope is obtained by taking $m=n-1$. They showed
that $T_{n-1,n}= \Pi_{n-1,n}$, see~\cite[Th.~2.8]{SZ},
but that $\Pi_{m,n}$ is no longer a transportation polytope when $m<n-1$. 
They also showed that the vertices of $\Pi_{n-1,n}$ are in bijective correspondence
with combinatorial objects called \NEW{linkage trees}. A linkage tree is 
a tree with set of nodes $[n]$, the edges of which are bijectively labeled
by the integers $1,\dots,n-1$. Given a linkage tree, we associate
to every $k\in [n]$ the matching between $[n-1]$ and $[n]\setminus\{k\}$,
 such that
$j\in[n]\setminus\{k\}$ is matched to the unique $i\in[n-1]$ labeling the edge
adjacent to $j$ in the path connecting $j$ to $k$ in this tree. 
Let $y^k$ denote the matrix
representing this matching. Then the vertex
of $\Pi_{n-1,n}$ that corresponds to this linkage tree is
$y^1+\dots +y^n$. 

We now associate to every $(n-1)\times n$
matrix $y$ a subgraph $G(y)$ of $K_{n-1,n}$, consisting of the edges $(i,j)$ such that $y_{ij}\neq 0$. If $y$ is a vertex of $\Pi_{n-1,n}$,
then the previous characterization in terms
of linkage trees implies
that $G(y)$ is a spanning tree of $K_{n-1,n}$.
 
To simplify the exposition, we shall assume first that
$c_{ij}\in {\Re}$, $i\in[n-1]$, $j\in[n]$.
Following the idea
of Richter-Gebert, Sturmfels, and Theobald~\cite{RGST},
we consider the {\em transportation problem} $\cP$
\begin{align*}
\max \sum_{i\in[n-1],\; j\in[n]} c_{ij}y_{ij} \;; \qquad 
{y}=(y_{ij})\in T_{n-1,n}
\tag{$\cP$}
\end{align*}
and its dual
\begin{align*}
\min \Bigl( n(\sum_{i\in[n-1]} u_i)+ (n-1)(\sum_{j\in[n]} v_j)\Bigr) \;; \qquad {u}=(u_i)\in \Re^{n-1}, 
{v}=(v_j)\in \Re^{n}, \\
c_{ij}\le u_i+v_j,\quad i\in[n-1],\ j\in[n] 
\qquad 
\tag{$\cD$}
\enspace .
\end{align*}
The values of these problems will be denoted by 
$\val\cP$ and  $\val \cD$, respectively.

Recall the {\em complementary slackness condition\/}: a 
primal feasible solution $y$ and a dual feasible solution
$({ u},{ v})$ are both optimal if and only if $y_{ij}(u_i+v_j-c_{ij})=0$ holds for all $i\in[n-1]$ and $j\in[n]$.

The case in which $c_{ij}=-\infty$ holds for some
$i,j$ can be dealt with by adopting the convention
that $(-\infty)\times 0=0$ in the expression
of the primal objective function. Equivalently, 
it can be considered as an ordinary linear programming problem by
adding the constraints $y_{ij}=0$ for all $(i,j)$ such that
$c_{ij}=-\infty$, and ignoring all $(i,j)$ such that $c_{ij}=-\infty$
in the formulation of the dual problem and in the complementary slackness conditions. 

We now give a slight refinement
of a result of~\cite{RGST},
allowing one to interpret the optimal dual variable in terms
of tropical Cramer determinants.
Corollary~5.4 of~\cite{RGST} only deals with the case where the matrix $C$ is generic, whereas Theorem~\ref{T_RGST} shows that this assumption is not needed for the optimal solution of the dual problem $(\cD)$ to be unique, up to 
a transformation by an additive constant.  Genericity is only needed for the uniqueness of the optimal solution of the primal problem, which we do not use here.

\begin{theorem}[Compare with~\cite{RGST}, Corollary~5.4] \label{T_RGST}
The primal problem $(\cP)$ is feasible if and only if all the tropical Cramer permanents of the matrix $C$ are finite. When this is the case, the 
optimal solution $(u,v)\in \Re^{n-1} \times \Re^n$
of the dual transportation problem $(\cD)$ is unique up to 
a modification of the vector $(u,-v)$ by an additive constant,
and 
\begin{align}
\label{e-per}
\per C_{|k)} = \sum_{i\in[n-1]} u_i + \sum_{j\in [n]\setminus \{k\}} v_j 
\enspace,
\qquad \text{ for all }k\in[n] \enspace .
\end{align}
\end{theorem}
\begin{proof}
We first consider the case in which all $c_{ij}$ are finite. 
As the initial step, we show the announced uniqueness
result for the optimal solution $(u,v)$ of the dual problem.

Let $y$ denote an optimal solution of the primal problem,
which we choose to be a vertex of $T_{n-1,n}=\Pi_{n-1,n}$,
so that $G(y)$ is a spanning tree of $K_{n-1,n}$. Using 
the complementary slackness condition, we have 
\begin{align}
c_{ij}=u_i+v_j\qquad \text{for all } (i,j)\in G(y)
\enspace .\label{e-slack-prop}
\end{align}
Using these relations and the fact that a spanning
tree is connected, we see that all the values
of the variables $u_i$ and $v_j$, $i\in [n-1]$ and $j\in [n]$, 
are uniquely determined by the value of a single variable
$u_\ell, \ell \in [n-1]$ (or dually, by the value of a single variable $v_k, k\in [n]$). Moreover, an increase
of $u_\ell$ by a constant increases every other variable $u_i$ by the same
constant and decreases every variable $v_j$ by the same constant. This
establishes the announced uniqueness result.

Since $c_{ij}\leq u_i + v_j$ holds
for all $i\in [n-1]$ and $j\in [n]$, we deduce that for
all $k\in [n]$ and for all bijections $\sigma$ from $[n-1]$ to $[n]\setminus \{k\}$ 
\begin{align}
\sum_{i\in[n-1]} c_{i\sigma(i)}\leq \sum_{i\in[n-1]} u_i + \sum_{j\in [n]\setminus \{k\}} v_j \enspace .
\label{e-tight}
\end{align}
Considering the maximum over all such bijections $\sigma$,
we deduce that $\per C_{|k)}$ is bounded from above by the right-hand side
of the latter inequality.

Now, using again the fact $T_{n-1,n}=\Pi_{n-1,n}$, we can write
$y=y^1+\dots+y^n$ where every $y^k$ represents a matching
between $[n-1]$ and $[n]\setminus\{k\}$, to which
we associate a bijection $\sigma^k $
from $[n-1]$ to $[n]\setminus\{k\}$.
By the complementary slackness condition we have $c_{i\sigma^k(i)}= u_i+v_{\sigma^k(i)}$ for all $i\in [n-1]$. It follows that $\sigma=\sigma^k$ achieves
the equality in~\eqref{e-tight}, showing that $\per C_{|k)}$ is given
by the right-hand side of~\eqref{e-tight}.

We now deal with the case in which some coefficients $c_{ij}$ can take
the $-\infty$ value.
The reward function $y\mapsto \sum_{i\in[n-1],\,j\in[n]}c_{ij}y_{ij}$
over $T_{n-1,n}$, with the convention $(-\infty)\times 0=0$, now
takes its value in $\Re\cup\{-\infty\}$. It is 
upper semicontinuous and concave, so it attains
its maximum at a vertex of $T_{n-1,n}=\Pi_{n-1,n}$. 
It follows that
the value of the primal problem is finite if and only if
every tropical Cramer permanent is finite. 
Then the remaining arguments of the proof above are easily checked
to carry over, by working with the modified linear programming
formulation, in which the constraints $y_{ij}=0$ for all $(i,j)$ such that
$c_{ij}=-\infty$, are added. 
\end{proof}

\section{Computing determinants}\label{sec-compute-det}
To compute determinants over $\Ti$ or $\smax$ we need
to recall how tropical singularity can be checked.
Let $A\in \MM_n(\rmax)$. We observed in Proposition~\ref{prop-hungarian}
that, as soon as $\per A\neq\zero$,
the Hungarian algorithm 
gives scalars $u_i,v_j\in \Re$, for $i,j\in [n]$, 
such that
\[
a_{ij}\leq u_i  v_j \enspace ,\qquad \text{and}\qquad 
\per A=\prod_i u_i \prod_j v_j \enspace .
\]
The optimal permutations $\sigma$
are characterized by the condition that $a_{i\sigma(i)}=u_iv_{\sigma(i)}$,
for all $i\in [n]$.
After multiplying $A$ by a permutation matrix,
we may always assume that the identity is a solution of the
optimal assignment problem.
Then we define the digraph $G$ with nodes $1,\dots,n$, and an arc from
$i$ to $j$ whenever $a_{ij}= u_i v_j$.
Butkovi\v{c} proved two results which can be formulated
equivalently as follows.
\begin{theorem}[See~{\cite{butkovip94} and~\cite{butkovic95}}]\label{thbut}
Let $A\in \MM_n(\rmax)$, and assume that $\per A\neq\zero$.
Then checking whether the 
optimal assignment problem has at least two optimal solutions reduces
to finding a cycle in the digraph $G$, whereas checking whether
it has at least two optimal solutions of a different parity reduces
to finding an even cycle in $G$.
\end{theorem}

By exploiting the proof technique of Theorem~\ref{thbut} we
obtain the following result.

\begin{corollary}
Assume, $A\in \MM_n(\Ti)$ or  $A\in \MM_n(\smax)$. Then the determinant $\det A$ can be computed in polynomial time.
\end{corollary}
\begin{proof}
Assume first that $A\in \MM_n(\Ti)$.
We set $B=|A|$, meaning that  $B_{ij}=|A_{ij}|$ for all $i,j$,
and  compute $\per B$ together with
a permutation $\sigma$ solving the optimal assignment problem
for the matrix $B$. If $\per B=\zero$,
then $\per A=\zero$, so we next assume that $\per B\neq \zero$.
After multiplying
$A$ by the inverse of the matrix of the optimal permutation $\sigma$, we may 
assume that this permutation is the identity.
If one of the diagonal coefficients
belongs to $\Ti^\circ$, then we conclude that $\det A=(\per B)^{\circ}$. 
Otherwise, we define the digraph $G$ as above, starting from the matrix
$B$ instead of $A$. If $G$ has a cycle, then this cycle can be
completed by loops (cycles of length one) to obtain an optimal permutation
for $B$ distinct from the identity and so $\det A=(\per B)^\circ$.
Otherwise, $\det A=\per B$.

To compute $\det A$ when $A\in \MM_n(\smax)$, we consider
firstly the case $A\in \MM_n(\rmax)$, where $\rmax$ is thought
of as a subsemiring of $\smax$. We define again $B=|A|$ and assume that
the identity is an optimal permutation for the matrix $B$
and we define $G$ as above. Then $\det A=\per B$
if $G$ has no even cycle. Otherwise $\det A=(\per B)^\circ$ since
any even cycle can be completed by loops to get an optimal
permutation for $B$ of odd parity.

We assume finally that $A\in \MM_n(\smax)$ and make the same assumptions and constructions. If
one diagonal entry of $A$ is in $\smax^\circ$, we conclude that $\det A=(\per B)^\circ$. Otherwise, for every $(i,j)$ in $G$ such that 
$a_{ij}\in \smax^{\circ}$, we check whether $(i,j)$ belongs to a cycle of $G$ (not necessarily an even one). If this is the case, we conclude that $\det A =(\per B)^{\circ}$. Otherwise, the
elements $a_{ij}$ such that $a_{ij}\in \smax^{\circ}$ do not contribute
to the optimal permutations of the matrix $B$, and we may replace
them by $\zero$ without changing the value of $\det A$. Then
we may write $A=A^+\ominus A^-$ where $A^+,A^-\in \MM_n(\rmax)$.
Recall that for all matrices $C,D$ with entries in a commutative
ring the following block formula for the determinant holds
\[ \left| \begin{array}{cc}C & D\\ I & I
\end{array}\right|
=\det{(C-D)} \enspace ,
\]
where $I$ is the identity matrix.
Moreover, when viewing the entries of $C$ and $D$ as independent indeterminates and expanding the expressions at left and at right of the latter
equality, we see that the same monomials appear (each with multiplicity one)
on each side of the equality. Therefore, the equality is valid over an
arbitrary symmetrized semiring. 
In particular,
\[
\det A =\left| \begin{array}{cc}A^+ & A^-\\ I & I
\end{array}\right| \enspace,
\]
which reduces the computation of $\det A$ to the computation
of the determinant of a matrix with entries in $\rmax$, which
is already solved thanks to Theorem~\ref{thbut}.
\end{proof}

\paragraph{Acknowledgment}
The authors thank the referee for the very careful reading and for the suggestions.
\bibliographystyle{alpha}
\def\cprime{$'$} \def\cprime{$'$}

\end{document}